\documentclass[11pt,reqno]{amsart}
\usepackage{amsfonts}
\usepackage{mathrsfs}
\usepackage{amsfonts}
\usepackage{amsmath}
\usepackage{stmaryrd}
\usepackage{amssymb}
\usepackage{amsthm}
\usepackage{mathrsfs}
\usepackage{amsfonts}
\usepackage{amscd}
\usepackage{indentfirst}
\usepackage{enumerate}
\usepackage{amsmath,amsfonts,amssymb,amsthm}
\usepackage{amsmath,amssymb,amsthm,amscd}
\usepackage{graphicx,mathrsfs}
\usepackage{a4wide}
\usepackage{amsmath}
\allowdisplaybreaks[4]
\usepackage[colorlinks,linkcolor=blue,anchorcolor=blue,linktocpage=true,urlcolor=blue,citecolor=blue]{hyperref}

\usepackage{color}

\usepackage{empheq}
\allowdisplaybreaks
\usepackage[titletoc]{appendix}

 \numberwithin{equation}{section}
\newtheorem{theorem}{Theorem}[section]
\newtheorem{proposition}[theorem]{Proposition}

\newtheorem{lemma}[theorem]{Lemma}

\newtheorem{remark}[theorem]{Remark}

\newcommand{\R}{\mathbb{R}}

\begin{document}

\title[New solutions to the H\'{e}non problem]
{New sign-changing solutions to the H\'{e}non problem in dimension $2$}
\author[Gladiali]{Francesca Gladiali}
\address{Dipartimento di Scienze CFMN,  Universit\`a degli Studi di Sassari, Via Vienna 2 - 07100 Sassari, Italy, e-mail: {\sf fgladiali@uniss.it}.}

\author[Yang]{Pingping Yang}
\address{{School of Mathematics and Statistics, Key Laboratory of Nonlinear Analysis
$\&$ Applications (Ministry of Education),
 Hubei Key Laboratory of Mathematical Sciences, Central China
Normal University, Wuhan, 430079, P. R. China \& Dipartimento SCFMN, Universit$\grave{A}$ degli Studi di Sassari, Via Vienna 2, 07100 Sassari, Italy, e-mail: \sf ypp15623175603@mails.ccnu.edu.cn}.}

\begin{abstract}
In this paper, we construct
new sign-changing solutions for the H\'{e}non problem
\begin{equation*}
\begin{cases}
-\Delta u= |x|^{\alpha} |u|^{p-1}u,\,\,\,\text{in}\,\,\,
\Omega,
\\[1mm]
u=0,\,\,\,\,\,\,\text{on}\,\,\,
\partial\Omega,
\end{cases}
\end{equation*}
where $\Omega\subseteq \mathbb{R}^2$ is a bounded smooth domain containing 0, $\alpha\in\mathbb{R}$ is a positive parameter and $p$ approaches $+\infty$.
\end{abstract}
\maketitle

\section{Introduction}\label{s1}
In this paper, we consider the H\'{e}non problem
\begin{equation}\label{1.0}
\begin{cases}
-\Delta u= |x|^{\alpha} |u|^{p-1}u,\,\,\,\text{in}\,\,\,
\Omega,
\\[1mm]
u=0,\,\,\,\,\,\,\text{on}\,\,\,
\partial\Omega,
\end{cases}
\end{equation}
where $\Omega\subseteq \mathbb{R}^2$ is a bounded smooth domain containing 0, $\alpha\in\mathbb{R}$ is a positive parameter  and $p>1$. Problem \eqref{1.0} was proposed by M. H\'{e}non in \cite{M} to study rotating stellar structures and has drawn considerable interest from the academic community.  For a fixed $\alpha\geq0$ equation \eqref{1.0} admits a ground state solution for any value of the exponent $p>1$,  and the ground state solution is positive. Due to the weight term $|x|^\alpha$,  
the symmetry result in \cite{GNN}, does not hold and the H\'{e}non equation exhibits symmetry breaking when $\Omega$ is a ball. 
Indeed,  Chen, Ni and Zhou in \cite{CNZ} noted, by numerical computations,  that, for some values of the parameters $\alpha$ and $p$, the ground state solution to problem \eqref{1.0}, in a ball is nonradial. Smets, Su and Willem in \cite{SWS} confirmed it partially, when $\Omega$ is the unit ball, by showing that for any $p$ fixed there exists $\alpha^{\ast}=\alpha^{\ast}(p)>0$ such that for $\alpha>\alpha^{\ast}$,  the ground states solutions to \eqref{1.0} are nonradial. Moreover, they  proved that if $p$ is near $1$,  the ground state solution must be radial.  Byeon  and  Wang in \cite{BW} studied the asymptotic properties of the ground state solutions, in  both cases,  radial and nonradial, obtaining again the symmetry breaking when $\alpha$ is large.  In the case of higher dimensions, namely for $N\geq 3$, they showed a boundary concentration phenomenon for the ground state solutions, see \cite{BW2} and also \cite{CPY}. This boundary concentration is prevented in dimension $2$, as observed by Esposito, Pistoia and  Wei in \cite{EPW}.  
In this last paper they proved the existence of positive solutions concentrating at $k$ symmetric points as $p\rightarrow\infty$, 
$\alpha>0$, for any $1\le k<\alpha+1$, when $\Omega\subset \R^2$ contains the origin.  Let us observe that this result yields a multiplicity of positive solutions to problem \eqref{1.0} in convex domains, for $p$ large.
One of the major difference with the higher dimensional case is that solutions to \eqref{1.0} are uniformly bounded as $p\to +\infty$.  This holds for any type of solution.
Indeed, besides ground state solutions, there is a rich variety of other solutions,  including both positive and sign-changing ones and the understanding of their shape remains incomplete, even in dimension two.
As for the positive solutions,  \cite{B} established a uniform bound on their mass for sufficiently large value of the exponent $p$. Building on this result, several recent works have investigated the possible location of concentration points and their number, see \cite{zhao},  \cite{du} and \cite{chen-li}. In this last paper it has been shown that positive solutions stay bounded as $p\to \infty$,  the concentration can happen only in a finite number of points and boundary concentration is forbidden. See also \cite{CY} for a similar result.
As for nodal solutions, the available theory is far less developed. Most existing results concern radial nodal solutions, whereas the nonradial case has received only limited attention.  Among the few available results in the nonradial setting, we mention the paper  \cite{ZYB} where the authors showed the existence of multipeak nodal solutions with any number of peaks placed at the origin and at the vertex of a regular polygon as $p\to \infty$, under some assumption on $\alpha$ when $\Omega$ is the unit ball.  For this class of solutions, the positive and negative peaks are arranged in an alternating pattern, all the peaks are simple and isolated.\\
We next review the main results on radial solutions, where the theory is richer,  of course when $\Omega=B_1\subset \R^2$.  For any $j\geq 1$ there exists a radial solution to  \eqref{1.0} with $j$ nodal regions. In the case $j=1$, the corresponding solution is positive. These solutions have been studied for $j=1,2$ in the paper \cite{AGG} and for any $j\geq 1$ in \cite{SI} using a correlation among radial solutions to \eqref{1.0} with radial solutions to the Lane-Emden problem, namely \eqref{1.0} with $\alpha=0$.  In this papers it is shown that radial solutions concentrate only in the origin as $p\to \infty$ and that the concentration is not simple.
This results in solutions resembling a tower of bubbles with alternating peaks at the origin.  Moreover in \cite{AGG} and \cite{ADI} the authors compute the exact Morse index of the radial solutions when $p$ is large enough.  See also \cite{DSP} for an estimate of the Morse index of radial solutions which holds for every value of the exponent $p$. Starting form the study of the radial solutions in $B_1$, in \cite{AGG}, Amadori and Gladiali investigated least energy positive and sign-changing solutions by minimizing the energy functional associated with problem \eqref{1.0} on the Nehari manifold and on the nodal Nehari manifold, respectively.
The nodal Nehari manifold, introduced in \cite{BW0}, provides the natural constraint for sign-changing solutions.  Bartsch and Weth in \cite{BW0}, proved that,  least energy nodal solutions to problem \eqref{1.0} have exactly two nodal regions and Morse index $2$. Moreover these solutions are foliated Schwarz symmetric by \cite{BWW1}, namely when nonradial they are strictly monotone in the angular variable in half domain, see also \cite{PW}.  In \cite[Theorem 1.1]{AGG1}, it is proved that every radial nodal solution satisfies
$
m(u)\geq 4+\Bigl[\frac{\alpha}{2}\Bigr]
$
for every $p$.  As a consequence, least energy nodal solutions to \eqref{1.0} in $B_1$ are nonradial,  the Morse index of the radial solutions depends on $\alpha$ and it goes to infinity, as $\alpha\rightarrow\infty$, see also \cite{LoW}. Using a weighted eigenvalue problem, studied in \cite{AGG0} and \cite{AGG1},  Amadori and Gladiali were able to identify the simmetries of the eigenfunctions of the linearized problem to \eqref{1.0} at a radial solution and to characterize the degeneracy points.  
Motivated by these insights,  in \cite{AGG} they proved the existence of $\lceil\frac{\alpha}{2}\rceil$ positive nonradial solutions and $\lceil\frac{2+\alpha}{2}t\rceil-1$ distinct nodal nonradial solutions when $p$ is large, see Theorem 1.5 and Theorem 1.6 in \cite{AGG}. The value of $t$ is characterized in \cite{GGP2}, in the context of the Lane-Emden problem,  by $t=1+\frac{2\sqrt{e}}{\overline{t}}$ where $\overline{t}$ is the unique solution of the equation $2\sqrt{e}\log\overline{t}+\overline{t}=0$, (see also \cite{GS}). \\
This multiplicity result in $B_1$ is obtained by exploiting spaces of functions that are invariant under rotations by an angle $2\pi/k$, for suitable values of $k$. We denote, here, these spaces by $H^1_{0,k}(B_1)$. Restricting the energy functional to these spaces, one can minimize it on the Nehari manifold and on the nodal Nehari manifold to obtain least energy positive and sign-changing solutions, respectively.
A priori, these least energy solutions in $H^1_{0,k}(B_1)$ may be radial. However, by computing the Morse index of the radial least energy positive and least energy nodal solutions in $H^1_{0,k}(B_1)$, one can show that, for suitable values of $k$, depending on $\alpha$, these solutions must be nonradial.  This is done in \cite{AGG}. These solutions are symmetric with respect to the bisector of a sector of opening angle $2\pi/k$ and are strictly decreasing with respect to the polar angle in one half of the sector, that is, in a sector of opening angle $\pi/k$ by \cite{G}.
In \cite{GS}, Gladiali and Stegel identify three possible configurations of the nodal sets of these solutions. In the first case, the solution has $2k$ nodal regions and the nodal set contains a connected component joining the origin to the boundary $\partial B_1$. In the second case, it has $k+1$ nodal regions, with one nodal domain contained in a sector of angle $2\pi/k$ and the remaining region being $k$-invariant and connected.  Finally, in the third case the solution is highly symmetric and $k$-invariant, with exactly two nodal regions,  the closure of the nodal set is contained in $B_1$ and does not intersect $\partial B_1$.
These last case correspond to \emph{quasiradial solution}. See Theorem 1.4 and Corollary 1.5 in \cite{GS}. Finally, we mention the result in \cite{A0}, where the author constructs positive and sign-changing solutions in the regime $p \to 1^+$, derives their asymptotic profiles and computes the corresponding Morse index. These results are then used in \cite{A} to prove a bifurcation phenomenon from the radial solution. All these last results hold when $\Omega=B_1$.\\

It can be seen that setting $\alpha=0$ in \eqref{1.0} makes the H\'{e}non problem degenerate into the Lane-Emden problem, which indicates an intimate link between the two equations. As reported in \cite{LIA}, sign-changing solutions with cluster structures have been constructed for the Lane-Emden equation as $p\to\infty$. These solutions exhibit non-simple concentration at the origin: their asymptotic profile consists of one radial positive bubble and $k$ symmetric non-radial negative bubbles. The negative bubbles concentrate more slowly, and their positions drift toward the origin at an even lower speed. These findings reveal the concentration mechanism for large-exponent solutions and lay a solid foundation for constructing non-radial sign-changing solutions.

Motivated by these recent results and the close connection between the two problems, this paper extends the relevant theories to the H\'{e}non equation. We prove the existence of such non-radial clustered sign-changing solutions and systematically describe their concentration features and asymptotic behaviors as $p\to\infty$. We observe that this construction requires some symmetry of the domain $\Omega$,  but extends the previous results to a broader class of domains than $B_1$.

In order to state it rigorously let us assume $ 0 \in \Omega $ and $\Omega $ to be invariant with respect to reflection across the \( x_1 \) axis:
\begin{equation}\label{1.1}x = (x_1, x_2) \in\Omega \, \, \quad \Longleftrightarrow \quad \overline{x} := (x_1, -x_2) \in\Omega. \end{equation}

We also ask \( \Omega \) to be a \( k \)-symmetric domain, for some integer \( k \) depending on $\alpha$, that is invariant under a rotation of angle \(\frac{2\pi}{k}\) with respect to the origin:
\begin{equation}\label{1.2}x \in \Omega \quad \Longleftrightarrow \quad e^{i\frac{2\pi}{k}} x := \left( x_1 \cos \frac{2\pi}{k} - x_2 \sin \frac{2\pi}{k}, x_1 \sin \frac{2\pi}{k} + x_2 \cos \frac{2\pi}{k} \right) \in \Omega, \end{equation}
where the values of $k$ satisfy
\begin{equation}\label{k.1}
\frac{3(2+\alpha)}{2}< k\leq\lceil\frac{2+\alpha}{2}t\rceil-1, \end{equation}
with $ t \sim 5.1869$. The value of $t$ is characterized in
\cite{GS}  by $t=1+\frac{2\sqrt{e}}{\overline{t}}$ where $\overline{t}$ is the unique solution of the equation $2\sqrt{e}\log\overline{t}+\overline{t}=0$. This also shows that the number \begin{equation}\label{0k0}k_0:= \frac{2+\alpha}{2}t,\end{equation} is the unique solution of the equation
\begin{equation}\label{k0}
\frac{1}{2} + \frac{2k_0-2-\alpha}{2k_0+2+\alpha} \log \frac{2(2+\alpha)}{2k_0-2-\alpha} = 0 .
\end{equation}

Obviously the ball $B_1$ satisfies both the assumption \eqref{1.1} and \eqref{1.2}.
Let us also introduce the following space of \( k \)-symmetric functions:\begin{equation}\label{1.3}
H^1_{0,k}(\Omega):= \left\{ u \in H^1_0(\Omega) : u(x) = u(\overline{x}) = u \left( e^{i\frac{2\pi}{k}} x \right) \text{ for a.e. } x \in \Omega \right\}. \end{equation}

Our main result reads as follows.

\begin{theorem}\label{athm1.1} Let $\alpha>0$ and let $k$ satisfy \eqref{k.1}.  Assume \( \Omega \subset \mathbb{R}^2 \) is a smooth bounded domain such that \( 0 \in \Omega \) which satisfies assumptions \eqref{1.1} and \eqref{1.2}. Then there exists \( p_k > 1 \) such that, for \( p > p_k \), there is a nodal, $k$-symmetric solution $ u_{k,p} \in H^1_{0,k}(\Omega) $ to \eqref{1.0} that satisfies
\begin{align*}
u_{k,p} = \tau_p (PU_{1} - \eta_p PU_{2}) + o\left(\frac{1}{p}\right), \quad \text{in} \ L^\infty_{\text{loc}}(\overline{\Omega} \setminus 0)\end{align*}
as $ p \to +\infty $, where $ PU_{1} $ and $ PU_{2} $ are the projections on $ H^1_0(\Omega) $ of
\begin{align*}U_{1}(x) := \log \frac{2(2+\alpha)^{2}\delta_p^{2+\alpha}}{(|x|^{2+\alpha} + \delta_p^{2+\alpha})^{2}}, \quad \text{which solves } -\Delta U_{1} =|x|^{\alpha} e^{U_{1}}, \quad \text{in} \ \mathbb{R}^2\end{align*}
and
\begin{align*}U_{2}(x) := \log \frac{8k^2\beta_p^{2k}}{(|x^k - \rho_p^k|^{2} + \beta_p^{2k})^2}, \quad \text{which solves } -\Delta U_{2} =
|x|^{2k-2}e^{U_{2}}, \quad \text{in } \mathbb{R}^2.\end{align*}
We denote by $\rho_p^k$ the point $\rho_p^k:=\binom{\rho_p^k}{0}\in\mathbb{C}$. All the real parameters, $\tau_p,\eta_p,\delta_p,\beta_p,\rho_p$ are positive and as $p\to+\infty$ satisfy
\begin{align*}
\tau_p\sim \frac{A_k}{p},\quad \eta_p\sim \eta_{\infty,k}:=\frac{2(2+\alpha)}
{2k-2-\alpha},
\end{align*}
\begin{align*}
\delta_p\sim B_k e^{-(b_k+o(1))p},\quad \beta_p\sim C_k e^{-(c_k+o(1))p}\quad \text{and}\quad \rho_p\sim R_k e^{-(r_k+o(1))p},
\end{align*}
where $B_k$, $C_k$ and $R_k$ are positive constants which only depend on $k$ and $\alpha$ and
\begin{align}\label{abrn}
A_k&:=\sqrt{e}\left(\frac{2(2+\alpha)}
{2k-2-\alpha}
\right)^{-\frac{8k(2+\alpha)}{
(2k+2+\alpha)^2}}, \ \ \ \ \ \ \ \ \ \ \ \ \ \ \ \ \ \ \
b_k:=\frac{1}{2(2+\alpha)}\left(1+\frac{16k(2+\alpha)}{(2k+2+
\alpha)^2
}\log\frac{2k-2-\alpha}{2(2+\alpha)}\right),\notag\\
c_k&:=\frac1{4k}\left(1+2\frac{(2k-2-\alpha)^2}{
(2k+2+
\alpha)^2}\log \frac{2k-2-\alpha}{2(2+\alpha)}\right),\ \ \ \ \ \ \ \
r_k:=\frac{2(2k-2-\alpha)}{(2k+2+
\alpha)^2}
\log\frac{2k-2-\alpha}{2(2+\alpha)}.
\end{align}
Moreover, we have 
\begin{align}\label{ffff+}
p|u|^{p+1} &\rightarrow 0, \ \text{in}\ L^\infty_{\text{loc}}(\Omega \setminus \{0\}),\notag\\
p\int_{\Omega} |x|^{\alpha}u(x)_+^{p+1} dx
&\rightarrow 4(2+\alpha)\pi e \left( \frac{2k-2-\alpha}{2(2+\alpha)} \right)^{\frac{16k(2+\alpha)}{
(2k+2+\alpha)^2}},\\
p\int_{\Omega} |x|^{\alpha}u_{-}(x)^{p+1} dx
&\rightarrow8k\pi e \left( \frac{2(2+\alpha)}{2k-2-\alpha} \right)^{\frac{2(2k-2-\alpha)^{2}}{
(2k+2+\alpha)^2}}.\notag
\end{align}
\end{theorem}
\begin{remark}\label{r1.3}
Since the parameters satisfy \eqref{abrn} and $k$ satisfies \eqref{k.1}, it holds that
$$0<\delta_p <\beta_p < \rho_p ,$$
which shows that the positive bubble concentrates more rapidly than its negative counterpart and both concentration rates exceed the speed at which the negative peaks converge toward the origin.\end{remark}
\begin{remark}\label{r1.2} The clustering behavior requires in particular that the speed at which the negative peaks approach the positive one is slower than the rate of concentration of the negative bubble, $0<\delta_p < \rho_p $ and $0<\beta_p < \rho_p $, namely
\begin{equation*}(2k-2-\alpha) \log \frac{2(2+\alpha)}{2k-2-\alpha} < 0
\quad \text{and} \quad \frac{1}{2} + \frac{2k-2-\alpha}{2k+2+\alpha} \log \frac{2(2+\alpha)}{2k-2-\alpha} > 0 .
\end{equation*}
This suggests the idea that these solutions can be "quasiradial" as defined in \cite{GS}.\\
By the first inequality we have the bound $k>3\frac{2+\alpha}{2}$. 
Moreover, since the function $\frac{1}{2} + \frac{2k-2-\alpha}{2k+2+\alpha} \log \frac{2(2+\alpha)}{2k-2-\alpha}$ is strictly decreasing,  together with \eqref{0k0} and \eqref{k0}, gives that $k<\frac{2+\alpha}{2}t$ and justifies assumption \eqref{k.1}.

\end{remark}

Let us conclude with a few remarks on this result.\\
First, we observe that the result in \cite{LIA} builds upon the previous work \cite{GI}, where the authors proved the existence of quasiradial solutions to the Lane-Emden equation in the unit disk $B_1$ for large $p$, working in the spaces $H^1_{0,k}$ with $k=4,5$ (see \cite[Theorem 1.4]{GI}). In the same paper, they also proved that the least energy nodal solution to the Lane-Emden problem in $H^1_{0,k}$,  for $k=3,4,5$,  is radial when $p$ is close to $1$ and nonradial for large $p$ (see Theorem 1.3). In fact, this change of symmetry is explained by a bifurcation from the radial least energy nodal solution (see Theorem 1.5). The existence result in \cite{LIA} strongly suggests that the solutions constructed there coincide with the quasiradial solutions found in \cite{GI}, since they arise precisely for the same values $k=4,5$.\\
The counterpart of these results for the Hénon problem has been investigated in \cite{AGG} and \cite{GS}. In \cite[Theorem 1.6]{AGG}, the authors prove the existence of at least $\lceil \frac{2+\alpha}{2}t\rceil-1$ distinct nonradial solutions for large $p$, where $t$ is the constant appearing in \eqref{k.1}. The paper \cite{GS} studies the nodal structure of the least energy solutions in the spaces $H^1_{0,k}$ and concludes that the nonradial solutions found in \cite{AGG} are quasiradial whenever
\[
k>\left[\frac{2+\alpha}{2}\gamma-1\right],
\]
where $\gamma\approx4.859$ is characterized in \cite[Proposition 2.6]{GS}. In Theorem \ref{athm1.1}, we prove the existence of clustered solutions to the Hénon problem for every $k$ satisfying \eqref{k.1}, and in particular for
\[
\frac{3(2+\alpha)}{2}<k\le\left[\frac{2+\alpha}{2}t-1\right].
\]
This suggests two possible scenarios: either the nonradial solutions obtained in \cite{AGG} are quasiradial also for these values of $k$, or the clustered solutions constructed in Theorem \ref{athm1.1} may have a nodal line that touches the boundary of $\Omega$ for low values of $k$. 
Moreover, this also suggests that quasiradial solutions may exist for smaller values of $k$, and that the lower bound in \eqref{k.1} reflects a technical limitation of our method rather than an intrinsic feature of the problem. We believe that this issue deserves further investigation. \\
The comparison between the Lane-Emden and the Hénon problems reveals both similarities and significant differences. In the case of concentration away from the origin, the two problems exhibit essentially the same behavior, and the construction relies on the same standard bubbles (see \eqref{001.040}). On the other hand, when concentration occurs at the origin, the Hénon problem displays a different profile, involving the bubbles in \eqref{1.040}, whereas the Lane-Emden problem is governed by the standard Liouville bubble.\\
The clustered solutions constructed here combine these two features. The positive peak is described by the Hénon bubble centered at the origin, while the negative peaks develop exactly as in \cite{LIA}. Their configuration is determined solely by the imposed $k$-symmetry and, in particular, is independent of the Hénon exponent $\alpha$. This further highlights that the role of $\alpha$ is confined to the profile of the concentrating positive peak, whereas the geometry of the surrounding negative peaks is dictated by the symmetry of the construction.\\
Although symmetry is a crucial ingredient in our construction, it remains an open question whether these solutions also exist in general, non-symmetric domains $\Omega$.
The Hénon problem continues to offer a wide range of interesting phenomena to explore. The variety of possible solution profiles is far from being completely understood, and several configurations, including tower-type solutions in a general domain $\Omega$, remain open problems for future investigation.

The paper is organized as follows: in Section 2, we first present the leading-order term of the approximate solution. Following the ideas in \cite{LIA, PMA,PT}, we further introduce the second-order and third-order correction terms, and finally present the overall form of the solution to the target equation \eqref{1.0}.
In Section 3, we first determine all parameters of the approximate solution, express them as functions of $\eta$, we clarify their dependence on $p$, and provide the preliminary estimates required for subsequent arguments. We then conduct error analysis separately for the neighborhoods of positive peaks, negative peaks and regions far from all peaks. Finally, we choose an appropriate  weighted norm to complete the overall error estimation. In Section 4, we establish the invertibility of the linear operator $\mathcal{L}$ in \eqref{1.9}. In Section 5, we firstly apply the implicit function theorem to reduce the original equation \eqref{1.0} to a finite-dimensional problem. We then demonstrate that solving this problem amounts to seeking critical points of a real-valued function of one variable. We also prove that these critical points exist and are exactly the minimizers. The proof of Theorem \ref{athm1.1} is thus completed.

\section{Preliminaries}

In this section, we present the essential elements needed to obtain the solution and establish its global ansatz. 

\vskip 0.1cm

\subsection*{ The first order term of the ansatz: the bubbles.}
Let us recall the standard bubble
\begin{equation}\label{001.040}U_{0}(x) := \log \frac{8}{(|x|^{2} + 1)^2},
\end{equation}
which solves
$$
-\Delta U_{0}(x) = e^{U_{0}(x)}, \quad \text{in} \ \mathbb{R}^2,$$
together with its rescaling
$ U_0\left(\frac{x}{\delta}\right) - 2\log\delta$ and its translations $U_0\left(x-x_0\right)$ for $x\in\mathbb{R}^2$.

For what concerns the H\'{e}non type problem \begin{equation}\label{01.040}
-\Delta U(x) = |x|^{\alpha}e^{U(x)}, \quad x \in \mathbb{R}^2,
\end{equation}
it is known that the radial solutions are given by
\begin{equation}\label{1.040}
U_{\alpha}(x) :=U_0\left(|x|^{\frac{2+\alpha}{2}}\right)
+2\log\frac{2+\alpha}{2}
= \log \frac{2(2+\alpha)^{2}}{(|x|^{2+\alpha} + 1)^2},
\end{equation}
together with its rescaling
\begin{equation}\label{u00}
U_1(x) := U_{\alpha}\left(\frac{x}{\delta}\right) - (2+\alpha)\log\delta = \log \frac{2(2+\alpha)^{2}\delta^{2+\alpha}}{(|x|^{2+\alpha} + \delta^{2+\alpha})^{2}}.
\end{equation}
It is also known that when $\alpha$ is even and in particular for $\alpha=2(k-1)$ with $k\in\mathbb{N}$ and $k\geq2$, equation \eqref{01.040} admits also nonradial solutions (see Theorem 1.1 in \cite{PTG}). These solutions are  given in complex notation by \begin{equation}\label{u01}
U_2(x) := U_{0} \left( \frac{x^k - \rho^k}{\beta^k} \right) + \log \left( k^2 \right) - 2k \log \beta =\log \frac{8k^2\beta^{2k}}{(|x^k - \rho^k|^{2} + \beta^{2k})^2}.
\end{equation}
for $\beta, \rho > 0$, $k \in \mathbb{N}$,  where by $x^k$ we mean $x^k=|x|^k (\cos k \theta, \sin k\theta)$, and solves
\begin{equation}\label{01.0040}
-\Delta U_{2,p} =
|x|^{2k-2}e^{U_{2,p}}, \quad x \in \mathbb{R}^2.\end{equation}
Both $U_1$ and $U_2$ will be used in our ansatz for the solution to problem \eqref{1.0} to approximate it at a first order.

Here and in the rest of the paper, we use the complex notation to mean
\[
x^k = \left( \Re \left( (x_1 + ix_2)^k \right), \Im \left( (x_1 + ix_2)^k \right) \right) \in \mathbb{R}^2, \quad \text{for } x = (x_1, x_2) \in \mathbb{R}^2.
\]

We also agree to identify any positive real number $t > 0$ with its corresponding point on the positive real half-line
\[
t = t + i \cdot 0 = (t, 0) \in \mathbb{R}^2.
\]

\subsection*{ A refinement of the ansatz: the second and the third order terms.}
To obtain a sufficiently small error term  (as will be demonstrated in the next subsection), we refine the approximation of the solution by incorporating additional lower-order terms following the idea of \cite{PMA} and \cite{PT}.

Hence we consider entire radial solutions to the following problems in the whole \(\mathbb{R}^2\):
\begin{align*}
\Delta V_{\alpha} + |x|^{\alpha}e^{U_{\alpha}} V_{\alpha}& = |x|^{\alpha}e^{U_{\alpha}} f_1(U_{\alpha}) , \quad \Delta W_{\alpha} + |x|^{\alpha}e^{U_{\alpha}} W_{\alpha} = |x|^{\alpha}e^{U_{\alpha}} f_2(U_{\alpha},V_{\alpha}),
\\
\Delta V_{0} + e^{U_{0}} V_{0} &= e^{U_{0}} f_1(U_{0}), \quad \ \ \quad \quad \ \quad \Delta W_{0} + e^{U_{0}} W_{0} = e^{U_{0}} f_2(U_{0},V_{0}),
\end{align*}
where \( U_{\alpha} \) and \( U_{0} \) are defined in \eqref{1.040} 
 and \eqref{001.040}, and 
\begin{align*}
f_1(u) := \frac{u^2}{2} \quad \text{and} \quad f_2(u,v) := uv - \frac{u^3}{2} - \frac{1}{2} \left( v - \frac{u^2}{2} \right)^2.
\end{align*}
The choice of \( f_1 \) and \( f_2 \) is done in order to control the error term and comes from \cite{LIA} and \cite{PT}.

The existence of suitable radial solutions \( V_{\alpha}, V_{0}, W_{\alpha}, W_{0} \) with logarithmic behavior at infinity follows by Lemma A.10 in \cite{PT}. In particular, \( V_{\alpha}, V_{0} \) have the following asymptotic behavior, as \( |x| \) goes to \( +\infty \):

\begin{equation}\label{1.40-bis}
V_{\alpha}(x) = C_1 \log |x| + O \left( \frac{1}{|x|} \right), \quad \nabla V_{\alpha}(x) = C_1 \frac{x}{|x|^2} + O \left( \frac{1}{|x|^2} \right),
\end{equation}
\begin{equation}\label{1..40}
V_{0}(x) = C_3 \log |x| + O \left( \frac{1}{|x|} \right), \quad \nabla V_{0}(x) = C_3 \frac{x}{|x|^2} + O \left( \frac{1}{|x|^2} \right)
\end{equation}
and similarly
\begin{equation}\label{1.41}
W_{\alpha}(x) = C_2 \log |x| + O \left( \frac{1}{|x|} \right), \quad \nabla W_{\alpha}(x) = C_2 \frac{x}{|x|^2} + O \left( \frac{1}{|x|^2} \right),
\end{equation}
\begin{equation}\label{1..41}
W_{0}(x) = C_4 \log |x| + O \left( \frac{1}{|x|} \right), \quad \nabla W_{0}(x) = C_4 \frac{x}{|x|^2} + O \left( \frac{1}{|x|^2} \right),
\end{equation}
where
\begin{align}\label{C..i1}
C_1 &= \int_{0}^{+\infty} r^{1+\alpha}\frac{1-r^{2+\alpha}}
{1+r^{2+\alpha}} e^{U_{\alpha}(r)} f_1(U_{\alpha}(r))dr=2(2+\alpha)\big(\log2
+2\log(2+\alpha)-3\big), \notag\\
C_2 &= \int_{0}^{+\infty} r^{1+\alpha}\frac{1-r^{2+\alpha}}{1+r^{2+\alpha}}^{} e^{U_{\alpha}(r)} f_2(U_{\alpha}(r), V_{\alpha}(r))dr,\\
C_3 &= \int_{0}^{+\infty} r\frac{1-r^{2}}{1+r^{2}} e^{U_{0}(r)} f_1(U_{0}(r))dr=12\big(\log2-1\big), \notag\\
C_4 &= \int_{0}^{+\infty} r\frac{1-r^{2}}{1+r^{2}} e^{U_{0}(r)} f_2(U_{0}(r),V_{0}(r))dr,\notag
\end{align}
where $C_2$ and $C_4$ do not have an explicit form, but they are finite, since, by \eqref{001.040}, \eqref{1.040}, \eqref{1.40-bis} and \eqref{1..40},  the functions $U_{0}$, $U_{\alpha}$, $V_{\alpha}$ and $V_{0}$
 grow logarithmically at infinity.  Moreover, $e^{U_{0}}$ and $e^{U_{\alpha}}$ decay as $\frac{1}{r^{4}}$ and $\frac{1}{r^{4+2\alpha}}$ respectively, which, together with the fact that $f$ grows polynomially,  implies that $C_2$ and $C_4$ are finite. Moreover, by direct computation, we have $$C_4=C_{3}-\frac{1}{2}\displaystyle\int_{0}^{+\infty} r\frac{1-r^{2}}{1+r^{2}} e^{U_{0}(r)}
\left( V_{0}(r) - \frac{U_{0}^2(r)}{2} -U_{0}(r) \right)^2dr:=C_{3}-\frac{1}{2}C_{5}.$$
A numerical computation gives $C_{5}\sim-15.012$, so that $C_4\neq0$. The same argument gives $C_2\neq0$.

Next, we define, for $\delta>0$,
\begin{align}\label{gy11}
V_1(x) := V_{\alpha}\left(\frac{x}{\delta}\right), \ \  W_1(x) := W_{\alpha}\left(\frac{x}{\delta}\right),  \ \  V_2(x) := V_{0}\left(\frac{x^k - \rho^k}{\beta^k}\right),  \ \  W_2(x) := W_{0}\left(\frac{x^k - \rho^k}{\beta^k}\right),
\end{align}
that solve,  for \( x \in \mathbb{R}^2 \)
\begin{align}\label{VY1}
\Delta V_1(x) + |x|^{\alpha}e^{U_1(x)}V_1(x) = |x|^{\alpha} e^{U_1(x)}f_1\left(U_{\alpha}\left(
\frac{x}{\delta}\right)\right),
\end{align}
\begin{align}\label{WY1}
\Delta W_1(x) +|x|^{\alpha} e^{U_1(x)}W_1(x) = |x|^{\alpha} e^{U_1(x)}f_2\left(U_{\alpha}\left(\frac{x}
{\delta}\right), V_{\alpha}\left(\frac{x}{\delta}
\right)\right),
\end{align}
\begin{align}\label{VY2}
\Delta V_2(x) + |x|^{2k-2}e^{U_2(x)}V_2(x) = |x|^{2k-2}e^{U_2(x)}f_1\left(U_{0}\left(\frac{x^k - \rho^k}{\beta^k}\right)\right),
\end{align}
\begin{align}\label{WY2}
\Delta W_2(x) + |x|^{2k-2}e^{U_2(x)}W_2(x) = |x|^{2k-2}e^{U_2(x)}f_2\left(U_{0}\left(\frac{x^k - \rho^k}{\beta^k}\right), V_{0}\left(\frac{x^k - \rho^k}{\beta^k}\right)\right).
\end{align}
They will be all used to get extra lower order terms for the approximation of the solution.

\subsection*{The final ansatz: fixing the Dirichlet boundary conditions.}

We will look for \( k \)-symmetric solutions to \eqref{1.0} (i.e. solutions in the space defined in \eqref{1.3}) in the form
\begin{equation}\label{1.5}
u = \Upsilon + \phi,
\end{equation}
where \( \Upsilon \) is explicit and given by:
\begin{equation}\label{1.6}
\Upsilon := \Upsilon_{1} + \Upsilon_{2},
\end{equation}
for some \( \tau, \eta > 0 \),  with
\[
\Upsilon_{1} := \tau \left( PU_1 + \frac{PV_1}{p} + \frac{PW_1}{p^2} \right) \ \ , \ \   \Upsilon_{2} := - \tau \eta \left( PU_2 + \frac{PV_2}{p} + \frac{PW_2}{p^2} \right), \]
 where  \( U_1 \), \( U_2 \), \( V_1\), \( V_2\), \( W_1 \) and \( W_2 \) are the ones defined in \eqref{u00}, \eqref{u01} and \eqref{gy11} and we are considering their projections on \( H_0^1(\Omega) \). 
For any function \( \varphi \in H^1(\Omega) \), its projection on \( H_0^1(\Omega) \), denoted by \( P\varphi \), is characterized as the solution of
\begin{equation*}
-\Delta(P\varphi) = -\Delta\varphi,\ \  \text{ in} \ \Omega, \quad P\varphi = 0, \ \text{ on} \ \partial\Omega,\end{equation*} which, together with the definitions of  \( U_1 \), \( U_2 \),  \( V_1\), \( V_2\), \( W_1 \) and \( W_2 \), gives that  $\Upsilon$  is \( k \)-symmetric.
Moreover, from the definition of $\Upsilon$, we have that the function $\Upsilon$ depends on $\eta$, $\tau$, $\delta$, $\beta$ and $\rho$. In particular the parameter $\eta$ will be chosen close to \begin{equation}\label{n1}
 \eta_{\infty,k}= \frac{2(2+\alpha)}{2k-2-\alpha},
 \end{equation}
by Proposition \ref{e.3}.  All the remaining parameters go to 0, as $p\rightarrow\infty$,  as we can see by Lemma \ref{lem2.1}.

\subsection*{The reduction process.}
Our aim is to find a solution to \eqref{1.0} as in \eqref{1.5}, where the higher order term \(\phi \in H_{0,k}^{1}(\Omega)\) solves the equation
\begin{equation}\label{1.7}
\mathcal{E} + \mathcal{L}\phi + \mathcal{N}(\phi) = 0,
\end{equation}
and \(\mathcal{E}\) is the error term:
\begin{equation}\label{1.8}
\mathcal{E} := \Delta\Upsilon + |x|^{\alpha} |\Upsilon|^{p-1}\Upsilon.
\end{equation}
$\mathcal{L} \phi $ is the linear term:
\begin{equation}\label{1.9}
\mathcal{L} \phi := \Delta \phi + p |x|^{\alpha} |\Upsilon|^{p-1} \phi
\end{equation}
and \( \mathcal{N}(\phi) \) is the nonlinear term:
\begin{equation}\label{1.10}
\mathcal{N}(\phi) := |x|^{\alpha}|\Upsilon + \phi|^{p-1} (\Upsilon + \phi) - |x|^{\alpha}|\Upsilon|^{p-1} \Upsilon - p|x|^{\alpha} |\Upsilon|^{p-1} \phi .
\end{equation}
Here we quote a result we will need the sequel.
\begin{lemma}\cite[Lemma 3.5]{LIA}\label{pe.0}
For any real numbers \( a, b, c \) such that for some \( C > 0 \), one has
$$
|b| + |c| < C p,\quad \big( -1+ \frac{1}{C}\big)p\leq a \leq C,
$$
the following asymptotic expansions hold true, as \( p \) goes to \( +\infty \):
\begin{align}\label{pe.1}
&\left(1+\frac{a}{p} + \frac{b}{p^2} + \frac{c}{p^3} \right)^{p} =e^{a}\Bigg(
1+\frac{\big(b-\frac{a^{2}}{2}\big)}{p}
+ \frac{\Big(c-ab+\frac{a^{3}}{3}
+\frac{\big(b-\frac{a^{2}}{2}\big)^{2}
}{2}\Big)}{p^2}+O\Big(
\frac{1+a^{6}+b^{6}+c^{6}}{p^{3}}
\Big)
\Bigg).
\end{align}
\end{lemma}
\section{the choice of the parameters and the estimates of the error}
In this section, we first determine the positive parameters appearing in the ansatz as functions of $p$,  (except for the parameter $\eta$), derive the relevant preliminary estimates, and estimate the error term in the framework of weighted norms.
\subsection*{The choice of parameters as functions of \(\eta\) and $p$.}\
Here we derive the parameters \(\tau,\delta,\beta,\rho\) as functions of the parameter $\eta$ and $p$.  Of course also the parameter $\eta$ will depend on $p$ and we will choose it in Proposition \ref{e.3}.  But, at this step, we only assume that $\eta\in\big[\eta_{\infty,k} - \varepsilon, \eta_{\infty,k} + \varepsilon\big],$ for a small $\varepsilon$, where $\eta_{\infty,k}$ is as defined in \eqref{n1}.
The parameters are defined by the following relations:
\begin{equation}\label{1.11}
\begin{cases}
\frac{\tau}{\delta^{2+\alpha}} = (p\tau)^p, \\[3mm]
\frac{k^2 \rho^{2k-2-\alpha}}{\beta^{2k}} \tau\eta= (p\tau\eta)^p ,\\[3mm]
\left(4 - \frac{C_3}{p} - \frac{C_4}{p^2}\right)\log \rho^{k\eta} -
\left(2(2+\alpha) - \frac{C_1}{p} - \frac{C_2}{p^2}\right)\log \delta - \log 2(2+\alpha)^{2} \\\quad
 - 2k\eta\pi \left(4 - \frac{C_3}{p} - \frac{C_4}{p^2}\right)H(0,0)
 + 2\pi \left(2(2+\alpha) - \frac{C_1}{p} - \frac{C_2}{p^2}\right)H(0,0) = p, \\[3mm]
\left(2(2+\alpha) - \frac{C_1}{p} - \frac{C_2}{p^2}\right)\log \rho^{\frac{1}{\eta}} -\frac{2}{\eta}\pi \left(2(2+\alpha) - \frac{C_1}{p} - \frac{C_2}{p^2}\right)H(0,0)- \log 8\\\quad-
\left(4 - \frac{C_3}{p} - \frac{C_4}{p^2}\right)\log \beta^k  + 2k\pi \left(4 - \frac{C_3}{p} - \frac{C_4}{p^2}\right)H(0,0) = p,
\end{cases}
\end{equation}
where $C_1$, $C_2$, $C_3$, $C_4$ are given by \eqref{1.40-bis}, \eqref{1..40}, \eqref{1.41} and \eqref{1..41}.

\begin{lemma}\label{lem2.1}
Assume $k$ satisfies \eqref{k.1} and $\eta\neq 1$. Then the parameters \(\tau, \delta, \beta, \rho\) satisfy, for $p$ large:
\begin{equation*}
\begin{cases}
\tau = \frac{\sqrt{e}}{p} \eta^{
        -\frac{2k\eta
       }
        {(2k\eta-2-\alpha)(1+\frac{1}{\eta})}}
        +O\left(\frac{\log p}{p^{2}}\right), \\[4mm]
\rho = e^{\frac{ \log \eta}{(2k\eta-2-\alpha)(1+\frac{1}{\eta})}p} \left(K_1(\eta) + O \left(\frac{1}{p}\right)\right), \\[4mm]
\delta = e^{-\frac{p}{2(2+\alpha)}
        \left(1-\frac{4k\eta\log\eta}
        {(2k\eta-2-\alpha)(1+\frac{1}
       {\eta})}
        \right)}\left(K_2(\eta)+
        O\left(\frac{1}{p}\right)\right), \\[4mm]
\beta = e^{-\frac{p}{4k}
        \left(1-\frac{ 2\log \eta}
        {(\frac{2k\eta}{2+\alpha}-1)
        (1+\eta)}
        \right)}\left(K_3(\eta)
        +O\left(\frac{1}{p}\right)\right),
\end{cases}
\end{equation*}
where \( K_i(\eta) \) are positive constants, smoothly depending on \(\eta\), which are uniformly bounded when \( \frac{2+\alpha}{2k} < C \leq \eta \leq C' < 1 \).
The estimates are uniform for $\eta\in\big[\eta_{\infty,k} - \epsilon, \eta_{\infty,k} + \epsilon\big],$  for a small $\varepsilon$, 
and the same estimates hold true for the derivative in \(\eta\) of all the parameters.
\end{lemma}
\begin{proof}
Let $t=\frac{4 - \frac{C_3}{p} - \frac{C_4}{p^2}}{2(2+\alpha) - \frac{C_1}{p} - \frac{C_2}{p^2}}$.
 From \eqref{1.11}, we have
    \begin{equation*}
\begin{cases}
\left(2(2+\alpha) - \frac{C_1}{p} - \frac{C_2}{p^2}\right)\log \frac{\rho^{k\eta t}}{\delta}  - \log 2(2+\alpha)^{2}
 - 2(k\eta t- 1)\pi \left(2(2+\alpha) - \frac{C_1}{p} - \frac{C_2}{p^2}\right)H(0,0) = p, \\[3mm]
\left(2(2+\alpha) - \frac{C_1}{p} - \frac{C_2}{p^2}\right)\log \frac{\rho^{\frac{1}{\eta}}}{\beta^{kt}} - \log 8  + 2\left(kt-\frac{1}{\eta}\right)\pi \left(2(2+\alpha) - \frac{C_1}{p} - \frac{C_2}{p^2}\right)H(0,0) = p.
\end{cases}
\end{equation*}
 Then \eqref{1.11} can be transformed into
    \begin{equation}\label{1.12}
        \begin{cases}
            -(2+\alpha)\log\delta = p\log p + (p-1)\log\tau, \\
            (2k-2-\alpha)\log\rho - 2k\log\beta = p\log p + (p-1)(\log\tau + \log\eta) - \log(k^2), \\
            2(2+\alpha)k\eta t\log\rho - 2(2+\alpha)\log\delta = p + \frac{C_1}{2(2+\alpha)} + 4(2+\alpha)(k\eta t-1)\pi H(0,0) + \log 2(2+\alpha)^{2} + O\left(\frac{1}{p}\right), \\
            \frac{2(2+\alpha)}{\eta}\log\rho - 2(2+\alpha)kt\log\beta = p + \frac{C_1}{2(2+\alpha)} - 4(2+\alpha)(kt-\frac{1}{\eta})\pi H(0,0) + \log 8 + O\left(\frac{1}{p}\right).
        \end{cases}
    \end{equation}
Then, we can get the following two equalities:
    \begin{equation}\label{1.13}
        \begin{cases}
            (2+\alpha)kt\eta\log\rho = \frac{p}{2} - p\log p - (p-1)\log\tau + \frac{C_1}{4(2+\alpha)} + 2(2+\alpha)(kt\eta-1)\pi H(0,0) \\
            \quad + \frac{\log 2(2+\alpha)^{2}}{2} + O\left(\frac{1}{p}\right), \\
            \left(\frac{2}{t\eta} - 2k + 2+\alpha\right)\log\rho = \frac{p}{(2+\alpha)t} - p\log p - (p-1)(\log\tau + \log\eta) \\
            \quad -\frac{4}{t}(kt-\frac{1}{\eta})\pi H(0,0) + \frac{C_1}{2t(2+\alpha)^{2}} + \frac{\log 8}{(2+\alpha)t} + 2\log k + O\left(\frac{1}{p}\right),
        \end{cases}
    \end{equation}
which gives
    \begin{align*}
        \log\rho &= \frac{\log\eta^{(p-1)}
        +\frac{\left(t(2+\alpha)-2\right) \left(p+\frac{C_1}{2(2+\alpha)}\right)
        }{2t(2+\alpha)}+
        \frac{\log 2(2+\alpha)^{2}}{2}-
        \frac{\log 8}{(2+\alpha)t}
         - \log k^{2} + O\left(\frac{1}{p}\right)}{(2+\alpha)kt\eta - \frac{2}{t\eta} + 2k -2- \alpha} + 2\pi H(0,0)\\
        &=  \frac{\log\eta^{(p-1)}+\frac{\left(t(2+\alpha)-2\right) \left(p+\frac{C_1}{2(2+\alpha)}\right)
        }{2t(2+\alpha)}+
        \frac{\log 2(2+\alpha)^{2}}{2}-
        \frac{\log 8}{(2+\alpha)t}
         - \log k^{2} }{(kt\eta-1)
        (2+\alpha+\frac{2}
{t\eta})}+ 2\pi H(0,0) + O\left(\frac{1}{p}\right).
    \end{align*}
    In view of the definition of $t$, we have $t=\frac{2}{2+\alpha}+O\left(\frac{1}
    {p}\right)$,  so that 
    \begin{align*}
        \rho =e^{\frac{ \log \eta}{(2k\eta-2-\alpha)(1+\frac{1}{\eta})}p} \left(K_1(\eta) + O \left(\frac{1}{p}\right)\right),
   \end{align*}
    which gives the estimate for \(\rho\). Then,  using again \eqref{1.12}, we have
    \begin{align*}
        \delta &= \rho^{kt\eta}e^{-\frac{p}{2(2+\alpha)}
        -\frac{C_1}{4(2+\alpha)^{2}}-2(kt\eta-1)\pi H(0,0)-\frac{\log 2(2+\alpha)^{2}}{2(2+\alpha)}+O\left(\frac{1}{p}\right)} \\
        &= \left(e^{\frac{ \log \eta}{(2k\eta-2-\alpha)(1+\frac{1}{\eta})}p}
         \left(K_1(\eta)+O\left(\frac{1}{p}\right)\right)
         \right)^{kt\eta}e^{-\frac{p}{2(2+\alpha)}}
        e^{-\frac{C_1}{4(2+\alpha)^{2}}-2(k\eta-1)\pi H(0,0)-\frac{\log 2(2+\alpha)^{2}}{2(2+\alpha)}}
       \left(1+O\left(\frac{1}{p}\right)\right) \\
        &=e^{-\frac{p}{2(2+\alpha)}}
       e^{\frac{ \frac{2k\eta}{2+\alpha}p\log \eta}{(2k\eta-2-\alpha)(1+\frac{1}
       {\eta})}}\left(K_2(\eta)+
        O\left(\frac{1}{p}\right)\right)
       \\
        &= e^{-\frac{p}{2(2+\alpha)}
        \left(1-\frac{4k\eta\log\eta}
        {(2k\eta-2-\alpha)(1+\frac{1}
       {\eta})}
        \right)}\left(K_2(\eta)+
        O\left(\frac{1}{p}\right)\right).
    \end{align*}
    From $\frac{1}{t}=\frac{2+\alpha}{2}+O\left(\frac{1}{p}\right)$,
     we can obtain
    \begin{align*}
        \beta &= \rho^{\frac{1}{kt\eta}}e^{-\frac{p}
        {2(2+\alpha)kt}
        -\frac{C_1}{4kt(2+\alpha)^{2}}+
        2\frac{kt-\frac{1}{\eta}}{kt}\pi H(0,0)-\frac{\log 8}{2tk(2+\alpha)}+O\left(\frac{1}{p}\right)} \\
        &= e^{\frac{ p\log \eta}
        {2k\eta(\frac{2k\eta}{2+\alpha}-1)
        (1+\frac{1}{\eta})}
        }e^{-\frac{p}
        {2(2+\alpha)kt}
        -\frac{C_1}{4kt(2+\alpha)^{2}}+
        2\frac{kt-\frac{1}{\eta}}{kt}\pi H(0,0)-\frac{\log 8}{2tk(2+\alpha)}}
       \left(K_1(\eta)+O\left(\frac{1}
        {p}\right)\right)^{\frac{1}{kt\eta}}\left(1+O\left
        (\frac{1}{p}\right)\right) \\
        &= e^{-\frac{p}{4k}
        \left(1-\frac{ 2\log \eta}
        {(\frac{2k\eta}{2+\alpha}-1)
        (1+\eta)}
        \right)}\left(K_3(\eta)
        +O\left(\frac{1}{p}\right)\right)
    \end{align*}
    and
    \begin{align*}
        \tau &= p^{-\frac{p}{p-1}}\delta^{-\frac{2+\alpha}{p-1}} \\
        &= p^{-\frac{p}{p-1}}e^{\frac{p}{2(p-1)}
        \left(1-\frac{4k\eta\log\eta}
        {(2k\eta-2-\alpha)(1+\frac{1}{\eta})}\right)}
        \left(1+O\left(\frac{1}{p}\right)
        \right)\\
        &= \frac{\sqrt{e}}{p} \eta^{
        -\frac{2k\eta
       }
        {(2k\eta-2-\alpha)(1+\frac{1}{\eta})}}
        +O\left(\frac{\log p}{p^{2}}\right).
    \end{align*}
\end{proof}
\subsection*{ Some crucial preliminary estimates.}
\begin{lemma}\label{lemd.1} Assume \( \alpha>0 \) is fixed and   \( k \) satisfies \eqref{k.1},
then there exists \( \epsilon > 0 \) such that if
\begin{align}\label{x.1}
\eta_{\infty,k} - \epsilon \leq \eta \leq \eta_{\infty,k} + \epsilon,
\end{align}
where \( \eta_{\infty,k} \) is defined in \eqref{n1}, then
\begin{align}\label{x.2}
\frac{2+\alpha}{k} < \eta < 1,
\end{align}
\begin{align}\label{x.3}
\frac{1}{2} + \frac{\log \eta}{\eta+1} > 0,
\end{align}
\begin{align}\label{x.4}
\frac{1}{2} + \frac{2k\eta^2}{(2k\eta-2-\alpha)(\eta+1)} \log \eta > 0
\end{align}
and
\begin{align}\label{x.5}
1 + \frac{(2k+2+\alpha)\eta}{(2k\eta-2-\alpha)(\eta+1)} \log \eta > 0.
\end{align}
\end{lemma}

\begin{proof} We just need to verify that all inequalities hold for $ \eta = \eta_{\infty,k} $ and then use the continuity with respect to $ \eta$. From \eqref{n1} and \eqref{k.1}, we can derive
 \begin{equation*}
1=
 \frac{1}{\frac{3}{2}
-\frac{1}{2}}
>\eta_{\infty,k}= \frac{1}{\frac{k}{2+\alpha}-\frac{1}{2}}
>\frac{2+\alpha}{k}
 \end{equation*}
and this proves \eqref{x.2}.

Next we observe that
\begin{align*}
\frac{1}{2} + \frac{\log \eta_{\infty,k}}{\eta_{\infty,k} + 1}
= \frac{1}{2} + \frac{2k-2-\alpha}{2k+2+\alpha} \log \frac{2(2+\alpha)}{2k-2-\alpha}
\end{align*}
and the last function is strictly decreasing with respect to $k$.
Since, by \eqref{k.1} and \eqref{0k0}, we can get $ k < k_0 $ and  $ k_0 $  is the unique solution to equation \eqref{k0}, then \eqref{x.3} follows.

In a similar manner, we have that
\begin{align*}
\frac{1}{2} + \frac{2k\eta_{\infty,k} ^2}{(2k\eta_{\infty,k}-2-\alpha)
(\eta_{\infty,k}+1)} \log \eta_{\infty,k} &=\frac{1}{2} +\frac{8k(2+\alpha)}{(2k+2+\alpha)^{2}}
\log \frac{2(2+\alpha)}{2k-2-\alpha}
\\&
\geq \frac{1}{2} +\frac{8k_0(2+\alpha)}{(2k_0+2+\alpha)^{2}}
\log \frac{2(2+\alpha)}{2k_0-2-\alpha}
\\&
> \frac{1}{2} + \frac{2k_0-2-\alpha}{(2k_0+2+\alpha)}
\log \frac{2(2+\alpha)}{2k_0-2-\alpha}
\\&
= 0,\end{align*}
where we make use of the facts that $\frac{8k(2+\alpha)}{(2k+2+\alpha)^{2}}
\log \frac{2(2+\alpha)}{2k-2-\alpha} $ is decreasing and \( k_0 \)  is the solution to equation \eqref{k0}. This proves \eqref{x.4}.

Finally, from \eqref{x.3}, we have
\begin{align*}
1 + \frac{(2k+2+\alpha)\eta_{\infty,k}}
{(2k\eta_{\infty,k}-2-\alpha)(
\eta_{\infty,k}+1)} \log \eta_{\infty,k} =1+\frac{2\log \eta_{\infty,k}}{\eta_{\infty,k}+1}=2\left( \frac{1}{2} + \frac{\log \eta_{\infty,k}}{\eta_{\infty,k} + 1}
\right)>0 ,
\end{align*}
so \eqref{x.5} holds true.
\end{proof}

\begin{lemma}\label{lem2.x}Assume $ \alpha>0$ is fixed,  $k$ satisfies \eqref{k.1} and $\eta $ as in \eqref{x.1}. Then there exists \(\varepsilon > 0\) such that, for \(p\) large enough,

\begin{align}\label{x.7}
\frac{\delta}{\rho} \leq Ce^{-\varepsilon p}, \end{align}

\begin{align}\label{x.8}
\frac{\beta}{\rho} \leq Ce^{-\varepsilon p}, \end{align}

\begin{align}\label{x.9}
\frac{\delta}{\rho^{\frac{4k\eta}{2+\alpha}}} \leq Ce^{-\varepsilon p}. \end{align}

Moreover, there exists \(\theta_0 > 0\) such that, if \(0 < \theta \leq \theta_0\), then

\begin{align}\label{x.10}
\frac{\delta^{1-\theta}}{\rho} \leq Ce^{-\frac{\varepsilon}{2}p}, \end{align}

\begin{align}\label{x.11}
\frac{\beta^{1-\theta}}{\rho} \leq Ce^{-\frac{\varepsilon}{2}p}. \end{align}
\end{lemma}
\begin{proof}
Let us choose
\begin{align*}
\varepsilon := \frac{1}{2(2+\alpha)} \min \left\{ 1, \frac{(2+\alpha)}{k} \left( \frac{1}{2} + \frac{\log(\eta_{\infty,k}-\epsilon)}{
\eta_{\infty,k}-\epsilon+1} \right), 1 + \frac{4k(\eta_{\infty,k}-\epsilon)^2
\log(\eta_{\infty,k}-\epsilon)}
{(2k(\eta_{\infty,k}-\epsilon)-2-\alpha)
(\eta_{\infty,k}-\epsilon+1)} \right\}
\end{align*}
where $\epsilon$ is as in \eqref{x.1}. 
By Lemma \ref{lem2.1} and  \eqref{x.2}, we can get
\begin{align*}
\frac{\delta}{\rho} &\leq Ce^{-\frac{p}{2(2+\alpha)}
        \left(1-\frac{4k\eta\log\eta}
        {(2k\eta-2-\alpha)(1+\frac{1}
       {\eta})}
        \right)} e^{-\frac{ \log \eta}{(2k\eta-2-\alpha)(1+\frac{1}{\eta})}p} \\& \leq Ce^{-\frac{p}{2(2+\alpha)} }
        e^{\frac{ \log \eta}{(2k\eta-2-\alpha)
        (1+\frac{1}{\eta})}(\frac{4k\eta }{2(2+\alpha)}-1)p}
        \leq Ce^{-\frac{p}{2(2+\alpha)}},
\end{align*}
for large \(p\), that is \eqref{x.7}.

Combining Lemma \ref{lem2.1}  and \eqref{x.3}, we can deduce
\begin{align*}
\frac{\beta}{\rho} &\leq Ce^{-\frac{p}{4k}
        \left(1-\frac{ 2\log \eta}
        {(\frac{2k\eta}{2+\alpha}-1)
        (1+\eta)}
        \right)}e^{-\frac{ \log \eta}{(2k\eta-2-\alpha)(1+\frac{1}{\eta})}p} \\& = e^{-\frac{p}{2k}
        \left(\frac{1}{2}+\frac{ \log \eta}
        {
        1+\eta}
        \right)} \leq Ce^{-\frac{1}{2k}
        \left(\frac{1}{2}+\frac{
        \log(\eta_{\infty,k}-\epsilon)}
        {\eta_{\infty,k}-\epsilon+1}\right) p} 
\end{align*}
and this proves \eqref{x.8}.

Similarly, by \eqref{x.4}, we have
\begin{align*}
\frac{\delta}{\rho^{\frac{4k\eta}{2+\alpha}}}
\leq Ce^{-\frac{p}{2+\alpha}\left(
\frac{1}{2}+\frac{2k\eta^2}
{(2k\eta-2-\alpha)(\eta+1)}\log\eta\right)}
\leq Ce^{-\frac{p}{2+\alpha}\left(
\frac{1}{2}+\frac{2k(\eta_{\infty,k}-\epsilon)^2}
{(2k\eta_{\infty,k}-\epsilon-2-\alpha)
(\eta_{\infty,k}-\epsilon+1)}\log(\eta_{\infty,k}-\epsilon)\right)}.
\end{align*}To prove \eqref{x.10} and \eqref{x.11}, we only have to consider the case \(\theta = \theta_0\).
Thanks to Lemma \ref{lem2.1}, we can obtain
\begin{align*}
\delta^{\theta_0} \geq \frac{1}{C} e^{-\frac{p}{2(2+\alpha)} \left( 1 - \frac{4k\eta^2 \log \eta}{(2k\eta-2-\alpha)(\eta+1)} \right) \theta_0} \geq \frac{1}{C} e^{-\frac{p}{2(2+\alpha)} \left( 1 - \frac{4k(\eta_{\infty,k}-\epsilon)^2 \log \eta_{\infty,k}-\epsilon}{(2k\eta_{\infty,k}-\epsilon-2-\alpha)
(\eta_{\infty,k}-\epsilon+1)} \right) \theta_0}.
\end{align*}
Therefore,  we choose \(\theta_0\) so small that
\[
\frac{\theta_0}{2(2+\alpha)} \left( 1 - \frac{4k(\eta_{\infty,k}-\epsilon)^2 \log \eta_{\infty,k}-\epsilon}{(2k\eta_{\infty,k}-\epsilon-2-\alpha)
(\eta_{\infty,k}-\epsilon+1)} \right)  \leq \frac{\varepsilon}{4}.
\]
Estimate \eqref{x.10} then follows.  The proof of \eqref{x.11} is very similar and we omit it.
\end{proof}

To estimate the error term, we first establish some key estimates.
For \( x, y \in \Omega \), we denote as \( G(x, y) \) the Green's function for \( -\Delta \) on \( \Omega \) and as \( H(x, y) \) its regular part, that is:
\[
\begin{cases}
-\Delta_x G(x, y) = \delta_y (x), & x \in \Omega, \\[4pt]
G(x, y) = 0, & x \in \partial \Omega,
\end{cases}
\qquad
H(x, y) := G(x, y) + \frac{1}{2\pi} \log |x - y|,
\]
where \( H \) solves
\begin{align}\label{2.016}
\begin{cases}
-\Delta_x H(x, y) = 0, & x \in \Omega \\[4pt]
H(x, y) = \frac{1}{2\pi} \log |x - y|, & x \in \partial \Omega.
\end{cases}
\end{align}

\begin{lemma}\label{lem2.2}
Let $C_1$, $C_2$, $C_3$ and $C_4$ be the constants defined in \eqref{C..i1}. The following estimates  hold uniformly in \( x \in \Omega \):
\begin{align}\label{p.1}
PU_1(x) &= U_{\alpha}\left( \frac{x}{\delta} \right) + 4(2+\alpha)\pi H(0,0) - \log 2(2+\alpha)^{2} - 2(2+\alpha)\log \delta + O\left( |x| + \delta^{2+\alpha} \right),  \\
PV_1(x) &= V_{\alpha}\left( \frac{x}{\delta} \right) - 2C_1\pi H(0,0) + C_1\log \delta + O\left( |x| + \delta \right), \notag\\
PW_1(x) &= W_{\alpha}\left( \frac{x}{\delta} \right) - 2C_2\pi H(0,0) + C_2\log \delta + O\left( |x| + \delta\right),\notag\end{align}
\begin{align}\label{p.2}
PU_2(x) &= U_{0}\left( \frac{x^k - \rho^k}{\beta^k} \right) + 8k\pi H(0,0) - \log8 - 4k\log \beta + O\left( |x^k - \rho^k| + \beta^{2k} + \rho^k \right),\\
PV_2(x) &= V_{0}\left( \frac{x^k - \rho^k}{\beta^k} \right) - 2kC_3\pi H(0,0) + kC_3\log \beta + O\left( |x^k - \rho^k| + \beta^k + \rho^k \right), \notag\\
PW_2(x) &= W_{0}\left( \frac{x^k - \rho^k}{\beta^k} \right) - 2kC_4\pi H(0,0) + kC_4\log \beta + O\left( |x^k - \rho^k| + \beta^k + \rho^k \right),\notag
\end{align}
where \( U_{0}, U_1,U_2, U_{\alpha}, V_{0}, V_1,V_2,V_{\alpha}, W_{0},W_1,W_2,W_{\alpha}\) are the functions defined in \eqref{001.040}, \eqref{1.040}, \eqref{1.40-bis}, \eqref{1..40}, \eqref{1.41}, \eqref{1..41} and \eqref{gy11}.

Moreover, if \( |y| \leq \frac{1}{2} \frac{\rho}{\delta} \), we have
\begin{align}\label{pp.3}
PU_2(\delta y) &= -4k\log \rho + 8k\pi H(0,0) + O\left( \frac{\delta^k}{\rho^k} |y|^k + \frac{\beta^{2k}}{\rho^{2k}} + \rho^k \right),  \\
PV_2(\delta y) &= kC_3\log \rho - 2kC_3\pi H(0,0) + \left( \frac{\delta^k}{\rho^k} |y|^k + \frac{\beta^k}{\rho^k} + \rho^k \right),\notag \\
PW_2(\delta y) &= kC_4\log \rho - 2kC_4\pi H(0,0) + O\left( \frac{\delta^k}{\rho^k} |y|^k + \frac{\beta^k}{\rho^k} + \rho^k \right).\notag
\end{align}
If \( |y| \leq \frac{1}{2} \frac{\rho^k}{\beta^k} \), we have
\begin{align}\label{p.4}
PU_1\left( \sqrt[k]{\beta^k y + \rho^k} \right) &= -2(2+\alpha)\log \rho + 4(2+\alpha)\pi H(0,0) + O\left( \frac{\beta^k}{\rho^k} |y| + \frac{\delta^{2+\alpha}}{\rho^{2+\alpha}} + \rho \right),  \\
PV_1\left( \sqrt[k]{\beta^k y + \rho^k} \right) &= C_1\log \rho - 2C_1\pi H(0,0) + O\left( \frac{\beta^k}{\rho^k} |y| + \frac{\delta}{\rho} + \rho \right),\notag\\
PW_1\left( \sqrt[k]{\beta^k y + \rho^k} \right) &= C_2\log \rho - 2C_2\pi H(0,0) + O\left( \frac{\beta^k}{\rho^k} |y| + \frac{\delta}{\rho} + \rho \right)\notag.
\end{align}
\end{lemma}
\begin{proof}The first estimate in \eqref{p.1} has been proved in Lemma A.3 in \cite{PT}  and the first equation of \eqref{p.2} follows by Lemma 3.4 in \cite{LIA}. So we just need to prove the remaining estimates.

From the definition of $V_1$ in
\eqref{gy11} and the asymptotic behavior of $V_\alpha$ in \eqref{1.40-bis},
we can derive
$$
PV_1(x) - V_1(x) = -V_1(x) = -V_\alpha\left(\frac{x}{\delta}\right) = -C_1 \log |x| + C_1 \log \delta + O(\delta), \quad x \in \partial\Omega,
$$
which, together with, the fact that \( H(x, y) \) is the solution to \eqref{2.016} and the maximum principle, gives
\begin{align*}
PV_1(x) &= V_1(x) - 2\pi C_1H(x, 0) + C_1 \log \delta + O(\delta)
\\&
= V_\alpha \left( \frac{x}{\delta} \right) - 2\pi C_1H(0, 0) + C_1 \log \delta + O(\delta + |x|).
\end{align*}
Since $x=\sqrt[k]{\beta^k y + \rho^k}=\rho\left(1+O(\frac{\beta^k}{\rho^k}
|y|)\right )$ and the asymptotic behavior of $V_\alpha$ in \eqref{1.40-bis}, we can obtain \begin{align*}
PV_1\left( \sqrt[k]{\beta^k y + \rho^k} \right) &= C_1\log \rho - 2C_1\pi H(0,0) + O\left( \frac{\beta^k}{\rho^k} |y| + \frac{\delta}{\rho} + \rho \right), \ \ \ \text{if}\ \ |y| \leq \frac{1}{2} \frac{\rho^k}{\beta^k}.
\end{align*}
Similarly, by utilizing \eqref{1.41}, \eqref{gy11}, \eqref{2.016} and the maximum principle,  we can get
\begin{align*}
PW_1(x) = W_\alpha \left( \frac{x}{\delta} \right) - 2C_2\pi H(0, 0) + C_2 \log \delta + O(|x| + \delta)
\end{align*}
and
\begin{align*}
PW_1\left( \sqrt[k]{\beta^k y + \rho^k} \right) &= C_2\log \rho - 2C_2\pi H(0,0) + O\left( \frac{\beta^k}{\rho^k} |y| + \frac{\delta}{\rho} + \rho \right), \ \ \ \text{if}\ \ |y| \leq \frac{1}{2} \frac{\rho^k}{\beta^k}.
\end{align*}
It follows from the definition of $V_2$ in
\eqref{gy11} and the asymptotic behavior of $V_0$ in \eqref{1..40} that
$$
PV_2(x) - V_2(x) = -V_2(x) = -V_0 \left( \frac{x^k - \rho^k}{\beta^k} \right) = -C_3 \log|x^k - \rho^k| + C_3 \log\beta^k + O(\beta^k), \quad x \in \partial\Omega.
$$Thanks to the fact that \( H(x, y) \) is the solution to \eqref{2.016}  and the maximum principle, we can obtain
\begin{align*}
PV_2(x) &= V_2(x) - 2\pi C_3H(x^k, \rho^k) + C_3 \log\beta^k + O(\beta^k)
\\&
= V_0 \left( \frac{x^k - \rho^k}{\beta^k} \right) - 2k\pi C_3H(0, 0) + C_3 \log\beta^k + O\left(\beta^k + |x^k - \rho^k| + \rho^k\right).
\end{align*}Since the asymptotic behavior of $V_0$ in \eqref{1..40} and $|y| \leq \frac{1}{2} \frac{\rho}{\delta}$, we can have
\begin{align*}
PV_2(\delta y) &= C_3 \log|\delta^k y^k - \rho^k|
- 2k\pi C_3H(0, 0) + O\left(\beta^k + \rho^k+ \frac{\beta^k}{\rho^k}\right)\\&
=C_3k\log \rho - 2k\pi C_3 H(0,0) + \left( \frac{\delta^k}{\rho^k} |y|^k + \frac{\beta^k}{\rho^k} + \rho^k \right).
\end{align*}
From \eqref{1..41}, \eqref{gy11}, \eqref{2.016} and the maximum principle, we can have
\begin{align*}
PW_2(x) = W_0 \left( \frac{x^k - \rho^k}{\beta^k} \right) - 2kC_4\pi H(0, 0) + kC_4 \log\beta^k + O\left(|x^k - \rho^k| + \beta^k + \rho^k\right)\end{align*}
and
 \begin{align*}
 PW_2(\delta y) &= kC_4\log \rho - 2kC_4\pi H(0,0) + O\left( \frac{\delta^k}{\rho^k} |y|^k + \frac{\beta^k}{\rho^k} + \rho^k \right), \ \ \ \text{if}\ \ |y| \leq \frac{1}{2} \frac{\rho}{\delta}.
 \end{align*}
\end{proof}
\subsection*{ The estimate of error close to the positive peak.}
\begin{lemma}\label{lem2.3}
Let $\mathcal{E}$ be as defined in \eqref{1.8}. There exist $\theta\in(0,1)$ and $C>0$ such that
\begin{align*}
|y|\leq\frac{1}{\delta^{\theta}}\qquad\Rightarrow\qquad
|\mathcal{E}( \delta y)|\leq\frac{C|y|^{\alpha}}{\delta^{2}p^{4}}\,
\frac{\log^{6}(|y|+2)}{(|y|^{2+\alpha}+1 )^{2}}
\end{align*}
for $p$ large enough.
\end{lemma}
\begin{proof} Noting that \(|\delta y|\leq\delta^{1-\theta}\), then from \eqref{1.6}, \eqref{p.1} and \eqref{pp.3}, we can deduce
\begin{align*}
\Upsilon(\delta y) &= \tau\Bigl(\operatorname{P}U_{1}(\delta y)
+\frac{\operatorname{P}V_{1}(\delta y)}{p}
+\frac{\operatorname{P}W_{1}(\delta y)}{p^{2}}\Bigr) \\
&\quad - \tau\eta\Bigl(\operatorname{P}U_{2}(\delta y)
+\frac{\operatorname{P}V_{2}(\delta y)}{p}
+\frac{\operatorname{P}W_{2}(\delta y)}{p^{2}}\Bigr) \\
&= \tau\Bigl(U_{\alpha}(y)+\frac{V_{\alpha}(y)
}{p}+\frac{W_{\alpha}(y)}{p^{2}}
+2\pi\left(2(2+\alpha) - \frac{C_1}{p} - \frac{C_2}{p^2}\right)H(0,0)
-\log 2(2+\alpha)^{2} \\
&\qquad -\left(2(2+\alpha) - \frac{C_1}{p} - \frac{C_2}{p^2}\right)\log \delta
 +O\bigl(\delta^{1-\theta}\bigr)\Bigr) \\
  &\quad - \tau\eta\Bigl(\Bigl(-4+\frac{C_{3}}{p}+\frac{C_{4}}{p^{2}}\Bigr)k\log\rho
      +2k\pi\Bigl(4-\frac{C_{3}}{p}-\frac{C_{4}}{p^{2}}\Bigr)H(0,0) \\
        &\qquad
  +O\Bigl(\frac{\delta^{k(1-\theta)}}{\rho^{k}}
   +\frac{1}{p}\frac{\beta^{k}}{\rho^{k}}
   +\frac{\beta^{2k}}{\rho^{2k}}
   +\rho^{k}\Bigr)\Bigr).
\end{align*}
 Lemma \ref{lem2.1} and \eqref{1.11} reveal that these remainders decay exponentially fast with respect to $p$. If $\theta$ is small enough, then we have
\begin{align}\label{2.18}
\Upsilon(\delta y)=\tau p\Bigl(1+\frac{U_{\alpha}(y)}{p}
+\frac{V_{\alpha}(y)}{p^{2}}+\frac{W_{\alpha}(y)}{p^{3}}+O\bigl(e^{-\varepsilon p}\bigr)\Bigr).
\end{align}
It follows from \eqref{pe.1}, the explicit expression \eqref{1.040} for \(U_{\alpha}\) and the asymptotic behaviors \eqref{1.40-bis} and \eqref{1.41} for $V_{\alpha}(y)$ and $W_{\alpha}(y)$ that
\begin{align}\label{2.19}
&\Bigl(1+\frac{U_{\alpha}(y)}{p}+\frac{V_{\alpha}
(y)}{p^{2}}+\frac{W_{\alpha}(y)}{p^{3}}
\Bigr)^{p}\notag
\\&= e^{U_{\alpha}(y)}\Bigl(1+\frac{1}{p}\Bigl(
V_{\alpha}(y)-\frac{U_{\alpha}(y)^{2}}{2}\Bigr) +\frac{1}{p^{2}}\Bigl(W_{\alpha}(y)-U_{\alpha}
(y)V_{\alpha}(y)+\frac{U_{\alpha}(y)^{3}}{3}
           +\frac{1}{2}\Bigl(V_{\alpha}
(y)-\frac{U_{\alpha}(y)^{2}}{2}\Bigr)^{2}\Bigr)\Bigr) \notag\\&\qquad + O\Bigl(\frac{1}{(|y|^{2+\alpha}+1)^{2}}\,
            \frac{\log^{6}(|y|^{2+\alpha}+2)}{p^{3}}\Bigr).
\end{align}
With the aid of \eqref{1.6}, \eqref{1.11}, \eqref{2.18} and \eqref{2.19}, we can get
\begin{align*}
&\Delta\Upsilon_{1}(\delta y)+|\delta y|^{\alpha}|\Upsilon(\delta y)|^{p-1}\Upsilon(\delta y) \\
&= \frac{\tau}{\delta^{2}}\Bigl(\Delta U_{\alpha}(y)
   +\frac{1}{p}\Delta V_{\alpha}(y)+\frac{1}{p^{2}}\Delta W_{\alpha}(y)\Bigr) \\
&\quad + |\delta y|^{\alpha}(\tau p)^{p}\Bigl(1+\frac{U_{\alpha}(y)}{p}
   +\frac{V_{\alpha}(y)}{p^{2}}+
\frac{W_{\alpha}(y)}{p^{3}}+O\bigl(e^{-\varepsilon p}\bigr)\Bigr)^{p} \\
&= \frac{\tau}{\delta^{2}}| y|^{\alpha}\Biggr\{- e^{U_{\alpha}(y)}\Bigl(1
   +\frac{1}{p}\Bigl(V_{\alpha}(y)-
\frac{U_{\alpha}(y)^{2}}{2}\Bigr)\Bigr.
\\
&\quad+\frac{1}{p^{2}}\Bigl(W_{\alpha}(y)
-U_{\alpha}(y)V_{\alpha}(y)+\frac{U_{\alpha}(y)^{3}}{3}
+\frac{1}{2}\bigl(V_{\alpha}(y)-\frac{U_{\alpha}(y)^{2}}{2}
\bigr)^{2}\Bigr)\Bigl) \\
&\quad+\Bigl(1+\frac{U_{\alpha}(y)}{p}
+\frac{V_{\alpha}(y)}{p^{2}}+\frac{W(y)}{p^{3}}
+O\Bigl(\frac{e^{-\varepsilon p}}{p}\Bigr)\Bigr)^{p}\Biggr\} \\
&= O\Bigl(\frac{\tau| y|^{\alpha}}{\delta^{2}p^{3}}
\frac{\log^{6}(|y|^{2+\alpha}+2)}{(|y|^{2+\alpha}+1)^{2}}\Bigr).
\end{align*}
Since $PU_{2}$, $PV_{2}$, $PW_{2}$ satisfy \eqref{01.0040}, \eqref{VY2} and  \eqref{WY2}, $U_{0}$, $V_{0}$, $W_{0}$ satisfy \eqref{001.040}, \eqref{1..40} and \eqref{1..41} and the parameters $\delta$, $\beta$ and $\rho$ satisfy Lemma \ref{lem2.x}, we can deduce
\begin{align*}
\Delta\Upsilon_{2}(\delta y) &=
O\Bigl(\frac{\tau\delta^{2k-2}|y|^{2k-2}
}
{\beta^{2k}}
e^{U}\bigl(1+U_{0}^{4}+V_{0}^{2}+|W_{0}|\bigr)
\Bigl(\frac{\delta^{k}y^{k}-\rho^{k}}{\beta^{k}}\Bigr)\Bigr) \\
&= O\Bigl(\tau\frac{\delta^{(1-\theta)(2k-2)}}{\beta^{2k}}
\frac{\log^{4}\frac{\rho}{\beta}}{\bigl(\frac{\rho}{\beta}\bigr)^{4k}}\Bigr) \\
&= O\Bigl(\tau
\frac{\rho^{\frac{4k\eta}{2+\alpha}(1-\theta)(2k-2)}}
{\rho^{2k}}\frac{\beta^{2k}}{\rho^{2k}}
\log^{4}\frac{\rho}{\beta}\Bigr) \\
&= O\bigl(\tau e^{-\varepsilon p}\bigr),
\end{align*}
where $\theta$ is a small enough constant. Then, we can conclude
\begin{align*}
\mathcal{E}(\delta y) &= \Delta\Upsilon_{1}(\delta y)
+ |\delta y|^{\alpha}|\Upsilon(\delta y)|^{p-1}\Upsilon(\delta y) + \Delta\Upsilon_{2}(\delta y) \\
&= O\Bigl(\frac{|y|^{\alpha}}{\delta^{2}p^{4}}
\frac{\log^{6}(|y|^{2+\alpha}+2)}{(|y|^{2+\alpha}+1)^{2}}\Bigr).
\end{align*}
\end{proof}
\subsection*{ The estimate of error close to the negative peaks.}
\begin{lemma}\label{lem2.4}
There exist $\theta\in(0,1)$ and $C>0$ such that
\[
|y|\leq\frac{1}{\beta^{k\theta}}\qquad
\Longrightarrow\qquad
\Bigl|\mathcal{E}\left(\sqrt[k]{\beta^{k}y+\rho^{k}}\right)\Bigr|
\le C\frac{\rho^{2k-2}\log^{6}(|y|+2)}
            {\beta^{2k}p^{4}\,
            (|y|^{2}+1)^{2}},
\]
for $p$ large enough, where $\mathcal{E}$ is defined in \eqref{1.8}.
\end{lemma}
\begin{proof}
The argument is similar to Lemma \ref{lem2.3}.
According to \eqref{1.6}, \eqref{p.2}, \eqref{p.4}, Lemma \ref{lem2.1} and \eqref{1.11}, we can get
\begin{align}\label{2.108}
\Upsilon\left(\sqrt[k]{\beta^{k}y+\rho^{k}}\right)
&= \tau\Bigl(\mathrm{P}U_{1}\!\left(\sqrt[k]{\beta^{k}y+\rho^{k}}\right)
   +\frac{\mathrm{P}V_{1}\!\left(\sqrt[k]{\beta^{k}y+\rho^{k}}\right)}{p}
   +\frac{\mathrm{P}W_{1}\!\left(\sqrt[k]{\beta^{k}y+\rho^{k}}\right)}{p^{2}}\Bigr) \notag \\
&\quad - \tau\eta\Bigl(\mathrm{P}U_{2}\!\left(\sqrt[k]{\beta^{k}y+\rho^{k}}\right)
   +\frac{\mathrm{P}V_{2}\!\left(\sqrt[k]{\beta^{k}y+\rho^{k}}\right)}{p}
   +\frac{\mathrm{P}W_{2}\!\left(\sqrt[k]{\beta^{k}y+\rho^{k}}\right)}{p^{2}}\Bigr) \notag\\
&= \tau\Biggl(\Bigl(-2(2+\alpha)+\frac{\mathcal{C}_{1}}{p}+\frac{\mathcal{C}_{2}}{p^{2}}\Bigr)\log\rho
   +2\pi\Bigl(2(2+\alpha)-\frac{\mathcal{C}_{1}}{p}-\frac{\mathcal{C}_{2}}{p^{2}}\Bigr)H(0,0) \notag \\
&\qquad +O\Bigl(\frac{\beta^{k(1-\theta)}}{\rho^{k}}
+\frac{1}{p}\frac{\delta}{\rho}
+ \frac{\delta^{2+\alpha}}{\rho^{2+\alpha}}
+\rho\Bigr)\Biggr)
 \notag\\
&\qquad -\tau\eta\Biggl(U_{0}(y)+\frac{V_{0}(y)}{p}+
\frac{W_{0}(y)}{p^2}
+2k\pi\Bigl(4-\frac{C_3}{p}-\frac{C_4}{p^2}\Bigr)H(0,0)-\log 8\notag \\
&\qquad +\Bigl(-4+\frac{C_3}{p}+\frac{C_4}{p^2}\Bigr)k\log\beta
+O\Bigl(\beta^{k(1-\theta)}+\rho^k\Bigr) \Biggr) \notag\\
&= - p\tau\eta\Bigl(1+\frac{U_{0}(y)}{p}+\frac{V_{0}
(y)}{p^2}+\frac{W_{0}(y)}{p^3}
+O\bigl(e^{-\frac{\varepsilon}{2}p}\bigr)
\Bigr),
\end{align}where \(\theta\) is small.
It follows from Lemma \ref{pe.0}, \eqref{2.108} and  \eqref{1.11} that
\begin{align*}
&\Delta \Upsilon_2\Bigl(\sqrt[k]{\beta^ky+\rho^k}\Bigr)
+\Bigl|\sqrt[k]{\beta^ky+\rho^k}\Bigr)\Bigr|^{\alpha}
\Bigl|\Upsilon\Bigl(\sqrt[k]{\beta^ky+
\rho^k}\Bigl)\Bigr|^{p-1}
\Upsilon\Bigl(\sqrt[k]{\beta^ky+\rho^k}\Bigr) \\
&= -\frac{\tau \eta k^2\bigl|\sqrt[k]{\beta^ky+\rho^k}\bigr
|^{2k-2}}{\beta^{2k}}
\Bigl(\Delta U_{0}(y)+\frac{\Delta V_{0}(y)}{p}+\frac{\Delta W_{0}(y)}{p^2}\Bigr) \\
&\quad - (p\tau\eta)^p\Bigl|\sqrt[k]{\beta^ky+\rho^k}
\Bigr)\Bigr|^{\alpha}
\Bigl(1+\frac{U_{0}(y)}{p}+\frac{V_{0}(y)}{p^2}
+\frac{W_{0}(y)}{p^3}
+O\bigl(e^{-\varepsilon p}\bigr)\Bigr)^p \\
&=\frac{\tau \eta k^2\bigl|\sqrt[k]{\beta^ky+\rho^k}\bigr
|^{2k-2}}{\beta^{2k}}e^{U_{0}}
\Bigg(1+\frac{1}{p}\Bigl(V_{0}(y)-\frac{U_{0}(y)^{2}}{2}\Bigr) \\&\quad+\frac{1}{p^{2}}\Bigl(W_{0}(y)-U_{0}(y)V_{0}(y)+
\frac{U_{0}(y)^{3}}{3}
 +\frac{1}{2}\Bigl(V_{0}(y)-\frac{U_{0}(y)^{2}}{2}\Bigr)^{2}\Bigr)\Bigg) \\&\quad -(p\tau\eta)^p
\Bigl|\sqrt[k]{\beta^ky+\rho^k}\Bigr)
\Bigr|^{\alpha}e^{U_{0}}\Bigg
 (1+\frac{1}{p}\Bigl(V_{0}(y)-\frac{U_{0}(y)^{2}}{2}
 \Bigr)\\
&\quad  +\frac{1}{p^{2}}\Bigl(W_{0}(y)-U_{0}(y)V_{0}(y)+
 \frac{U_{0}(y)^{3}}{3}
           +\frac{1}{2}\Bigl(V_{0}(y)-\frac{U_{0}
           (y)^{2}}{2}\Bigr)^{2}\Bigr)\Bigg) \\
&\quad + O\Bigl(\frac{(p\tau\eta)^p
           \rho^{\alpha}
}{(|y|^{2}+1)^{2}}\,
            \frac{\log^{6}(|y|+2)}{p^{3}}\Bigr)
\\
&= O\Bigl(\frac{\tau\rho^{2k-2}\log^6(|y|+2)}
{\beta^{2k}p^3}\frac{1}{(|y|^2+1)^2}+e^{-\varepsilon p}\Bigr) \\
&= O\Bigl(\frac{\rho^{2k-2}\log^6(|y|+2)}
{\beta^{2k}p^4}\frac{1}{(|y|^2+1)^2}+e^{-\varepsilon p}\Bigr).
\end{align*}
Since $PU_{1}$, $PV_{1}$, $PW_{1}$ satisfy \eqref{01.040}, \eqref{VY1} and  \eqref{WY1}, $U_{\alpha}$, $V_{\alpha}$, $W_{\alpha}$ satisfy \eqref{1.040}, \eqref{1.40-bis} and \eqref{1.41} and parameters $\delta$, $\beta$ and $\rho$ satisfy Lemma \ref{lem2.x}, we can obtain
\begin{align*}
\Delta \Upsilon_1\Bigl(\sqrt[k]{\beta^ky+\rho^k}\Bigr)
&= O\Bigl(\frac{\tau \rho^{\alpha} }{\delta^{2+\alpha}}e^{U_{\alpha}}\bigl(1+U_{\alpha}^4
+V_{\alpha}^2+|W_{\alpha}|\bigr)
\Bigl(\frac{\sqrt[k]{\beta^ky+\rho^k}}{\delta}\Bigr)\Bigr) \\
&= O\Bigl(\frac{\tau \rho^{\alpha} }{\delta^{2+\alpha}}\frac{\log^4\frac{\rho}{\delta}}
{\bigl(\frac{\rho}{\delta}\bigr)^{4+2\alpha}}\Bigr) 
= O\Bigl(\frac{\delta^{2+\alpha-\sigma}}
{\rho^{4+\alpha-\sigma}}\Bigr) \\
&= O\Bigl(\frac{\delta^{2+\alpha}}
{\rho^{\frac{8k\eta}{2+\alpha}+\alpha}}\Bigr) 
= O\bigl(e^{-2\varepsilon p}\bigr),
\end{align*}
where \(\sigma > 0\) satisfies \(\frac{\delta^\sigma}{\rho^{4\big(\frac{2k\eta}{
2+\alpha}-1\big)+\sigma}} \geq 1\).

\end{proof}
\subsection*{ The estimates of the error far away from the peaks.}
\begin{lemma}\label{lem2.5}
Assume $\mathcal{E}$ is as defined in \eqref{1.8}. There exist $\lambda \in (0,1)$ and $C > 0$ such that
\begin{align*}
|x| > \delta,\; |x^k - \rho^k| > \beta^k \quad \Longrightarrow \quad
|\mathcal{E}(x)| \leq \frac{C}{p}
\Bigl(\frac{\delta^{2(2+\alpha)\lambda}
|x|^{\alpha}}{|x|^{2+\alpha+2(2+\alpha)\lambda}}
+\frac{|x|^{2k-2}\beta^{2(2+\alpha)\lambda k}}{|x^k - \rho^k|^{2+2(2+\alpha)\lambda}}\Bigr)
\end{align*}
for $p$ large enough.
\end{lemma}
\begin{proof}
From \eqref{1.6} and \eqref{1.8}, we have \(\mathcal{E} = \Delta \Upsilon_1 +\Delta \Upsilon_2 + |x|^{\alpha}|\Upsilon|^{p-1}\Upsilon\). In the following, we estimate separately each term. In view of the logarithmic behavior of \(U_{\alpha},V_{\alpha},W_{\alpha}\) in \eqref{1.040}, \eqref{1.40-bis} and \eqref{1.41}, we have
\[
\begin{aligned}
\Delta\Upsilon_{1}(x) &= O\Bigl(\frac{\tau|x|^{\alpha}}{\delta^{2+\alpha}}
e^{U_{\alpha}}
\bigl(1+U_{\alpha}^{4}+V_{\alpha}^{2}+|W_{\alpha}|\bigr)\Bigl(\frac{x}{\delta}\Bigr)\Bigr) \\
&= O\Bigl(\frac{\tau|x|^{\alpha}}
{\delta^{2+\alpha}}\frac{\log^{4}\bigl|\frac{x}{\delta}\bigr|}
{\bigl|\frac{x}{\delta}\bigr|^{4+2\alpha}}\Bigr) 
= O\Bigl(\frac{1}{p}\frac{\delta^{2+\alpha-t
}}{|x|^{4+\alpha-t
}}\Bigr),
\end{aligned}
\]
where $t>0$ is a small enough constant.
Similarly, by \eqref{001.040}, \eqref{1..40} and \eqref{1..41}, we can get
\begin{align*}
\Delta \Upsilon_2(x) &= O\left(\frac{\tau|x|^{2k-2}}{\beta^{2k} } e^{U_{0}}
\bigl(1 + U_{0}^4 + V_{0}^2 + |W_{0}| \bigr) \left(\frac{x^k - \rho^k}{\beta^k}\right)\right) \\
&= O\left(\frac{\tau |x|^{2k-2} \log^4
\bigl|\frac{x^k - \rho^k}{\beta^k}\bigr|}
{\beta^{2k} \bigl|\frac{x^k - \rho^k}{\beta^k}\bigr|^4}\right) \\
&= O\left(\frac{1}{p}
\frac{\beta^{(2-t)k} |x|^{2k-2}}{|x^k - \rho^k|^{4-t}}\right).
\end{align*}
For the nonlinear term, by the definition of $\Upsilon$ in \eqref{1.6} and  the relations \eqref{1.11} between the parameters, we can get
\begin{align}\label{00agn0}
\Upsilon(x) &=
\tau \left( PU_1 + \frac{PV_1}{p} + \frac{PW_1}{p^2} \right) - \tau \eta \left( PU_2 + \frac{PV_2}{p} + \frac{PW_2}{p^2} \right)\notag
\\
&=
\tau\Bigl(p+U_{\alpha}\Bigl(\frac{x}{\delta}\Bigr)
+\frac{V_{\alpha}\bigl(\frac{x}{\delta}\bigr)}{p}
+\frac{W_{\alpha}\bigl(\frac{x}{\delta}\bigr)}{p^{2}}
+4\eta\log\Bigl|\frac{x^{k}}{\rho^{k}}-1\Bigr| \notag\\
&\qquad +O(1) +O\Bigl(|x|+\delta+\frac{\beta^{k}}{|x^{k}-
\rho^{k}|}+\rho^{k}\Bigr)\Bigr) \notag\\
&=\tau\Bigl(p+\Bigl(-2(2+\alpha)+O\big(\frac{1}
{p}\big)\Bigr)\log\Bigl|\frac{x}{\delta}\Bigr|
+4\eta\log\Bigl|\frac{x^{k}}{\rho^{k}}-1\Bigr| +O(1)\Bigr)
\end{align}
and \begin{align}\label{agn0}
\Upsilon(x) &= -\tau\eta\Bigl(p+U_{0}\Bigl(\frac{x^{k}-\rho^{k}}{\beta^{k}}\Bigr)
+\frac{V_{0}\bigl(\frac{x^{k}-\rho^{k}}{\beta^{k}}\bigr)}{p}
+\frac{W_{0}\bigl(\frac{x^{k}-\rho^{k}}{\beta^{k}}\bigr)}{p} \notag\\
&\qquad +\frac{2(2+\alpha)}{\eta}\log\Bigl|\frac{x}{\rho}\Bigr|
+O\Bigl(|x|+\frac{\alpha}{|x|}+\beta^{k}+
\rho^{k}\Bigr)+O(1)\Bigr)\notag
\\
&= -\tau\eta\Biggl(p-\left(4+O\big(\frac{1}
{p}\big)\right)\log\Bigl|\frac{x^{k}-\rho^{k}}{\beta^{k}}\Bigr|
+\frac{2(2+\alpha)}{\eta p}\log\Bigl|\frac{x}{\rho}\Bigr|
+O(1)\Biggr). \end{align}
Then, if $\Upsilon(x)>0$, from \eqref{00agn0}, we can have
\begin{align*}
\Upsilon(x)
&\leq \tau\Bigg(p+\Bigl(-2(2+\alpha)+O\big(\frac{1}
{p}\big)\Bigr)
\log\Bigl|\frac{x}{\delta}\Bigr|
+4k\eta\log\Bigl(\frac{|x|}{\rho}+1\Bigr)\Bigg)
+O\big(\frac{1}{p}\big) \\
&\leq p\tau\Bigl(1-\frac{2(2+\alpha)-\sigma}{p}\log\Bigl|\frac{x}
{\delta}\Bigr|+\frac{4k\eta}{p}\log
\Bigl(\frac{|x|}{\rho}+1\Bigr)
+\frac{C}{p}\Bigr),
\end{align*}
for any small \(\sigma>0\). In view of \eqref{x.2}, we have $4k\eta-2-\alpha+\sigma+2(2+\alpha)
\lambda>0$, which, together with
 the elementary inequality \(\bigl(1+\frac{x}{p}\bigr)^{p}\leq C e^{x}\) and $|x| > \delta$, implies
\begin{align}\label{al0}
|\Upsilon(x)|^{p} &\leq C (p\tau)^{p}\Bigl(1-\frac{2(2+\alpha)-\sigma}{p}\log\Bigl|
\frac{x}{\delta}\Bigr|+\frac{4k\eta}{p}\log
\Bigl(\frac{|x|}{\rho}+1\Bigr)
+\frac{C}{p}\Bigl)^p\notag \\
&\leq C(p\tau)^{p}\frac{\delta^{2(2+\alpha)-\sigma}}
{|x|^{2(2+\alpha)-\sigma}}
\Bigl(\frac{|x|}{\rho}+1\Bigr)^{4k\eta}\notag \\
&\leq \frac{C}{p}\frac{\delta^{2+\alpha-\sigma}}
{|x|^{2(2+\alpha)-\sigma}}
\Bigl(
\frac{|x|^{4k\eta}}{\rho^{4k\eta}}
+1\Bigr)\notag
\\
&\leq\frac{C}{p}\frac{\delta^{2(2+\alpha)\lambda}
}{|x|^{2+\alpha+2(2+\alpha)\lambda}}
\Bigl(
|x|^{4k\eta-2-\alpha+\sigma+2(2+\alpha)
\lambda}
\frac{\delta^{2+\alpha-\sigma-2(2+\alpha)\lambda}
}{\rho^{4k\eta}}+\frac{\delta^{2+\alpha-\sigma-2(2+\alpha)\lambda}
}{|x|^{2+\alpha-\sigma-
2(2+\alpha)
\lambda
}}\Bigr)
\notag \\
&\leq \frac{C}{p}\frac{\delta^{2(2+\alpha)\lambda}
}{|x|^{2+\alpha+2(2+\alpha)\lambda}},
\end{align}
since $\frac{\delta^{2+\alpha-\sigma-2(2+\alpha)\lambda}
}{\rho^{4k\eta}}\leq1$ from \eqref{x.9} and $2+\alpha-\sigma-
2(2+\alpha)\lambda>0$.                        \\
If $\Upsilon(x)<0$, by \eqref{agn0}, we have
\begin{align*}
\Upsilon(x)
\geq -p\tau\eta\Biggl(1-\frac{4-\sigma}{p}\log\Bigl|\frac{x^{k}-\rho^{k}}{\beta^{k}}\Bigr|
+\frac{2(2+\alpha)}{\eta p}\log\Bigl|\frac{x}{\rho}\Bigr|
+\frac{C}{p}\Biggr). \end{align*}
From \eqref{x.2} we have $2-\sigma-2(2+\alpha)\lambda-
\frac{2(2+\alpha)}{k\eta }>0$  for small $\sigma$ and $\lambda$. Then if $|x|\geq \frac{1}{2}\rho$, we can obtain
\begin{align}\label{a0l1}
|\Upsilon(x)|^{p} &\leq
(p\tau\eta)^{p}\Biggl(1-\frac{4-\sigma}{p}
\log\Bigl|\frac{x^{k}-\rho^{k}}
{\beta^{k}}\Bigr|
+\frac{2(2+\alpha)}{\eta pk}\log\left(\Bigl|\frac{x^{k}-\rho^{k}}
{\rho^{k}}\Bigr|+1\right)
+\frac{C}{p}\Biggr)^{p}\notag \\
&\leq \frac{C}{p}
\frac{\rho^{2k-2-\alpha}}{\beta^{2k}}
\frac{\beta^{(4-\sigma)k}}{|x^{k}-
\rho^{k}|^{4-\sigma}}
\left(\Bigl|\frac{x^{k}-\rho^{k}}
{\rho^{k}}\Bigr|+1\right)^{\frac{2(2+
\alpha)}{k\eta }}\notag
\\
&\leq \frac{C}{p}\frac{|x|^{2k-2-\alpha}\beta^{2(2+\alpha)\lambda k}}{|x^k - \rho^k|^{2+2(2+\alpha)\lambda}}
\frac{\beta^{(2-\sigma-
2(2+\alpha)\lambda
)k}}{|x^{k}-
\rho^{k}|^{2-\sigma-2(2+\alpha)\lambda }}\frac{\rho^{2k-2-\alpha}}{|x|^{2k-2-\alpha}}
\left(
\frac{|x^{k}-\rho^{k}|^{\frac{2(2+\alpha)}{
k\eta }}}{\rho^{\frac{2(2+\alpha)}{\eta }}}
+
1
\right)\notag\\
&\leq \frac{C}{p}\frac{|x|^{2k-2-\alpha}\beta^{2(2+\alpha)\lambda k}}{|x^k - \rho^k|^{2+2(2+\alpha)\lambda}}
\left(\frac{\beta^{(2-\frac{2(2+\alpha)}{\eta }-\sigma-
2(2+\alpha)\lambda
)k}}{|x^{k}-
\rho^{k}|^{2-\sigma-2(2+\alpha)\lambda-
\frac{2(2+\alpha)}{k\eta }
 }}
+
1
\right)
\notag\\
&\leq \frac{C}{p}\frac{|x|^{2k-2-\alpha}\beta^{2(2+\alpha)\lambda k}}{|x^k - \rho^k|^{2+2(2+\alpha)\lambda}}
.
\end{align}
From Lemma \ref{lem2.1}, we have  $\frac{\beta^{k}}{
\rho^{k}}\leq Ce^{-\frac{p}{2}\big(\frac{1}{2} + \frac{\log \eta}{\eta+1}\big)}$, then
if $|x|< \frac{1}{2}\rho$,  we have
\begin{align*}
|\Upsilon(x)|^{p} &\leq (p\tau\eta)^{p}\Biggl(1-\frac{4-\sigma}{p}\log\Bigl|\frac{x^{k}-\rho^{k}}{\beta^{k}}
\Bigr|+\frac{2(2+\alpha)}{\eta p}\log\Bigl|\frac{x}{\rho}\Bigr|
+\frac{C}{p}\Biggr)^{p} \\
&\leq \frac{C}{p}
\frac{\rho^{2k-2-\alpha}}{\beta^{2k}}
\frac{\beta^{(4-\sigma)k}}{|x^{k}-
\rho^{k}|^{4-\sigma}}
\frac{|x|^{\frac{2(2+\alpha)}{\eta }}}{\rho^{\frac{2(2+\alpha)}{\eta }}}
\\
&\leq \frac{C}{p}\frac{|x|^{2k-2-\alpha}\beta^{2(2+\alpha)\lambda k}}{|x^k - \rho^k|^{2+2(2+\alpha)\lambda}}
\frac{\beta^{(2-\sigma-
2(2+\alpha)\lambda
)k}}{|x^{k}-
\rho^{k}|^{2-\sigma-2(2+\alpha)\lambda }}\frac{\rho^{2k-2-\alpha}}{|x|^{2k-2-\alpha}}
\frac{|x|^{\frac{2(2+\alpha)}{\eta }}}{\rho^{\frac{2(2+\alpha)}{\eta }}}
\\
&\leq \frac{C}{p}\frac{|x|^{2k-2-\alpha}\beta^{2(2+\alpha)\lambda k}}{|x^k - \rho^k|^{2+2(2+\alpha)\lambda}}
e^{-\frac{(2-\sigma-
2(2+\alpha)\lambda
)}{2}\big(\frac{1}{2} + \frac{\log \eta}{\eta+1}\big)p}
\frac{\rho^{2k-2-\alpha-\frac{2(2+\alpha)}{\eta }}}
{|x|^{2k-2-\alpha-\frac{2(2+\alpha)}{\eta }}}.
\end{align*}
If $2k-2-\alpha-\frac{2(2+\alpha)}{\eta }\leq0$, by \eqref{x.3} and $|x|< \frac{1}{2}\rho$, we have
\begin{align}\label{a0l01}
|\Upsilon(x)|^{p} \leq \frac{C}{p}\frac{|x|^{2k-2-\alpha}\beta^{2(2+\alpha)\lambda k}}{|x^k - \rho^k|^{2+2(2+\alpha)\lambda}}
.
\end{align}
From Lemma \ref{lem2.1}, we have $\frac{\rho}
{\delta}\leq C
e^{\frac{ p}{2+\alpha}\big(\frac{1}{2} - \frac{\eta\log \eta}{\eta+1}\big)
}$, then if $2k-2-\alpha-\frac{2(2+\alpha)}{\eta }>0$, by \eqref{n1} and \eqref{x.1}, we have
\begin{align}\label{al1}
|\Upsilon(x)|^{p} &\leq \frac{C}{p}\frac{|x|^{2k-2-\alpha}\beta^{2(2+\alpha)\lambda k}}{|x^k - \rho^k|^{2+2(2+\alpha)\lambda}}
e^{-\frac{(2-\sigma-
2(2+\alpha)\lambda
)}{2}\big(\frac{1}{2} + \frac{\log \eta}{\eta+1}\big)p}
e^{\big( 1- \frac{2\eta\log \eta}{\eta+1}\big)\big(\frac{1}{\eta_{\infty,k}}-\frac{1}{\eta}
\big)p
}
\notag\\
&\leq \frac{C}{p}\frac{|x|^{2k-2-\alpha}\beta^{2(2+\alpha)\lambda k}}{|x^k - \rho^k|^{2+2(2+\alpha)\lambda}}
.
\end{align}
Combining \eqref{al0}, \eqref{a0l1}, \eqref{a0l01} with \eqref{al1}, we can get
\begin{align*}
|\Upsilon(x)|^{p} \leq \frac{C}{p}
\frac{\delta^{\sigma-2(2+\alpha)\lambda}
}{|x|^{\sigma-2(2+\alpha)\lambda}}
\frac{\delta^{2(2+\alpha)\lambda}
}{|x|^{2+\alpha+2(2+\alpha)\lambda}}+
\frac{C}{p}\frac{|x|^{2k-2-\alpha}\beta^{2(2+\alpha)\lambda k}}{|x^k - \rho^k|^{2+2(2+\alpha)\lambda}}.
\end{align*}
Above all, we can conclude that
\[|x|^{\alpha}
|\Upsilon(x)|^{p}\leq \frac{C}{p}
\Bigl(\frac{\delta^{2(2+\alpha)\lambda}
|x|^{\alpha}}{|x|^{2+\alpha+2(2+\alpha)\lambda}}
+\frac{|x|^{2k-2}\beta^{2(2+\alpha)\lambda k}}{|x^k - \rho^k|^{2+2(2+\alpha)\lambda}}\Bigr).
\]
\end{proof}

\subsection*{The size of the error in a suitable norm}

For \(h\in L^{\infty}(\Omega)\), we consider the weighted norm
\begin{align}\label{h1}
\|h\|_{*}:=\sup_{x\in\Omega}
\left|
\frac{h(x)}
{\displaystyle\frac{\delta^{(2+\alpha)\lambda}
|x|^{\alpha}}{(|x|^{2+\alpha}+\delta^{2+\alpha
})^{1+\lambda}}
+\frac{\beta^{(2+\alpha)k\lambda}|x|^{2k-2}}
{(|x^{k}-\rho^{k}|^{2}+\beta^{2k})^{1+
\frac{2+\alpha}{2}\lambda}}}
\right|,
\end{align}
where \(\lambda\in(0,1)\) is given by Lemma \ref{lem2.5}.
 Since
\begin{align*}
|x|^{\alpha}e^{U_{1}(x)}=\frac{2(2+\alpha
)^{2}\delta^{2+\alpha}|x|^{\alpha}}{(|x|^{2+\alpha} + \delta^{2+\alpha})^{2}}\leq\frac{2(2+\alpha
)^{2}\delta^{(2+\alpha)\lambda}|x|^{\alpha}}{
(|x|^{2+\alpha} + \delta^{2+\alpha})^{1+\lambda}}
\end{align*}
and\begin{align*}|x|^{2k-2}e^{U_{2}(x)}=\frac{8k^2
\beta^{2k}|x|^{2k-2}}{(|x^k - \rho^k|^{2} + \beta^{2k})^2}
\leq\frac{8k^2
\beta^{(2+\alpha)\lambda k}|x|^{2k-2}}{(|x^k - \rho^k|^{2} + \beta^{2k})^{1+\frac{2+\alpha}{2}\lambda}},
\end{align*}
then we can deduce
\begin{align}\label{h2}
\bigl\||x|^{\alpha}e^{U_{1}(x)}\bigr\|_{*}+\bigl\||x|^{2k-2}e^{U_{2}(x)}\bigr\|_{*}\le C.
\end{align}
It's easy to check that for any \(h\in L^{\infty}(\Omega)\) one has
\begin{align}\label{h3}
\int_{\Omega}|h|\le\int_{\Omega}\|h\|_{*}
\Bigl(\displaystyle\frac{\delta^{(2+\alpha)\lambda}
|x|^{\alpha}}{(|x|^{2+\alpha}+\delta^{2+\alpha
})^{1+\lambda}}
+\frac{\beta^{(2+\alpha)k\lambda}|x|^{2k-2}}
{(|x^{k}-\rho^{k}|^{2}+\beta^{2k})^{1+
\frac{2+\alpha}{2}\lambda}}\Bigr)\,
\mathrm{d}x\le C\|h\|_{*},\end{align}
which is crucial in subsequent calculations.
\begin{proposition}\label{lem2.6}
There exist \(C\) such that
\begin{align*}
\|\mathcal{E}\|_{*}\le\frac{C}{p^{4}}, \end{align*} 
for $p$ large enough, where $\mathcal{E}= \Delta\Upsilon+|x|^{\alpha}|\Upsilon
|^{p-1}\Upsilon$ and $\Upsilon$ is defined in \eqref{1.6}.
\end{proposition}

\begin{proof}

If \(\bigl|\frac{x}{\delta}\bigr|\le\frac{1
}{\delta^{\theta}}\), then, by Lemma \ref{lem2.3}, we have
\begin{align}\label{h4}
|\mathcal{E}(x)|\le\frac{C|\frac{x}{\delta}|^{\alpha}}{\delta^{2}p^{4}}\,
\frac{\log^{6}(|\frac{x}{\delta}|^{2+
\alpha}+2)}{(|\frac{x}{\delta}|^{2+\alpha}+1 )^{2}}
\le\frac{C|x|^{\alpha}}{\delta^{2+\alpha}p^{4}}
\frac{\bigl(|\frac{x}{\delta}|^{2+\alpha}
+1\bigr)^{1-\lambda}}
{\bigl(|\frac{x}{\delta}|^{2+\alpha}+1\bigr)^{2}}
=\frac{C}{p^{4}}\displaystyle\frac{|x|^{\alpha}\delta^{(2+\alpha)\lambda}
}{(|x|^{2+\alpha}+\delta^{2+\alpha
})^{1+\lambda}}.
\end{align}
Similarly, if \(\bigl|\frac{x^{k}-\rho^{k}}{\beta^{k}}\bigr|\le\frac{1}{\beta^{k\theta}}\), then we have $\frac{1}{2}\rho\leq|x|\leq2\rho$, which, together with Lemma \ref{lem2.4}, yields
\begin{align}\label{h5}
|\mathcal{E}(x)|&\le C\frac{\rho^{2k-2}\log^{6}(|\frac{x^{k}-\rho^{k}}{\beta^{k}}
|+2)}
            {\beta^{2k}p^{4}\,(|\frac{x^{k}-
            \rho^{k}}{\beta^{k}}|^{2}+1)^{2}}
\le\frac{C\rho^{2k-2}}{p^{4}\beta^{2k}}
\frac{(|\frac{x^{k}-
            \rho^{k}}{\beta^{k}}|^{2}+1
            )^{1-\frac{2+\alpha}{2}\lambda}}
{(|\frac{x^{k}-
            \rho^{k}}{\beta^{k}}|^{2}+1)^{2}}\notag\\
&\leq \frac{C}{p^{4}}
\frac{\beta^{(2+\alpha)k\lambda}|x|^{2k-2}}
{(|x^{k}-\rho^{k}|^{2}+\beta^{2k})^{1+
\frac{2+\alpha}{2}\lambda}}.
\end{align}
 If $\bigl|\frac{x}{\delta}\bigr|>\frac{1}{\delta^{\theta}}$ and $\bigl|\frac{x^{k}-\rho^{k}}{\beta^{k}}\bigr|>\frac{1}{\beta^{k\theta}}$, then $|x|>\delta^{1-\theta}$ and $\bigl|x^{k}-\rho^{k}\bigr|>\beta^{k(1-\theta)}$, which, together with Lemma \ref{lem2.1} and Lemma \ref{lem2.5}, implies
\begin{align}\label{h6}
|\mathcal{E}(x)| &\le\frac{C}{p}
\Bigl(\frac{\delta^{2(2+\alpha)\lambda}
|x|^{\alpha}}{|x|^{2+\alpha+2(2+\alpha)\lambda}}
+\frac{|x|^{2k-2}\beta^{2(2+\alpha)\lambda k}}{|x^k - \rho^k|^{2+2(2+\alpha)\lambda}}\Bigr)\notag \\& \le\frac{C}{p}
\Bigl(\frac{\delta^{(2+\alpha)(1+\theta)\lambda}
|x|^{\alpha}}{|x|^{2+\alpha+(2+\alpha)\lambda}}
+\frac{|x|^{2k-2}\beta^{(2+\alpha)(1+\theta)\lambda k}}{|x^k - \rho^k|^{2+(2+\alpha)\lambda}}\Bigr)\notag\\
&\leq\frac{C}{p^{4}}\Bigg(\displaystyle\frac{|x|^{\alpha}\delta^{(2+\alpha)\lambda}
}{(|x|^{2+\alpha}+\delta^{2+\alpha
})^{1+\lambda}}+
\frac{\beta^{(2+\alpha)k\lambda}|x|^{2k-2}}
{(|x^{k}-\rho^{k}|^{2}+\beta^{2k})^{1+
\frac{2+\alpha}{2}\lambda}}\Bigg).
\end{align}
Then from \eqref{h4}, \eqref{h5} and \eqref{h6}, we can conclude that $$\|\mathcal{E}\|_{*}\le\frac{C}{p^{4}}.$$

\end{proof}
\section{ The linear analysis}

The main purpose of this section is to prove that the linear operator $\mathcal{L}$,  defined in \eqref{1.9}, is invertible on a codimension-one subspace of $H^1_{0,k}(\Omega)$.  As a preliminary step, we review the classical results on the linearized equations in the whole space (see \cite{BP,LIA,MP}) and fix some notations.

From \cite{GGN}, we have that
if $\alpha=2m$ for some integer $m\geq1$, all the solutions of the equation:
\begin{align}\label{j.0}
-\Delta V(x)=|x|^{\alpha}e^{U_{\alpha}(x)}V(x),\qquad x\in\mathbb{R}^{2}
\end{align} are linear combination of
\begin{align}\label{j.01}
V_{\alpha,0}(x)
=\frac{|x|^{2+\alpha}-1}
{|x|^{2+\alpha}+
1}, \ \  V_{\alpha,1}(x)
=\frac{|x|^{m+1}\displaystyle\cos((
m+1)\theta)}
{|x|^{2+\alpha}+
1}, \ \
V_{\alpha,2}(x)
=\frac{|x|^{m+1}
\displaystyle\sin((
m+1)\theta)}
{|x|^{2+\alpha}+
1},
\end{align}
if $\alpha\notin2\mathbb{N}$, the solutions to equation \eqref{j.0} are spanned by
\begin{align}\label{j.02}V_{\alpha,0}(x)
=\frac{|x|^{2+\alpha}-1}
{|x|^{2+\alpha}+
1}
. \end{align}So,
if $\alpha=2m$, all the solutions of the equation: \begin{align}\label{J.0}
-\Delta V(x)=|x|^{\alpha}e^{U_{1}(x)}V(x),\qquad x\in\mathbb{R}^{2}
\end{align} are linear combination of
\begin{align}\label{z.06}V_{\alpha,0,\delta}(x)
=\frac{|\frac{x}{\delta}|^{2+\alpha}-1}
{|\frac{x}{\delta}|^{2+\alpha}+
1}, \ \  V_{\alpha,1,\delta}(x)
=\frac{|\frac{x}{\delta}|^{m+1}
\displaystyle\cos((
m+1)\theta)}
{|\frac{x}{\delta}|^{2+\alpha}+
1}, \ \
V_{\alpha,2,\delta}(x)
=\frac{|\frac{x}{\delta}|^{m+1}
\displaystyle\sin((
m+1)\theta)}
{|\frac{x}{\delta}|^{2+\alpha}+
1}.
\end{align}
If $\alpha\notin2\mathbb{N}$, the solutions to equation \eqref{J.0} are spanned by
\begin{align}\label{z.06-bis}V_{\alpha,0,\delta}(x):
=V_{\alpha,0}(\frac{x}{\delta})
=\frac{|\frac{x}{\delta}|^{2+\alpha}-1}
{|\frac{x}{\delta}|^{2+\alpha}+
1}
=\frac{|x|^{2+\alpha}-\delta^{2+\alpha}}
{|x|^{2+\alpha}+\delta^{2+\alpha}}. \end{align}

Furthermore, let \(U_{2}\) be the profile defined in \eqref{u01}. It is also known that all the solutions of the linear equation
\begin{align}\label{z.06-ter}
-\Delta Z(x)=|x|^{2k-2}e^{U_{2}(x)}Z(x),\qquad x=(x_{1},x_{2})\in\mathbb{R}^{2}
\end{align}
are linear combination of
\begin{align}\label{z.6}
Z_{0,2}(x) &= Z_{0}\left(\frac{x^{k}-\rho^{k}}{\beta^{k}}\right)=\frac{|x^{k}-\rho^{k}|^{2}-\beta^{2k}}{|x^{k}-\rho^{k}|^{2}+\beta^{2k}}, \notag \\
Z_{1,2}(x) &= Z_{1}\left(\frac{x^{k}-\rho^{k}}{\beta^{k}}
\right)=\frac{4\beta^{k}\bigl(
|x|^{k}\cos k\theta-\rho^{k}\bigr)}{|x^{k}-\rho^{k}|^{2}+\beta^{2k}}, \notag \\
Z_{2,2}(x) &= Z_{2}\left(\frac{x^{k}-\rho^{k}}{\beta^{k}}
\right)=\frac{4\beta^{k}|x|^k\sin k\theta}{|x^{k}
-\rho^{k}|^{2}+\beta^{2k}},
\end{align}
where
\begin{align}\label{pz.06}
Z_{0}(x)=\frac{|x|^{2}-1}{|x|^{2}+1},\quad
Z_{1}(x)=\frac{4x_{1}}{|x|^{2}+1},\quad
Z_{2}(x)=\frac{4x_{2}}{|x|^{2}+1}. \end{align}
It's obvious that
\begin{align}\label{z0}
\int_{\mathbb{R}^{2}}e^{U_{0}}Z_{i}Z_{j}=0\;\text{ if }\;i,j=0,1,2,\;i\neq j.
\end{align}
We set
\begin{align}\label{00z}
\widetilde{Z}(x):=Z_{1,2}(x) \end{align}
and we ask that the remainder term \(\phi\)  satisfies the orthogonality condition:
\begin{align}\label{0z}
\int_{\Omega}\nabla P\widetilde{Z}\cdot\nabla\phi
=\int_{\Omega}|x|^{2k-2}e^{U_{2}(x)}\widetilde{Z}(x)\phi(x)\,\mathrm{d}x=0. \end{align}

Define the \(k\)-symmetric functions :
\begin{align}\label{L0}
L^{p}_{k}(\Omega):=\Bigl\{u\in L^{p}(\Omega):\;u(x)=u(\overline{x})
=u\!\left(e^{\mathrm{i}\frac{2\pi}{k}}x\right)\text{ for a.e. }x\in\Omega\Bigr\}\subseteq L^{p}(\Omega)
\end{align}
and the \(k\)-symmetric functions \(H^{1}_{0,k}(\Omega)\subseteq H^{1}_{0}(\Omega)\) defined in \eqref{1.3}. Let us observe that \(H^{1}_{0,k}(\Omega)\) splits as the direct sum of its subspaces
\begin{align*}
\mathbf{Z}:=\bigl\{cP\widetilde{Z}:\;c\in\mathbb{R}\bigr\}\quad\text{and}\quad
\mathbf{Z}^{\perp}:=\Bigl\{\phi\in H^{1}_{0,k}(\Omega):\;
\int_{\Omega}|x|^{2k-2}e^{U_{2}(x)}\widetilde{Z}(x)\phi(x)\,\mathrm{d}x=0\Bigr\}.
\end{align*}

We now state the main result of this section, which establishes the invertibility of the linear operator $\mathcal{L}$, defined in \eqref{1.9}, on the space $\mathbf{Z}^{\perp}$.
\begin{proposition}\label{p.3}
Assume $\mathcal{L}$ is the linear operator introduced in \eqref{1.9} and that the parameters \(\delta,\beta,\rho,\tau,\eta\) in the definition of \(\Upsilon\) satisfy \eqref{1.11} and \eqref{x.1} with $\eta_{\infty,k}$ as defined in \eqref{n1}. Then there exists \(C>0\) such that, for any \(h\in L^{\infty}_{k}(\Omega)\), the linear problem
\begin{equation}\label{1.l1}
\begin{cases}
\mathcal{L}\phi(x)=h(x)+c|x|^{2k-2}e^{U_{2}(x)}\widetilde{Z}(x) & x\in\Omega \\[4pt]
\phi(x)=0 & x\in\partial\Omega\end{cases}
\end{equation}has a unique solution $(\phi, c) \in H_{0,k}^1(\Omega) \cap L_k^\infty(\Omega) \times \mathbb{R}$ such that $\phi$ satisfies \eqref{0z} and
\begin{equation}\label{1.l2}
c = \frac{p \displaystyle\int_{\Omega}|x|^{\alpha} |\Upsilon|^{p-1} P\widetilde{Z}\phi -hP\widetilde{Z}}
{\int_{\Omega}\big|\nabla P\widetilde{Z}\big|^{2}}.
\end{equation}
Furthermore
\begin{align}\label{ch1}
\|\phi\|_\infty + |c| \leq C p \|h\|_*, \end{align}
where the weighted norm $\|\cdot\|_*$ is defined in \eqref{h1}.
\end{proposition}
In order to prove Proposition \ref{p.3}, we need some preliminary Lemmas. In particular, we assume that the parameters \(\delta,\beta,\rho,\tau,\eta\) in the definition of \(\Upsilon\) satisfy \eqref{1.11} and \eqref{x.1} in  Proposition \ref{p.3}, so we can use all the estimates in Section 3.

\begin{lemma}\label{lh1}
There exist $\theta$ and $C>0$ such that if $|y|\leq\frac{1}{\delta^{\theta}}$ then
\begin{align}\label{ch2}
p|\Upsilon(\delta y)|^{p-1}= \frac{e^{U_{\alpha}(y)}}{\delta^{2+\alpha}}
\Bigl(1+\frac{1}{p}\Bigl(V_{\alpha}(y)-
U_{\alpha}(y)-\frac{U_{\alpha}(y)^{2}}{2}\Bigr)
+O\Bigl(\frac{\log^{4}(|y|^{2+\alpha}+2)}
{p^{2}}\Bigr)\Bigr)
\end{align}
and if $|y|\leq\frac{1}{\beta^{k\theta}}$ then
\begin{align}\label{ch3}
p\Bigl|\Upsilon\Bigl(\sqrt[k]{\beta^{k}y+\rho^{k}}\Bigr)\Bigr|^{p-1}
= \frac{k^{2}\rho^{2k-2-\alpha}}{\beta^{2k}}e^{U_{0}(y)}
\Bigl(1+\frac{1}{p}\Bigl(V_{0}(y)-U_{0}(y)-
\frac{U_{0}(y)^{2}}{2}\Bigr)
+O\Bigl(\frac{\log^{4}(|y|+2)}{p^{2}}\Bigr)
\Bigr),
\end{align}
where $\Upsilon$ is defined in \eqref{1.6}.

Moreover, there exists \(C>0\) such that
\begin{align}\label{ch4}
p|x|^{\alpha}|\Upsilon(x)|^{p-1} \leq C\bigl(|x|^{\alpha}e^{U_{1}(x)}+|x|^{2k-2}e^{U_{2}(x)}\bigr). \end{align}where $U_{1}(x)$ and $U_{2}(x)$ are  defined in \eqref{u00} and \eqref{u01}, separately.
\end{lemma}

\begin{proof}
If \(|y|\leq\frac{1}{\delta^{\theta}}\), then from \eqref{2.18} and \eqref{pe.1}, we can get
\begin{align*}
p|\Upsilon(\delta y)|^{p-1}
&= p\Bigl(\tau\Bigl(p+U_{\alpha}(y)+\frac{V_{\alpha}(y)}{p}
+O\Bigl(\frac{\log(|y|^{2+\alpha}+2)}{p^{2}}\Bigr)\Bigr)\Bigr)^{p-1} \\
&= \frac{1}{\delta^{2+\alpha}}\Bigl(1+\frac{U_{\alpha}
(y)}{p}+\frac{V_{\alpha}(y)}{p^{2}}
+O\Bigl(\frac{\log(|y|+2)}{p^{3}}\Bigr)\Bigr)^{p-1} \\
&= \frac{e^{U_{\alpha}(y)}}{\delta^{2+
\alpha}}\Bigl(1+\frac{1}{p}\Bigl(V_{\alpha}
(y)-\frac{U_{\alpha}(y)^{2}}{2}\Bigr)
+O\Bigl(\frac{\log^{4}(|y|^{2+\alpha}
+2)}{p^{2}}\Bigr)\Bigr) \\
&\quad\times\Bigl(1-\frac{U_{\alpha}
(y)}{p}+O\Bigl(\frac{\log(|y|+2)}{p^{2}}\Bigr)\Bigr) \\
&= \frac{e^{U_{\alpha}(y)}}{\delta^{2+
\alpha}}\Bigl(1+\frac{1}{p}\Bigl(V_{\alpha}
(y)-U_{\alpha}(y)-\frac{U_{\alpha}(y)^{2}}{2}\Bigr)
+O\Bigl(\frac{\log^{4}(|y|^{2+\alpha}+2)}
{p^{2}}\Bigr)\Bigr)
\end{align*}and this proves \eqref{ch2} when $|y|\leq\frac{1}{\delta^{\theta}}$.

With the aid of \eqref{2.108} and \eqref{pe.1}, we can obtain
\begin{align*}
p\Bigl|\Upsilon\Bigl(\sqrt[k]{\beta^{k}y+\rho^{k}}\Bigr)\Bigr|^{p-1}
&= p\Bigl(\tau\eta\Bigl(p+U_{0}(y)+\frac{V_{0}(y)}{p}
+O\Bigl(\frac{\log(|y|+2)}{p^{2}}\Bigr)\Bigr)\Bigr)^{p-1} \\
&= \frac{k^{2}\rho^{2k-2-\alpha}}{\beta^{2k}}
\Bigl(1+\frac{U_{0}(y)}{p}+\frac{V_{0}(y)}{p^{2}}
+O\Bigl(\frac{\log(|y|+2)}{p^{3}}\Bigr)\Bigr)^{p-1} \\
&= \frac{k^{2}\rho^{2k-2-\alpha}}{\beta^{2k}} e^{U_{0}(y)}\Bigl(1+\frac{1}{p}\Bigl(V_{0}(y)
-\frac{U_{0}(y)^{2}}{2}\Bigr)
+O\Bigl(\frac{\log^{4}(|y|+2)}{p^{2}}\Bigr)\Bigr) \\
&\quad\cdot\Bigl(1-\frac{U_{0}(y)}{p}+
O\Bigl(\frac{\log(|y|+2)}{p^{2}}\Bigr)\Bigr) \\
&= \frac{k^{2}\rho^{2k-2-\alpha}}{\beta^{2k}} e^{U_{0}(y)}\Bigl(1+\frac{1}{p}\Bigl(V_{0}(y)
-U_{0}(y)-\frac{U_{0}(y)^{2}}{2}\Bigr)
+O\Bigl(\frac{\log^{4}(|y|+2)}{p^{2}}\Bigr)
\Bigr)
\end{align*}
and this proves \eqref{ch3} when $|y|\leq\frac{1}{\beta^{k\theta}}$.

By  Lemma \ref{lem2.1} and Lemma \ref{lem2.2}, one has
\begin{align}\label{frac1}
\Upsilon(x)&=\tau\left(PU_{1}+\frac{V_{1}}
{p}
+\frac{W_{1}}{p^{2}}-\eta PU_{2}-\frac{\eta }{p^{1}}V_{2}
-\frac{\eta}{p^{2}} W_{2}\right)\notag\\
&=\tau\left(U_{\alpha}\left( \frac{x}{\delta} \right)-\eta U_{0}\left( \frac{x^k - \rho^k}{\beta^k} \right)+\frac{V_{\alpha}\left( \frac{x}{\delta} \right)}{p}-\frac{\eta}{p} V_{0}\left( \frac{x^k - \rho^k}{\beta^k} \right)+\frac{W_{\alpha}\left( \frac{x}{\delta} \right)}{p^{2}}- \frac{\eta W_{0}}{p^{2}}\left( \frac{x^k - \rho^k}{\beta^k} \right)
\right)\notag\\
&\quad+\tau\big(p+4k\eta\log\beta-4k\eta\log\rho
+O(1)\big)\notag\\
&=\tau\left(p+\log\frac{\bigl(|x^{k}-\rho^{k}|^{2}+
\beta^{2k}\bigr)^{2\eta}}
{(|x|^{2+\alpha}+\delta^{2+\alpha})^{2}}
+O\Big(\frac{1+o(1)}
{p}\big(V_{\alpha}\left( \frac{x}{\delta} \right)
-\eta V_{0}\left( \frac{x^k - \rho^k}{\beta^k} \right)
\big)
\Big)\right)
\notag
\\
&\quad+\tau\big(\log\delta^{2(2+\alpha)}
-4k\eta\log\rho
+O(1)\big)
\notag\\
&=\tau\left(p+\log\bigl(|x^{k}-\rho^{k}|^{2}+
\beta^{2k}\bigr)^{2\eta}
-2\log(|x|^{2+\alpha}+\delta^{2+\alpha})
+\log\delta^{2(2+\alpha)}-4k\eta\log\rho\right)
+O(\frac{1}{p})
\end{align}
and
\begin{align}\label{frac2}
\Upsilon(x)&=-\eta \tau\left(PU_{2}-\frac{PV_{2} }{p}
-\frac{PW_{2}}{p^{2}} -\frac{PU_{1}}{\eta}-\frac{PV_{1}}
{p\eta}
-\frac{PW_{1}}{p^{2}\eta} \right)\notag\\
&=-\tau\eta\left(U_{0}\left( \frac{x^k - \rho^k}{\beta^k} \right)-\frac{U_{\alpha}\left( \frac{x}{\delta} \right)}{\eta } +
\frac{V_{0}\left( \frac{x^k - \rho^k}{\beta^k} \right)}{p} -\frac{V_{\alpha}\left( \frac{x}{\delta} \right)}{p\eta}
+\frac{W_{0}}{p^{2}}\left( \frac{x^k - \rho^k}{\beta^k} \right)
-\frac{W_{\alpha}\left( \frac{x}{\delta} \right)}{p^{2}\eta}
\right)\notag\\
&\quad-\tau\eta\big(p
+\frac{2(2+\alpha)}{\eta}\log\delta-
\frac{2(2+\alpha)}{\eta}\log\rho
+O(1)\big)\notag\\
&=-\eta\tau\left(p+\frac{2}{\eta}\log(|x|^{2+\alpha}
+\delta^{2+\alpha})
-2\log
\bigl(|x^{k}-\rho^{k}|^{2}+
\beta^{2k}\bigr)
-
\frac{2(2+\alpha)}{\eta}\log\rho
+4k\log\beta
\right)\notag\\
&\quad
+O(\frac{1}{p}).
\end{align}

If \(|x|\le\frac{\rho}{2}\), then \(\frac{\rho^{2k}}{C}\le\bigl|x^{k}-\rho^{k}\bigr|^{2}+\beta^{2k}\le C\rho^{2k}\) for some \(C>0\), therefore, in view of \eqref{1.11} and \eqref{frac1}, one gets
\begin{align*}
\Upsilon(x)= \tau\Bigl(p+\log\frac{\delta^{2(2+\alpha)}}
{(|x|^{2+\alpha}+\delta^{2+\alpha})^{2}}\Bigr)+O\Bigl(\frac{1}{p}\Bigr) ,
\end{align*}
hence
\begin{align*}
p|x|^{\alpha}|\Upsilon(x)|^{p-1}
&= O\Bigl(|x|^{\alpha}p^{p}\tau^{p-1}
\Bigl(1+\frac{1}{p}
\log\frac{\delta^{2(2+\alpha)}}
{(|x|^{2+\alpha}+\delta^{2+\alpha})^{2}}
+O(\frac{1}{p})\Bigr)^{p-1}\Bigr) \\
&= O\Bigl(\frac{|x|^{\alpha}}{\delta^{2+\alpha}}
e^{
\log\frac{\delta^{2(2+\alpha)}}
{(|x|^{2+\alpha}+\alpha^{2+\alpha})^{2}}\left(1-\frac{1}{p}\right)}\Bigr) \\
&= O\bigl(|x|^{\alpha}e^{U_{1}(x)}\bigr).
\end{align*}

On the other hand, if \(\frac{\rho}{2}<|x|\le 2\rho\), then \(\frac{\rho^{2+\alpha}}{C}\le|x|^{2+\alpha}
+\delta^{2+\alpha}\le C\rho^{2+\alpha}\), by \eqref{frac2}, we have
\begin{align*}
\Upsilon(x) &= -p\tau\eta-\tau\eta
\log\frac{\beta^{4k}}
{\bigl(|x^{k}-\rho^{k}|^{2}+
\beta^{2k}\bigr)^{2}}+O\Bigl(\frac{1}{p}\Bigr),
\end{align*}
then we can obtain \(p|x|^{\alpha}|\Upsilon(x)|^{p-1}=O\bigl(|x|^{2k-2}e^{U_{2}(x)}\bigr)\).
If $1\leq|x|$, we can easily get
$$p|x|^{\alpha}|\Upsilon(x)|^{p-1}\leq C|x|^{\alpha}e^{U_{1}(x)}.$$

Finally, if \(1>|x|>2\rho\), then
\begin{align*}
\frac{|x|^{2+
\alpha}}{C}\le|x|^{2+
\alpha}+\delta^{2+
\alpha}\le C|x|^{2+
\alpha},\qquad
\frac{|x|^{2k}}{C}\le\bigl|x^{k}-\rho^{k}\bigr|^{2}+\beta^{2k}\le C|x|^{2k},
\end{align*}
therefore
\begin{align*}
\Upsilon(x)=4\tau(k\eta-1-\frac{\alpha}{2})\log|x|
+O\Bigl(\frac{1}{p}\Bigr).
\end{align*}

 If $e^{-\frac{p-1}{2k+2+\alpha}}\le r\le 1$, the function $r\mapsto r^{2k+2+\alpha}\log^{p-1}\frac{1}{r}$ is decreasing.   Noting that \(\eta\) satisfies \eqref{x.5}, combining with Lemma \ref{lem2.1}, we can deduce $e^{-\frac{p-1}{2k+2+\alpha}}\le C\rho$ for some $C>0$. Then for $ 1>|x|>2\rho$, we can conclude that
\begin{align*}
\log^{p-1}\frac{2}{C|x|}\le \frac{(2\rho)^{2k+2+\alpha}}{|x|^{2k+2+\alpha}}
\Bigl(\log\frac{1}{C\rho}\Bigr)^{p-1}
\Rightarrow \log\frac{1}{|x|}\leq
\left(\frac{2\rho}{|x|}\right
)^{\frac{2k+2+\alpha}{p-1}}
\log\frac{1}{\rho}+O(1).
\end{align*}
From Lemma \ref{lem2.1} and \eqref{x.2}, we can get $\frac{\rho^{\frac{2(2+\alpha)}{\eta}}}{
\beta^{4k}}=O(e^{p})$ and $k-\frac{2+\alpha}{2
\eta}>0$. Then, by \eqref{x.5}, we have
\begin{align*}
p|x|^{\alpha}|\Upsilon(x)|^{p-1}
&= p^{p}\tau^{p-1}|x|^{\alpha}
\Bigl(1+O\Bigl(\frac{1}{p}\Bigr)\Bigr)^{p-1}
\Bigl(\frac{4(k\eta-1-\frac{\alpha}{2})\log\frac{1}{|x|}+O(1)}{p}\Bigr)^{p-1} \\
&= O\Bigl(|x|^{\alpha}\frac{\rho^{2k-2-\alpha}}{\beta^{2k}}
\Bigl(\frac{4\bigl(k-\frac{1}{\eta}-\frac{\alpha}{2\eta}\bigr)
\log\frac{1}{|x|}+O(1)}{p}\Bigr)^{p-1}\Bigr) \\
&= O\Bigl(\frac{\rho^{4k}}{\beta^{2k}}
\Bigl(\frac{4\bigl(k-\frac{1}{\eta}-\frac{\alpha}{2\eta}
\bigr)\log\frac{1}{\rho}+O(1)}{p}\Bigr)^{p-1}
\frac{|x|^{\alpha}}{|x|^{2k+2+\alpha}}\Bigr) \\
&= O\Bigl(\rho^{4\bigl(k-\frac{2+\alpha}{2
\eta}\bigr)}e^{p}
\Bigl(\frac{4\bigl(k-\frac{1}{\eta}-\frac{\alpha}{2\eta}\bigr)\log\frac{1}{\rho}+O(1)}{p}\Bigr)^{p}
\frac{\beta^{2k}|x|^{\alpha}}{|x|^{2k+2+\alpha}}\Bigr) \\
&= O\Bigl(\frac{\beta^{2k}}{|x|^{2k+2}}\Bigr) 
= O\bigl(|x|^{2k-2}e^{U_{2}(x)}\bigr),
\end{align*}where we use the elementary inequality $
|x|\,e^{p}\Bigl(\frac{\log\frac{1}{|x|}}{p}\Bigr)^{p}\le 1,$ for $ |x|\in(0,1).$ This concludes the proof.
\end{proof}
\begin{lemma}\label{lh3} Let $\phi\in\mathbf{Z}^{\perp}$ be a solution to \eqref{1.l1}. Then
\begin{align*}
\|\phi\|_{H_0^1} \leq C \left( \|\phi\|_\infty + \|h\|_* \right).\end{align*}
\end{lemma}
\begin{proof}
Multiplying equation \eqref{1.l1} by \( \phi \) and integrating by parts,  together with \eqref{ch4}, gives that
\begin{align*}
\|\phi\|_{H_0^1}^2 &= \int_{\Omega} p |x|^{\alpha} |\Upsilon|^{p-1} \phi^2 - \int_{\Omega} h \phi\\&
 \leq C \|\phi\|_\infty^2 \int_{\Omega}|x|^{\alpha}  \left( e^{U_1(x)} + |x|^{2k-2} e^{U_2(x)} \right) dx
\\&\quad+
C \|\phi\|_\infty \|h\|_* \int_{\Omega} \left( \displaystyle\frac{\delta^{(2+\alpha)\lambda}
|x|^{\alpha}}{(|x|^{2+\alpha}+\delta^{2+\alpha
})^{1+\lambda}}
+\frac{\beta^{(2+\alpha)k\lambda}|x|^{2k-2}}
{(|x^{k}-\rho^{k}|^{2}+\beta^{2k})^{1+
\frac{2+\alpha}{2}\lambda}}\right) dx
\\&\leq C \left( \|\phi\|_\infty^2 + \|\phi\|_\infty \|h\|_* \right),
\end{align*}
that is
\begin{align*}
\|\phi\|_{H_0^1} \leq C \left( \|\phi\|_\infty + \|h\|_* \right).\end{align*}
\end{proof}
\begin{lemma}\label{lh4} The operator $ \mathcal{L} $ satisfies the maximum principle on the domain
\begin{align}\label{LPf1}
\tilde{\Omega} := \left\{ x \in \Omega : \left| \frac{x}{\delta} \right|^{2+\alpha} > R^{2}, \; \left| \frac{x^k - \rho^k}{\beta^k} \right| > R \right\},
\end{align}for some $R>0$, independent of $p$.
Namely,
\begin{align*}
\begin{cases}
\mathcal{L}\psi \leq 0, & \text{in } \tilde{\Omega}, \\
\psi \geq 0, & \text{on } \partial\tilde{\Omega},
\end{cases}
\quad \Rightarrow \quad \psi \geq 0, \
\text{ in } \ \tilde{\Omega}.
\end{align*}\end{lemma}
\begin{proof}
We argue it by contradiction. Assume there is a \( \psi \) satisfying
\begin{align*}
\begin{cases}
\mathcal{L}\psi \leq 0, & \text{in } \ \tilde{\Omega}, \\
\psi < 0, & \text{in } \ \omega, \\
\psi = 0, & \text{on }\  \partial\omega,
\end{cases}
\end{align*}
for some region \( \omega \subset \tilde{\Omega} \) .
Define
\begin{align*}
\zeta(x) := 1 - 2C \frac{\delta^{2+\alpha}}{|x|^{2+\alpha}} - 2C \frac{\beta^{2k}}{|x^k - \rho^k|^2},
\end{align*}
where $C > 0 $ is as in \eqref{ch4}.

 When \( R \geq 2\sqrt{C} \), $\zeta(x)$ satisfies  \begin{align*}1>
 \zeta > 0, \quad \text{in } \  \tilde{\Omega}(x)\end{align*}and

\begin{align*}
\Delta\zeta(x) &= -2C \left((2+\alpha^{2}) \frac{\delta^{2+\alpha}}{|x|^{4+\alpha}} + \frac{4k^2\beta^{2k}|x|^{2k-2}}{|x^k - \rho^k|^4} \right)
\\&
< -C \left( |x|^{\alpha}e^{U_1(x)} + |x|^{2k-2}e^{U_2(x)} \right)
\\&
< -C \left( |x|^{\alpha}e^{U_1(x)} + |x|^{2k-2}e^{U_2(x)} \right) \zeta(x),
\end{align*}
which, together with \eqref{ch4}, implies that
\begin{align}\label{ff1}
\mathcal{L}\zeta(x) = \Delta\zeta(x) + p|\Upsilon(x)|^{p-1}\zeta(x) < 0.
\end{align}
Then, we can conclude that
\begin{align*}
0 < \int_{\omega} \psi \mathcal{L}\zeta = \int_{\omega} \zeta \mathcal{L}\psi - \int_{\partial\omega} \zeta \partial_n \psi \leq 0,
\end{align*}
which gives a contradiction.
\end{proof}

\begin{lemma}\label{lh5} There exists some \( C > 0 \) such that for any $ h \in L^\infty_k(\Omega) $, if $ \phi \in H^1_{0,k}(\Omega) \cap L^\infty_k(\Omega) $ satisfies
\begin{align}\label{0lh6}
\begin{cases}
\mathcal{L}\phi = h, & \text{in } \ \Omega, \\
\phi = 0, & \text{on } \ \partial\Omega,
\end{cases}
\end{align}
it holds that
\begin{align*}
\|\phi\|_\infty \leq C\bigl(\|\phi\|_{L^\infty(\Omega \setminus \tilde{\Omega})} + \|h\|_*\bigr),
\end{align*}where $\tilde{\Omega}$ is defined in \eqref{LPf1}.
\end{lemma}
\begin{proof} 
Let us define the functions \( \psi_i \), with $i = 1, 2 $,
\begin{align*}
\psi_1(x) &= -\frac{\delta^{(2+\alpha)\lambda}}
{(2+\alpha)^2\lambda^2} \frac{1}{|x|^{(2+\alpha)\lambda}} + A_1 + B_1 \log |x|,
\\
\psi_2(x) &= -\frac{\beta^{(2+\alpha)k\lambda}}
{(2+\alpha)^2
\lambda^2k^2 } \frac{1}{|x^k - \rho^k|^{(2+\alpha)\lambda}} + A_2 + B_2 \log |x^k - \rho^k|,
\end{align*}
with
\begin{align*}
A_1 &= \frac{\delta^{(2+\alpha)\lambda}}
{(2+\alpha)^2
\lambda^2 M_1^{(2+\alpha)\lambda}} - B_1 \log M_1,
\, \, \quad \quad \quad B_1
= \frac{\left( \frac{\delta^{(2+\alpha)\lambda}}{M_1^{(2+\alpha)\lambda}} - \frac{1}{R^{2\lambda}} \right)}{(2+\alpha)^2
\lambda^2}  \frac{1}{\log \frac{M_1}{R^{\frac{2}{2+\alpha}}\delta}} < 0,
\\
A_2 &= \frac{\beta^{(2+\alpha)k\lambda}}{(2+\alpha)^2
\lambda^2k^2 M_2^{(2+\alpha)\lambda}} - B_2 \log (M_2),
\, \,\quad B_2
= \frac{\left( \frac{\beta^{(2+\alpha)k\lambda}}
{M_2^{(2+\alpha)\lambda}} - \frac{1}{R^{(2+\alpha)\lambda}} \right) }{(2+\alpha)^2
\lambda^2k^2} \frac{1}{\log \frac{M_2}{R\beta^k}} < 0,
\end{align*}
where $\lambda \in (0,1)$ is as in Lemma \ref{lem2.5} and $M_1 := 2 \operatorname{diam}(\Omega)$, $M_2 := 2 \operatorname{diam}(\Omega^k)$. They satisfy
\begin{align}\label{D1}
\begin{cases}
-\Delta\psi_1(x) = \frac{|x|^{\alpha}\delta^{(2+\alpha)\lambda}}{
|x|^{(2+\alpha)(1+\lambda)}}, & R^{\frac{2}{2+\alpha}}\delta < |x| < M_1 \\
\psi_1(x) = 0, & |x| = R^{\frac{2}{2+\alpha}}\delta, \; |x| = M_1
\end{cases}
\end{align}and
\begin{align}\label{D2}
\begin{cases}
-\Delta \psi_2(x) = \frac{\beta^{(2+\alpha)k\lambda} |x|^{2k-2}}{|x^k - \rho^k|^{2+(2+\alpha)\lambda}}, & R \beta^k < |x^k - \rho^k| < M_2, \\
\psi_2(x) = 0, & |x^k - \rho^k| = R \beta^k, \; |x^k - \rho^k| = M_2.
\end{cases}
\end{align}
By the maximum principle for $-\Delta$, both functions are positive and satisfy
\begin{align*}
\psi_1(x) \leq A_1 + B_1 \log(R^{\frac{2}{2+\alpha}}\delta) = \frac{\delta^{(2+\alpha)\lambda}}
{(2+\alpha)^2
\lambda^2 M_1^{(2+\alpha)\lambda}}  - B_1 \log \frac{M_1}{R^{\frac{2}{2+\alpha}}\delta} = \frac{1}{(2+\alpha)^2
\lambda^2  R^{2\lambda}}
\end{align*}
and
\begin{align*}
\psi_2(x) \leq A_2 + B_2 \log(R\beta^k) = \frac{\beta^{(2+\alpha)k\lambda}}{(2+\alpha)^2
\lambda^2k^2 M_2^{(2+\alpha)\lambda}} - B_2 \log \frac{(M_2)}{R\beta^k}= \frac{1}{(2+\alpha)^2
\lambda^2k^2R^{(2+\alpha)\lambda}},
\end{align*}
so that they are uniformly bounded. In view of the definitions of $M_1$ and $M_2$, we have
\begin{align*}
|x| < M_1, \quad |x^k - \rho^k| < M_2, \quad \forall x \in \tilde{\Omega},
\end{align*}
which shows that $\psi_1, \psi_2$ respectively satisfy \eqref{D1}, \eqref{D2} in $\tilde{\Omega}$.

Now, we take $h, \phi$ solving \eqref{0lh6} and define
\begin{align}\label{hhh2}
\tilde{\phi} := 2\|\phi\|_{L^\infty(\Omega \setminus \tilde{\Omega})} \zeta + 2\|h\|_* (\psi_1 + \psi_2),
\end{align}
where
$\zeta$ is the barrier function defined in the proof of Lemma \ref{lh4}. 
We can choose a larger value of $R$ in the definition of $\tilde{\Omega}$, namely we let 
$R^\lambda  \geq \max\{1,\frac{2k
\sqrt{2C}}{\lambda}\}$ 
, with $C$ as in \eqref{ch4}, so that we can derive $\zeta \geq \frac{1}{2}$ on $\tilde{\Omega}$. Therefore if either $|x| = R^{\frac{2}{2+\alpha}}\lambda$ or $|x^k - \rho^k| = R\beta^k$,
we have
\[
\tilde{\phi}(x) \geq 2\|\phi\|_{L^\infty(\Omega \setminus \tilde{\Omega})} \zeta(x) \geq \|\phi\|_{L^\infty(\Omega \setminus \tilde{\Omega})} \geq |\phi(x)|.
\]
Noting that $\zeta, \psi_1, \psi_2$ are all positive, we can get $\tilde{\phi} \geq 0 = \phi$ on $\partial\Omega$. 
By \eqref{ff1} 
and \eqref{hhh2},  we have
\begin{align*}
&\mathcal{L}\tilde{\phi}(x) =2\|\phi\|_{L^\infty(\Omega \setminus \tilde{\Omega})} \mathcal{L}\zeta+2\|h\|_* (\mathcal{L}\psi_1 + \mathcal{L}\psi_2)\\&\leq 2\|h\|_* (\mathcal{L}\psi_1 + \mathcal{L}\psi_2)\\
&=2\|h\|_* \left(- \displaystyle\frac{\delta^{(2+\alpha)\lambda}
|x|^{\alpha}}{|x|^{(2+\alpha)(1+\lambda)}}
-\frac{\beta^{(2+\alpha)k\lambda}|x|^{2k-2}}
{|x^{k}-\rho^{k}|^{2+(2+\alpha)\lambda}}+ p|x|^{\alpha}|\Upsilon|^{p-1}(x)(\psi_1(x) + \psi_2(x))\right)\\
&\leq  -2\|h\|_* \left( \displaystyle\frac{\delta^{(2+\alpha)\lambda}
|x|^{\alpha}}{(|x|^{2+\alpha}+\delta^{2+\alpha
})^{1+\lambda}}
+\frac{\beta^{(2+\alpha)k\lambda}|x|^{2k-2}}
{(|x^{k}-\rho^{k}|^{2}+\beta^{2k})^{1+
\frac{2+\alpha}{2}\lambda}}\right)\\&\quad+
2C\|h\|_*\left(\frac{2(2+\alpha)^{2}|x|^{\alpha}
\delta^{2+\alpha}}{(|x|^{2+\alpha} + \delta^{2+\alpha})^{2}}+\frac{8k^2|x|^{2k-2}
\beta^{2k}}{(|x^k - \rho^k|^{2} + \beta^{2k})^2}\right)\left(\frac{1}{(2+\alpha)^2
\lambda^2  R^{2\lambda}}+
\frac{1}{(2+\alpha)^2
\lambda^2k^2R^{(2+\alpha)\lambda}}
\right)\\
&\leq 
-2\|h\|_* \left(1-\frac{4Ck^2}{\lambda^2R^{2\lambda}}
\right)\left( \displaystyle\frac{\delta^{(2+\alpha)\lambda}
|x|^{\alpha}}{(|x|^{2+\alpha}+\delta^{2+\alpha
})^{1+\lambda}}
+\frac{\beta^{(2+\alpha)k\lambda}|x|^{2k-2}}
{(|x^{k}-\rho^{k}|^{2}+\beta^{2k})^{1+
\frac{2+\alpha}{2}\lambda}}
\right)\\
&
\leq -|h(x)|=-|\mathcal{L}\phi(x)|
\end{align*}
by the choice of $R$.
Then by
Lemma \ref{lh4}, we have $ |\phi(x)| \leq \tilde{\phi}(x) \) for all \( x \in \tilde{\Omega}$, which, together with \eqref{hhh2}, gives that
\begin{align*}
\|\phi\|_{\infty} \leq \| \tilde{\phi}\|_{\infty} \leq 2\|\phi\|_{L^{\infty}(\Omega \setminus \tilde{\Omega})} \|\zeta\|_{\infty} + 2\|h\|_* (\|\psi_1\|_{\infty} + \|\psi_2\|_{\infty})
\leq 2\|\phi\|_{L^{\infty}(\Omega \setminus \tilde{\Omega})} + \frac{
1}{\lambda^2 R^{2\lambda}} \|h\|_*.
\end{align*}\end{proof}
Here we give a uniform estimate of the solution $\phi$ of the linear problem \eqref{1.l1} with $c=0$.

\begin{lemma}\label{hD2}
There exists some $C>0$ such that, for any $h\in L^{\infty}_{k}(\Omega)$, if $\phi\in H^{1}_{0,k}(\Omega)\cap L^{\infty}_{k}(\Omega)$ satisfies
\begin{equation}\label{h1.0}
\begin{cases}
\mathcal{L}\phi=h, &\text{in }\ \Omega,\\[4pt]
\phi=0, &\text{on }\ \partial\Omega,\\[4pt]
\displaystyle\int_{\Omega}|x|^{2k-2}
e^{U_{2}(x)}\widetilde{Z}(x)\phi(x)\,
\mathrm{d}x=0,
\end{cases}\end{equation}
then
$$
\|\phi\|_{\infty}\leq Cp\|h\|_{*},
$$
where \(\widetilde{Z}\) is defined in \eqref{00z}, \(\|\cdot\|_{*}\) is the weighted norm defined in \eqref{h1}.
\end{lemma}

\begin{proof}
We argue it by contradiction and we assume there exists sequences \(\phi_{n},h_{n}\) satisfying \eqref{h1.0} for $p_{n}\to+\infty$ such that  $\|\phi_{n}\|_{\infty}=1$ and $p_{n}\|h_{n}\|_{*}\to 0$.

Denoting by $
\phi_{n,1}(y):=\phi_{n}(\delta y)$ and $
\phi_{n,2}(y):=\phi_{n}\left(\sqrt[k]
{\beta^{k}y+\rho^{k}}\right)$, from  \eqref{h1.0}, we have
\begin{align*}
\Delta\phi_{n,1}(y)+\delta^{2+\alpha}p_n|y|^{
\alpha}|\Upsilon(\delta y)|^{p_n-1}\phi_{n,1}(y)=\delta^{2}h(\delta y),
\quad \text{in }\ \tilde{\Omega}_{n,1},
\end{align*}
where $\tilde{\Omega}_{n,1}:=\left\{y\in\frac{\Omega}{\delta}:|y|<\frac{1}
{\delta^{\theta}}\right\}$, and
\begin{align*}
\Delta\phi_{n,2}+\frac{\beta^{2k}}{k^{2}}
\big|\beta^{k}y+\rho^{k}\big|^{\frac{\alpha}{k}+\frac{2}{k}-2}
p_n\Bigl|\Upsilon\Bigl(\sqrt[k]
{\beta^{k}y+\rho^{k}}\Bigr)\Bigr|^{p_n-1}\phi_{n,2}(y)
=\frac{\beta^{2k}}{k^{2}}
\frac{h\left(\sqrt[k]{\beta^{k}y+\rho^{k}}\right)
}{\big(\beta^{k}y+\rho^{k}\big)^{2-
\frac{2}{k}}},
\ \text{in }\ \tilde{\Omega}_{n,2},
\end{align*}
where $\tilde{\Omega}_{n,2}:=\left\{y\in\frac{\Omega^{k}-\rho^{k}}{\beta^{k}}
:|y|<\frac{1}{\beta^{k\theta}}\right\}$. Observe that $\tilde{\Omega}_{n,1}$ and $\tilde{\Omega}_{n,2}$ converge to $\R^2$ as $n\to \infty$.
With the aid of estimate \eqref{ch2} in $\tilde{\Omega}_{n,1}$, estimate \eqref{ch3}  in $\tilde{\Omega}_{n,2}$, the fact that $\|\phi_{n,1}\|\leq1$, $\|\phi_{n,2}\|\leq1$ and elliptic estimates, we have that $\phi_{n,1}$ and $\phi_{n,2}$ converge locally uniformly, as $n\to \infty$, to bounded solutions $\phi_{\infty,1}$, $\phi_{\infty,2}$ to
\begin{equation*}
\Delta\phi_{\infty,1}+|y|^{
\alpha}e^{U_{\alpha}}\phi_{\infty,1}=
0,\qquad\quad\ \Delta\phi_{\infty,2}+e^{U_{0}}\phi_{\infty,2}=0,\qquad\quad\text{in }\ \mathbb{R}^{2},
\end{equation*}
respectively. Now, from \eqref{j.01}, we have that if $\alpha=2m$ for some $m\geq1$, then $\phi_{\infty,1}=a_{0}V_{\alpha,0}+a_{1}V_{\alpha,1}+a_{2}V_{\alpha,2}$, while if $\alpha\neq2m$, then $\phi_{\infty,1}=a_{0}V_{\alpha,0}$, for some $a_0,a_1,a_2\in \R$, where $V_{\alpha,i}$,  for  $i=0,1,2$, are as defined in \eqref{j.01}.  Recalling that $\phi_{n}\in L^{\infty}_{k}(\Omega)$, we have that  \begin{align*}\phi_{n,1}(x_1, -x_2)=\phi_{n,1}(x_1,x_2)\ \ \text{and} \ \ \phi_{n,1}\big(e^{i\frac{2\pi}{k}} x
\big)=\phi_{n,1}(x) \end{align*} for every $x\in\Omega$ and the same holds for the function  $\phi_{\infty,1}$ for every $y\in \R^2$.  This implies $a_{2}=0$ and, combining with \eqref{k.1} also $a_{1}=0$. Then for every $\alpha>0$,
$$\phi_{\infty,1}=a_{0}V_{\alpha,0}.$$
By \eqref{z.06-ter}, \eqref{z.6} and \eqref{pz.06}, we have that
 $\phi_{\infty,2}=c_{0}Z_{0}+ c_{1}Z_{1}+c_{2}Z_{2}$. Recalling the definition of $U_2(x)$,  we can pass to the limit in the third equation of \eqref{h1.0}, by the dominated convergence theorem,  getting
\begin{align*}
\int_{\mathbb{R}^{2}}e^{U_0(y)}Z_{1}(y)
\phi_{\infty,2}(y)\,\mathrm{d}y=0,
\end{align*}
 which implies $c_{1}=0$. Using that
 $$\phi_{n,2}(y_1, -y_2)=\phi_{n}\left(\sqrt[k]
{\beta^{k}(y_1, -y_2)+\rho^{k}}\right)=
\phi_{n}\left(\sqrt[k]
{\beta^{k}(y_1, y_2)+\rho^{k}}\right)
 =\phi_{n,2}(y_1, y_2),$$
 then also $c_{2}=0$ and $\phi_{\infty,2}=c_{0}Z_{0}.$
Thus, we can conclude that
\begin{align*}
\phi_{n,1}\to a_{0}V_{\alpha,0},\qquad
\phi_{n,2}\to c_{0}Z_{0},\qquad
\text{in }\ L^{\infty}_{\text{loc}}\left(\mathbb{R}^{2}\right),
\end{align*}
for some $a_{0},c_{0}\in\mathbb{R}$. If $a_{0}=c_{0}=0$,  then both $\phi_{n,1}$ and $\phi_{n,2}$ converge to $0$ in $L^{\infty}_{\text{loc}}\left(\mathbb{R}^{2}
\right)$. From Lemma \ref{lh5} and the definition of $\tilde{\Omega}$, it follows that
\begin{align*}
1 &= \|\phi_{n}\|_{\infty} \\
 &\le C\bigl(\|\phi_{n}\|_{L^{\infty}(\Omega\backslash\widetilde{\Omega})}
        +\|h_{n}\|_{*}\bigr) \\
 &=  C\Bigg(\max\Bigl\{\sup_{|x|\leq R^{\frac{2} {2+\alpha}}}|\phi_{n,1}(x)|,\;
        \sup_{|x|\leq R}|\phi_{n,2}(x)|\Bigr\}+o(1)\Bigg) \to 0,
\end{align*}
which gives a contradiction. So to conclude the proof it suffices to show that $a_{0}=c_{0}=0$.
Let us we define the function
\begin{align*}
S(x)&:=\frac{2}{3}\frac{|x|^{2+\alpha}-1}{|x|^{2+\alpha}+1}
       \log\left(|x|^{2+\alpha}+1\right)-
       \frac{4}{3}\frac{1}{|x|^{2+\alpha}
       +1}, \ \ \
T(x):=\frac{2}{|x|^{2+\alpha}+1},
\end{align*}which solves
\begin{align*}
\Delta S + |x|^{\alpha}e^{U_{\alpha} } S = |x|^{\alpha}e^{U_{\alpha} }V_{\alpha,0} \quad \text{in } \mathbb{R}^2, \quad \Delta T +  |x|^{\alpha}e^{U_{\alpha} } T =  |x|^{\alpha}e^{U_{\alpha} }, \quad \text{in } \mathbb{R}^2.
\end{align*}
Next, we consider the function
\begin{align*}
Q_{1}(x): = S\left(\frac{x}{\delta}\right)
          +\Bigl(\frac{2(2+\alpha)}{3}
         \log\delta \Bigr)V_{\alpha,0,\delta}(x)
          +\frac{4(2+\alpha)\pi}{3}H(0,0)\,
          T\left(\frac{x}{\delta}\right),
\end{align*}
where $V_{\alpha,0,\delta}(x)$ is defined in \eqref{z.06-bis}.
Due to \begin{align*}
S(x) = \frac{2(2+\alpha)}{3} \log |x| + O\left(\frac{1}{|x|}\right), \quad T(x) = O\left(\frac{1}{|x|^{2+\alpha}}\right), \ \ \text{as} \ \
 |x|\rightarrow\infty,
\end{align*}
we have, uniformly for $ x \in \partial\Omega$,
\begin{align*}
PQ_1(x) - Q_1(x) &= -\frac{2(2+\alpha)}{3} \log |x| + O(\delta)
\end{align*}
and, reasoning as in the proof of Lemma \ref{lem2.2}, we can obtain, for $ x \in \Omega$,
\begin{align}\label{pq21}
PQ_1(x) = Q_1(x) - \frac{4(2+\alpha)}{3}\pi H(0,0) + O(|x| + \delta)
\end{align}
and, for $|y| \leq \frac{1}{2}\frac{\rho^k}{\beta^k}$,
\begin{align*}
PQ_1\left((\beta^k y + \rho^k)^{1/k}\right) = \frac{2(2+\alpha)}{3}k \log \rho - \frac{4(2+\alpha)}{3}\pi H(0, 0) + O\left(\frac{\beta^k}{\rho^k}|y| + \frac{\delta^{2+\alpha}}{\rho^{2+\alpha}} + \rho\right).
\end{align*}The function
$PQ_1$ solves
\begin{align}\label{PQ1}
\Delta PQ_1 + p_n |x|^{\alpha}|\Upsilon|^{p_n-1}PQ_1 =  |x|^{\alpha} e^{U_1}V_{\alpha,0,\delta}+ (p_n|\Upsilon|^{p_n-1} - e^{U_1}) |x|^{\alpha}PQ_1 + R_1,
\end{align}
with Dirichlet boundary conditions, where
$$
R_1(x) := \left( PQ_1(x) - Q_1(x) + \frac{4(2+\alpha)}{3} \pi H(0, 0) \right) |x|^{\alpha}e^{U_1(x)}.
$$
By \eqref{pq21}, we can derive
\begin{align}\label{R1}
R_1(x) = O\big(|x|^{\alpha}e^{U_1(x)}(|x| + \delta)\big), \ \ \text{for} \ \ x\in\Omega. \end{align}
Multiplying \eqref{PQ1} by $ \phi_n $ and integrating by parts, we can get
\begin{align}\label{fgyp0}
\int_{\Omega} PQ_1 h_n \, dx= \int_{\Omega} |x|^{\alpha} e^{U_1}V_{\alpha,0,\delta} \phi_n\, dx + \int_{\Omega} |x|^{\alpha} (p_n|\Upsilon|^{p_n-1} - e^{U_1}) PQ_1 \phi_n \, dx+ \int_{\Omega} R_1 \phi_n\, dx.
\end{align}
We estimate each term. The assumption $ \|h_n\|_* = o\left(\frac{1}{p_n}\right) $, \eqref{pq21}
and the properties of $ Q_1 $ yield
\begin{align}\label{fgyp1}
\int_{\Omega} &PQ_1 h_n \, dx= O\left(\|h_n\|_* \int_{\Omega} \Big( \displaystyle\frac{\delta^{(2+\alpha)\lambda}
|x|^{\alpha}}{(|x|^{2+\alpha}+\delta^{2+\alpha
})^{1+\lambda}}
+\frac{\beta^{(2+\alpha)k\lambda}|x|^{2k-2}}
{(|x^{k}-\rho^{k}|^{2}+\beta^{2k})^{1+
\frac{2+\alpha}{2}\lambda}}\Big)|PQ_1(x)| \, dx\right)
\notag\\&
= o\left(\frac{1}{p_n} \int_{\Omega} \Big( \displaystyle\frac{\delta^{(2+\alpha)\lambda}
|x|^{\alpha}}{(|x|^{2+\alpha}+\delta^{2+\alpha
})^{1+\lambda}}
+\frac{\beta^{(2+\alpha)k\lambda}|x|^{2k-2}}
{(|x^{k}-\rho^{k}|^{2}+\beta^{2k})^{1+
\frac{2+\alpha}{2}\lambda}}\Big) (|Q_1(x)| + 1) \, dx\right)\notag\\&
= o\left(\frac{1}{p_n}
\int_{\Omega} \Big( \displaystyle\frac{\delta^{(2+\alpha)\lambda}
|x|^{\alpha}}{(|x|^{2+\alpha}+\delta^{2+\alpha
})^{1+\lambda}}
+\frac{\beta^{(2+\alpha)k\lambda}|x|^{2k-2}}
{(|x^{k}-\rho^{k}|^{2}+\beta^{2k})^{1+
\frac{2+\alpha}{2}\lambda}}\Big)
 \left(\log\frac{1}{\delta} + \log\left(\frac{|x|^{2+\alpha}}
 {\delta^{2+\alpha}} + 1\right)\right) dx\right)\notag\\
&=o\left(\frac{1}{p_n}
\int_{\mathbb{R}^2}\displaystyle\frac{
|y|^{\alpha}}{(|y|^{2+\alpha}+1)^{1+\lambda}}
\left(\log\frac{1}{\delta} + \log\left(|y|^{2+\alpha}+1\right) \right)\, dy
\right)\notag \\&
\quad+o\left(\frac{1}{p_n}
\int_{\mathbb{R}^2}\displaystyle\frac{
1}{(|y|^{2}+1)^{1+\frac{2+\alpha}{2}\lambda}}
\left(\log\frac{1}{\delta} + \log\left(\frac{|\beta^{k}y+\rho^{k}|^{\frac{2+\alpha}{k
}}}{\delta^{2+\alpha}
}+1\right)\, dy \right)
\right) \notag\\&= o(1),
\end{align}
where we used that $\frac 1 p_n\log \frac 1\delta\leq C$ by Lemma \ref{lem2.1}. 
By the dominated convergence theorem,  using that $\phi_{n,1}\to a_{0}V_{\alpha,0}$, we have
\begin{align}\label{fgyp2}
\int_{\Omega}|x|^{\alpha} e^{U_1} V_{\alpha,0,\delta} \phi_n\, dx \to a_{0} \int_{\mathbb{R}^2}|y|^{\alpha} e^{U_{\alpha}} V_{\alpha,0}^{ 2}\, dx= \frac{4(2+\alpha)}{3} \pi a_{0}.
\end{align}
Next, by Lemma \ref{lh1}, we can get
\begin{align*}
&\quad\int_{\Omega} |x|^{\alpha} (p_n|\Upsilon|^{p_n-1} - e^{U_1}) PQ_1 \phi_n\, dx
\\
&=  \int_{\{x\in \Omega : |x| < \delta^{1-\theta}\}} |x|^{\alpha} (p_n |\Upsilon|^{p_n-1} - e^{U_1}) PQ_1 \phi_n\, dx + \int_{\{x\in \Omega : |x^{k}-\rho^{k}|\leq\beta^{k(1-\theta)} \} } p_n |x|^\alpha  |\Upsilon|^{p_n-1} PQ_1 \phi_n\, dx\\
&
\quad+ 
O \left(  \int_{\{x\in\Omega :  
|x| \geq
\delta^{1-\theta}\}}
|x|^{\alpha}e^{U_1(x)} dx  \|PQ_1\|_{\infty} \|\phi_n\|_{L^\infty}\right)\\&
\quad+ 
O \left( \int_{\big\{x\in  \Omega : |x| \geq \delta^{1-\theta}\}\big\}\backslash\{|x^{k}
-\rho^{k}|\leq\beta^{k(1-\theta)}\}} p_n  |x|^{\alpha} |\Upsilon|^{p_n-1}  dx \|PQ_1\|_{\infty} \|\phi_n\|_{L^\infty}\right)\\
&=M_{1}
+M_{2}+M_{3}+M_{4}.
\end{align*}
By \eqref{ch2}, we can get 
\begin{align*}
 &M_{1}=
\int_{\{y\in \R^2 : |y| < \delta^{-\theta}\}} e^{U_{\alpha}(y)} \frac{|y|^{\alpha}}{p_n} \left( V_{\alpha}(y) - U_{\alpha}(y) - \frac{U_{\alpha}(y)^2}{2} + O \left( \frac{\log^4(|y|^{2+\alpha}+2)}{p_n} \right) \right)\\&\quad\times\left( \left( \frac{2(2+\alpha)}{3} \log \delta \right) V_{\alpha,0}+ \frac{2}{3}\frac{|x|^{2+\alpha}-1}{|x|^{2+\alpha}+1}
       \log\left(|x|^{2+\alpha}+1\right)+O(1) \right) \phi_{n,1}(y) \, dy
 \\&
=\Big(\frac{4k\eta^2\log \eta}{3(2k\eta-2-\alpha)(\eta+1)} - \frac{1}{3} \Big)\int_{\{y: |y| < \delta^{-\theta}\}} |y|^{\alpha}e^{U_{\alpha}(y)}
\left( V_{\alpha}(y) - U_{\alpha}(y) - \frac{U^2_{\alpha}(y)}{2} \right)V_{\alpha,0}  \phi_{n,1}(y)\, dy
\\&\quad+O\left(\int_{\{y\in \R^2: |y| < \delta^{-\theta}\}} e^{U_{\alpha}(y)} \frac{|y|^{\alpha}}{p_n}  \log\left(|x|^{2+\alpha}+1\right)\left| V_{\alpha}(y) - U_{\alpha}(y) - \frac{U_{\alpha}(y)^2}{2}  \right|\, dy
\right)
\\&\quad
+ O \left( \int_{\{y\in \R^2: |y| < \delta^{-\theta}\}} e^{U_{\alpha}(y)} \frac{|y|^{\alpha}}{p_n}  \log\left(|x|^{2+\alpha}+1\right)\frac{\log^4(|y|^{2+\alpha}+2)}{p_n} \, dy\right)+
        O\left(\frac{1}{p_n}\right)
\\&=\Big(\frac{4k\eta^2\log \eta}{3(2k\eta-2-\alpha)(\eta+1)} - \frac{1}{3} \Big)\int_{\{y: |y| < \delta^{-\theta}\}} |y|^{\alpha}e^{U_{\alpha}(y)}
\left( V_{\alpha}(y) - U_{\alpha}(y) - \frac{U^2_{\alpha}(y)}{2} \right)V_{\alpha,0}  \phi_{n,1}(y)\, dy\\&\quad +O\left(\frac{1}{p_n}\right),
\end{align*} where we use the following estimate   \begin{align*}
&\frac{2(2+\alpha)}{3} \log \delta = \frac{2(2+\alpha)}{3} \log \left(
e^{-\frac{p_n}{2(2+\alpha)}
        \left(1-\frac{4k\eta\log\eta}
        {(2k\eta-2-\alpha)(1+\frac{1}
       {\eta})}
        \right)}\left(K_2(\eta)+
        O\left(\frac{1}{p_n}\right)\right)
\right) \\&=
p_n\Big(\frac{4k\eta^2\log \eta}{3(2k\eta-2-\alpha)(\eta+1)} - \frac{1}{3} \Big)+
\frac{2(2+\alpha)}{3} \log 
\left(K_2(\eta)+
        O\left(\frac{1}{p_n}\right)
\right).
\end{align*}
Passing to the limit we have
\begin{align*} 
&\int_{\{y: |y| < \delta^{-\theta}\}} |y|^{\alpha}e^{U_{\alpha}(y)}
\left( V_{\alpha}(y) - U_{\alpha}(y) - \frac{U_{\alpha}(y)}{2} \right) V_{\alpha,0}(y)  \phi_{n,1}(y)\, dy  \\&\quad\to a_0 \int_{\mathbb{R}^2}|y|^{\alpha}e^{U_{\alpha}(y)} \left( V_{\alpha} - U_{\alpha} - \frac{U_{\alpha}^2}{2} \right) V_{\alpha,0}^2 \, dy\\&\quad= 8(2+\alpha)\pi a_{0}\Big(1-\frac{1}{2}\log2(2+\alpha)^{2}\Big).
\end{align*}
where the exact value of the last integral can be deduced reasoning as in Lemma 6.1 in \cite{PT}, see also Lemma A.10 in \cite{PT}.
So we can get   \begin{align*}
 M_{1}&=\frac{4(2+\alpha)\pi a_0}{3}  \Big(\log2(2+\alpha)^{2}-2\Big)\Big(1- \frac{4k\eta^2\log \eta}{(2k\eta-2-\alpha)(\eta+1)} \Big)+O\left(\frac{1}{p_n}\right).
 \end{align*}
By \eqref{ch3}, we have
\begin{align*}
 M_{2}&=k
\int_{\{y\in \R^2: |y| \leq \frac{1}{\beta^{k\theta}}\}} e^{U_{0}(y)} \frac{\rho^{2k-2-\alpha}}
{(\beta^{k}y+\rho^{k})^{2-\frac{2}{k}-\frac{\alpha}{k}}} \left( 1 + O \left( \frac{\log^2(|y|+2)}{p} \right) \right)\\&\quad\times\left(  \frac{2(2+\alpha)}{3} \log \rho + O(1) \right) \phi_{n,2}(y) \, dy\\&=k\left(1+ o(1)\right)\left(  \frac{2(2+\alpha)}{3} \log \rho + O(1) \right) 
\int_{\{y\in \R^2: |y| < \beta^{-k\theta}\}} e^{U_{0}(y)}\phi_{n,2}(y) \, dy
\end{align*}
Assume (up to a subsequence) that $\phi_{n,2} \to c_0 Z_0$ fast enough to have 
\begin{align*}
\sup_{\{y: |y| \leq \frac{1}{\beta^{k\theta}}\}} |\phi_{n,2} - c_0 Z_0| \leq \frac{1}{p_n^2},\end{align*}
so that 
\begin{align*}
\int_{\{|y| \leq \frac{1}{\beta^{k\theta}}\}} e^{U_{0}} \phi_{n,2}\, dy
&
= c_0 \underbrace{\int_{\mathbb{R}^2} e^{U_{0}} Z_{0}\, dy }_{=0}+ \int_{\{|y| \leq \frac{1}{\beta^{k\theta}}\}} e^{U_{0}} (\phi_{n,2} - c_0 Z_{0}) \, dy+ \int_{\{|y| > \frac{1}{\beta^{k\theta}}\}} e^{U_{0}} Z_{0}\, dy\\&= O \left( \frac{1}{p_n^2} + \beta^{2k\theta} \right),
\end{align*}
then \begin{align*}
 M_{2}
= O \left( \frac{1}{p_n} +p_n \beta^{2k\theta} \right).
\end{align*}
By \eqref{ch4}, we can have
\begin{align*}
 M_{3}+M_{4}&=O \left( p_n \int_{\big\{x: \Omega\backslash\{|x| < \delta^{1-\theta}\}\big\}} |x|^{\alpha}e^{U_1(x)} dx  \right)\\&
\quad+ O \left( p_n\int_{\big\{x: \Omega\backslash\{|x| < \delta^{1-\theta}\}\big\}\backslash\{|x^{k}
-\rho^{k}|\leq\beta^{k(1-\theta)}\}} |x|^{\alpha}e^{U_{1}(x)}+|x|^{2k-2}
e^{U_{2}(x)}
  dx \right)\\&=O \left(p_n\delta^{(2+\alpha)\theta}\right)
  + O \left( p_n\beta^{2\theta}\right)=\left( p_n e^{-\varepsilon p} \right).
\end{align*}
Hence we have
\begin{align}\label{fgyp3}
&\int_{\Omega} |x|^{\alpha} (p_n|\Upsilon|^{p_n-1} - e^{U_1}) PQ_1 \phi_n \, dy\notag\\&
\quad= \frac{8(2+\alpha)\pi a_0}{3} \Big(\frac{\log2(2+\alpha)^{2}}{2}-1\Big)\Big(1- \frac{4k\eta^2\log \eta}{(2k\eta-2-\alpha)(\eta+1)} \Big)+ o(1).
\end{align}
Finally, with the aid of \eqref{R1}, we can get
\begin{align}\label{fgyp4}
\int_{\Omega} R_1 \phi_{n} &= O \left( \int_{\Omega} |x|^{\alpha} e^{U_1(x)} (|x| + \delta) |\phi_{n}(x)| \, dx \right)
\notag\\
&= O \left( \delta \int_{\mathbb{R}^2} e^{U_{\alpha}(y)}|y|^{\alpha} (|y| + 1) \, dy \right)= O(\delta).
\end{align}
Substituing \eqref{fgyp1}, \eqref{fgyp2}, \eqref{fgyp3} and \eqref{fgyp4} into \eqref{fgyp0}, we get
\begin{align*}
\frac{4(2+\alpha)\pi a_0}{3} \underbrace{\left(1+ \Big(\log2(2+\alpha)^{2}-2\Big)\Big(1- \frac{4k\eta^2\log \eta}{(2k\eta-2-\alpha)(\eta+1)} \Big)\right) }_{\neq 0}= o(1),
\end{align*}

which gives $ a_0 = 0 $.
In a similar manner, by using a function $PQ_2$ as in \cite{LIA}, one can show that also $c_0=0$. 
\end{proof}
We are now ready to give the proof of Proposition \ref{p.3}.

\begin{proof}[Proof of Proposition \ref{p.3}]Let us first prove that the a priori estimate \eqref{ch1} implies existence and uniqueness of a solution $\phi$ to \eqref{1.l1} that satisfies the orthogonality condition \eqref{0z} with $c$ is as in \eqref{1.l2}.

Let us split $ H_{0,k}^1(\Omega)$ into the direct sum of its subspaces
\begin{align*}
Z := \operatorname{span}\{ P\tilde{Z} \} \quad \text{and} \quad Z^\perp := \left\{ \phi \in H_{0,k}^1(\Omega) : \int_{\Omega} |x|^{2k-2} e^{U_2(x)} \tilde{Z}(x) \phi(x) \, dx = 0 \right\},
\end{align*}
with the respective projections denoted by $\Pi : H_{0,k}^1(\Omega) \to Z$ and $\Pi^\perp : H_{0,k}^1(\Omega) \to Z^\perp$.

Problem \eqref{1.l1} can be equivalently written as
\begin{align}\label{plh1}
(\operatorname{Id} - K) \phi = \tilde{h}, \quad \text{with} \quad K \phi := \Pi^\perp (-\Delta)^{-1} (p|\Upsilon|^{p-1} \phi) \quad \text{and} \quad \tilde{h} := -\Pi^\perp (-\Delta)^{-1} h.
\end{align}Here $(-\Delta)^{-1} : L^2_k(\Omega) \to H_{0,k}^1(\Omega)$ is the inverse of the Dirichlet Laplacian and is a compact operator, $\Pi^\perp$ is continuous and $p|\Upsilon|^{p-1} \in L^\infty(\Omega)$ (at fixed $p$), hence $K$ is also compact on $L^2_k(\Omega)$. Similar to Proposition 4.1 in \cite{MPM},  from \eqref{ch1} $K$ is injective. Therefore, from Fredholm's alternative, \eqref{plh1} has a unique solution $\phi \in Z^\perp$ for any $\tilde{h} \in Z^\perp$, this is true in particular if $\tilde{h} := -\Pi^\perp(-\Delta)^{-1}h$ for some $h \in L^\infty_k(\Omega)$ and by elliptic regularity $\phi \in H_{0,k}^1(\Omega) \cap W^{2,2}(\Omega) \subset H_{0,k}^1(\Omega) \cap L^\infty_k(\Omega)$. Moreover, $\phi \in Z^\perp$ if and only if 
$\int_\Omega \nabla P\widetilde{Z}\cdot \nabla \phi=0$, which, from \eqref{1.l1} and by using Lemma 4.7 in \cite{LIA},  gives
\[c =\frac{p \displaystyle\int_{\Omega}|x|^{\alpha} |\Upsilon|^{p-1} P\widetilde{Z}\phi -hP\widetilde{Z}}{\int_\Omega |x|^{2k-2}e^{U_2(x)}\widetilde{Z} P\widetilde{Z}}=\frac{p \displaystyle\int_{\Omega}|x|^{\alpha} |\Upsilon|^{p-1} P\widetilde{Z}\phi -hP\widetilde{Z}}
{\int_{\Omega}\big|\nabla P\widetilde{Z}\big|^{2}}.\]
So, in order to conclude the proof, we need to prove estimate \eqref{ch1}.
Lemma \ref{hD2} and \eqref{h2} give that
\begin{align}\label{lhp2}
\|\phi\|_{\infty} &\leq C p \left( \|h\|_* + \bigl\| c|x|^{2k-2}e^{U_2(x)}\tilde{Z}(x)\bigr\|_{*} \right)\notag
\\&
\leq C p \left( \|h\|_* + |c|\bigl\| \tilde{Z} \bigr\|_{\infty} \bigl\| |x|^{2k-2}e^{U_2(x)} \bigr\|_{*} \right)\notag\\&
\leq C p \bigl( \|h\|_* + |c| \bigr),
\end{align}
therefore it suffices to show that $c = O(\|h\|_*)$. This is equivalent to prove that
\begin{align*}
\lambda_p \leq C \|h\|_*, \quad \text{for any } |c| = \lambda_p,
\end{align*}
where $\lambda_p > 0$ is fixed. We argue by contradiction, assuming that there exist solutions to \eqref{1.l1} for
\begin{align}\label{lhp4}
\|h\|_* = o(\lambda_p).
\end{align}
For our convenience, let us choose $\lambda_p = \frac{1}{p}$, so that $\|\phi\|_{\infty}$ is bounded by \eqref{lhp2}. Multiplying equation \eqref{1.l1} by $P \tilde{Z}$ and integrating by parts,  using $c=\frac 1p$ and \eqref{0z},  we get
\begin{align}\label{lhp5}
\int_\Omega p|x|^{\alpha}|\Upsilon|^{p-1}\phi P \tilde{Z} \underbrace{- \int_\Omega |x|^{2k-2}e^{U_2(x)}\phi(x)
 \tilde{Z}(x)dx }_{= 0}= \int_\Omega h P \tilde{Z} + \frac{1}{p} \int_\Omega |\nabla P \tilde{Z}|^2
\end{align}
since \(\phi\)  satisfies the orthogonality condition \eqref{0z}:
\begin{align*}
\int_{\Omega}\nabla P\widetilde{Z}\cdot\nabla\phi
=\int_{\Omega}|x|^{2k-2}e^{U_{2}(x)}\widetilde{Z}(x)\phi(x)\,\mathrm{d}x=0. \end{align*}

From Lemma \ref{lh1} and \eqref{lhp2},  we deduce
\begin{align}\label{lhp6}
&\int_{\Omega} p|x|^{\alpha} |\Upsilon|^{p-1} \phi P \tilde{Z}\underbrace{- \int_\Omega |x|^{2k-2}e^{U_2(x)}\phi(x)
 \tilde{Z}(x)dx}_{= 0}
\notag\\&= \int_{\{|x^k - \rho^k| \leq \beta^{k(1-\theta)}\}} (p|x|^{\alpha} |\Upsilon(x)|^{p-1} - |x|^{2k-2} e^{U_2(x)}) \phi(x) P \tilde{Z}(x) \, dx\notag\\&
+ \int_{\{|x^k - \rho^k| \leq \beta^{k(1-\theta)}\}}  |x|^{2k-2} e^{U_2(x)} \phi(x) \bigl( P \tilde{Z}(x) - \tilde{Z}(x) \bigr) \, dx
\notag
\\&- \int_{\{|x^k - \rho^k| \geq \beta^{k(1-\theta)}\}}  |x|^{2k-2} e^{U_2(x)} \phi(x)\tilde{Z}(x)\, dx+ \int_{\{|x^k - \rho^k| \geq \beta^{k(1-\theta)}\}}p|x|^{\alpha} |\Upsilon|^{p-1} \phi P \tilde{Z}\notag
\\&= \frac{1}{p} \int_{\{|y| \leq \frac{1}{\beta^{k\theta}}\}} e^{U_{0}(y)} \left( V_{0}(y) - U_{0}(y) - \frac{U_{0}(y)^2}{2} + O \left( \frac{\log^4(|y| + 2)}{p^2} \right) \right) Z_1(y) \phi_2(y) \, dy+o\left(\frac{1}{p}\right),
\end{align}
where \( Z_1 \) is the function defined in \eqref{pz.06} and $\phi_2(y) := \phi(\beta^k y + \rho^k)$.

Arguing as in the proof of Lemma \ref{hD2}, we have that the rescalement $\phi_2(y)$ converges to $ C Z_0 $, for some $ C \in \mathbb{R} $, where $ Z_0 $ is the function defined in \eqref{pz.06}. Therefore, we can derive
\begin{align*}
& \int_{\{|y| \leq \frac{1}{\beta^{k(1-\theta)}}\}} e^{U_{0}(y)} \left( V_{0}(y) - U_{0}(y) - \frac{U_{0}(y)^2}{2} + O \left( \frac{\log^4(|y| + 2)}{p^2} \right) \right) Z_1(y) \phi_2(y) \, dy
\\ &\quad
\to C \int_{\mathbb{R}^2} e^U_{0} \left( V_{0} - U_{0} - \frac{U_{0}^2}{2} \right) Z_1 Z_0 = 0.
\end{align*} 
so that 
\begin{align*}
\int_{\Omega} p |\Upsilon|^{p-1} \phi P \tilde{Z} - \int_{\Omega} |x|^{2k-2} e^{U_2(x)} \phi(x)  \tilde{Z}(x) \, dx = o \left( \frac{1}{p} \right).
\end{align*}
In the right-hand side of \eqref{lhp5}, using \eqref{h3}, we can get
\begin{align*}
\left| \int_{\Omega} h P \tilde{Z} \right| \leq C \int_{\Omega} |h| \leq C \|h\|_* = o \left( \frac{1}{p} \right).
\end{align*}

Hence,  using Lemma 4.7 in \cite{LIA},  substituting into \eqref{lhp5}, we obtain the following contradiction
\begin{align*}
o \left( \frac{1}{p} \right) = \frac{8}{3p} k \pi + o \left( \frac{1}{p} \right),
\end{align*}

which completes the proof.
\end{proof}

\section{ Proof of Theorem \ref{athm1.1} }

In this section, we solve the nonlinear problem by virtue of the contraction mapping principle. The original equation is then reduced to a finite-dimensional energy minimization problem involving the parameter $\eta$. We thus finish the proof of Theorem \ref{athm1.1}.

First of all, we solve the problem
\begin{equation}\label{5.0}
\begin{cases}
\mathcal{L}\phi(x) + \mathcal{E}(x) + \mathcal{N}(\phi)(x) = c|x|^{2k-2}e^{U_2(x)}\tilde{Z}(x), & x \in \Omega,
\\
\phi(x) = 0, & x \in \partial\Omega, \\
\displaystyle\int_{\Omega}|x|^{2k-2}e^{U_2(x)}\tilde{Z}(x)\phi(x)\,dx = 0
\end{cases}
\end{equation}
for $c\in \R$.

\begin{proposition}\label{5.1}
Assume that the parameters \(\delta, \beta, \rho, \tau, \eta\) satisfy \eqref{1.11}. Then there exists \(C > 0\) such that, for any \(\eta\) satisfying \eqref{x.1}, equation \eqref{5.0} has, for $p$ large enough, a unique solution \((\phi, c) = (\phi_p(\eta), c_p(\eta)) \in H_{0,k}^1(\Omega) \cap L^\infty_K(\Omega) \times \mathbb{R}\), such that
\begin{equation}\label{5.2}
\|\phi_p(\eta)\|_{H_0^1} + \|\phi_p(\eta)\|_\infty \leq \frac{C}{p^3} \quad \text{and} \quad |c_p(\eta)| \leq \frac{C}{p^4},\end{equation}where $\mathcal{L}$, $\mathcal{E}$, $\mathcal{N}$ are defined in \eqref{1.8}, \eqref{1.9}, \eqref{1.10}.
Moreover, the map \(\eta \mapsto \phi_p(\eta)\) is of class \(C^1\).
\end{proposition}

Before proving the result, we need an essential estimate.

\begin{lemma}\label{5.3}
Assume $\phi$ is a solution to \eqref{5.0}.
There exists \(C > 0\) such that, uniformly in \(\Omega\),
\begin{align}\label{new3}
p|x|^{\alpha}|\Upsilon(x) + \phi(x)|^{p-2} \leq C \left(e^{U_1(x)} + |x|^{2k-2}e^{U_2(x)}\right) + (p\|\phi\|_\infty)^{p-2}.
\end{align}
\end{lemma}
\begin{proof}
Arguing as in the proof of Lemma \ref{lh1} where we proved  \eqref{ch4},  we can get 
\begin{align*}
p|x|^{\alpha}|\Upsilon(x)|^{p-2} \leq C \left(|x|^{\alpha}e^{U_1(x)} + |x|^{2k-2}e^{U_2(x)}\right).
\end{align*}
Then, thanks to the convexity of the map \(t \mapsto |t|^{p-2}\), as in \cite[Lemma 5.2 ]{LIA},  \eqref{new3} follows.
\end{proof}
\begin{proof}[Proof of Proposition \ref{5.1}.]
From Proposition \ref{p.3}, we can consider the linear map \( T : (L^\infty(\Omega), \|\cdot\|_*) \to L^\infty(\Omega) \), defined as \( h \mapsto Th := \phi \), where \( \phi \) is the solution to \eqref{1.l1} in the space
\( L^\infty(\Omega) \) endowed with the weighted norm \( \|\cdot\|_* \) defined in \eqref{h1}. In view of \eqref{ch1}  there exists \( C > 0 \), such that
\begin{equation}\label{5.4}
\|Th\|_\infty \leq C p \|h\|_*.\end{equation}
Using the operator \( T \), problem \eqref{5.0} can be transformed to the nonlinear equation

\begin{equation}\label{5.5}
\phi = \mathcal{A}(\phi) := -T(\mathcal{E} + \mathcal{N}(\phi)),
\end{equation}
where \( \mathcal{A} : (L^\infty_k(\Omega), \|\cdot\|_\infty) \to (L^\infty_k(\Omega), \|\cdot\|_\infty) \).
In the following, we will show that \( \mathcal{A}\) is a contraction on  the set $ S : =\big\{\phi\in L^\infty_k(\Omega) : \|\phi\|_\infty\leq\frac{M}{p^3} \big\}$,  for $p$ large enough, where \( M \) is a suitable positive constant. 
The Lagrange's theorem gives
\begin{align*}
|\mathcal{N}(\phi)| \leq p(p-1)|x|^{\alpha}|\Upsilon + O(|\phi|)|^{p-2}\phi^2,
\end{align*}
\begin{align*}
|\mathcal{N}(\phi_1) - \mathcal{N}(\phi_2)| \leq p(p-1)|x|^{\alpha}|\Upsilon + O(|\phi_1| + |\phi_2|)|^{p-2}(|\phi_1| + |\phi_2|)|\phi_1 - \phi_2|.
\end{align*}
By Lemma \ref{5.3} and \eqref{h2}, we can obtain
\begin{align}\label{5.6}
\|\mathcal{N}(\phi)\|_* &\leq C
 p \|\phi\|_\infty^2\big\|p|x|^{\alpha}|\Upsilon + O(|\phi|)|^{p-2}\big\|_*
\notag\\&\leq C p \|\phi\|_\infty^2 \big(1 + p^{p-2}\|\phi\|_\infty^{p-2}\|1\|_*\big)
\end{align}and
\begin{align}\label{5.7}
\|\mathcal{N}(\phi_1) - \mathcal{N}(\phi_2)\|_*&\leq C
p (\|\phi_1\|_\infty + \|\phi_2\|_\infty) \|\phi_1 - \phi_2\|_\infty
\big\|p|x|^{\alpha}|\Upsilon + O(|\phi_{1}|+|\phi_{2}|)|^{p-2}\big\|_*
\notag\\&\leq C p (\|\phi_1\|_\infty + \|\phi_2\|_\infty) \|\phi_1 - \phi_2\|_\infty \big(1 + p^{p-2}(\|\phi_1\|_\infty + \|\phi_2\|_\infty)^{p-2}\|1\|_*\big).
\end{align}

By Proposition \ref{lem2.6}, \eqref{5.4}, \eqref{5.5} and \eqref{5.6}, we can get
\begin{align*}
\|\mathcal{A}(\phi)\|_\infty \leq C p (\|\mathcal{E}\|_* + \|\mathcal{N}(\phi)\|_*) \leq \frac{C^2}{p^3} + \frac{C^2 M^2}{p^4} + \frac{C^2 M^p}{p^{2p}}\|1\|_*\leq \frac{M}{p^3},
\end{align*}
if \( M \) is strictly larger than \( C^2 \) and \( p \) is large enough.
Similarly, by \eqref{5.7}, we can get
\begin{align*}
\|\mathcal{A}(\phi_1) - \mathcal{A}(\phi_2)\|_\infty &\leq C p \|\mathcal{N}(\phi_1) - \mathcal{N}(\phi_2)\|_* \\
&\leq \frac{2C^2 M}{p} \|\phi_1 - \phi_2\|_\infty + \frac{C^2 (2M)^{p-1}}{p^{2p-3}} \|\phi_1 - \phi_2\|_\infty\|1\|_* \\
&\leq \frac{1}{2} \|\phi_1 - \phi_2\|_\infty,
\end{align*}
for \( p \) large enough. Therefore, \( \mathcal{A} \) is a contraction on \( S \). Then \( \mathcal{A} \) has a fixed point \( \phi \in L^\infty_k(\Omega) \), which is the unique solution to \eqref{5.5},  and also to  \eqref{5.0},  satisfying \( \|\phi\|_\infty \leq M/p^3 \).

Furthermore, by reasoning as in the proof of Proposition \ref{p.3}, we can find
\begin{align*}
|c| \leq C \left( \|h\|_* + \frac{\|\phi\|_\infty}{p} \right)\leq C \left( \|\mathcal{E}\|_* + \|\mathcal{N}(\phi)\|_* + \frac{\|\phi\|_\infty}{p} \right) \leq \frac{C}{p^4}.
\end{align*}
This, together with Lemma \ref{lh3}, implies that
\begin{align*}
\|\phi\|_{H_0^1} \leq C \left( \|\phi\|_\infty + \|\mathcal{E}\|_* + \|\mathcal{N}(\phi)\|_* \right) \leq \frac{C}{p^3}.
\end{align*}

The standard application of the implicit function theorem gives the regularity of $ \eta \mapsto \phi $ and concludes the proof.
\end{proof}
\subsection*{ The reduced energy $F_p$}

Let us introduce the functional
\begin{align*}
I_p(u) := \int_{\Omega} \left( \frac{1}{2}|\nabla u|^2 - \frac{1}{p+1}|x|^\alpha|u|^{p+1} \right), \quad u \in H_0^1(\Omega),
\end{align*}
whose critical points are solutions to \eqref{1.0}.  Let us define the reduced energy
\begin{equation}\label{F.2}
F_p(\eta) := I_p(\Upsilon + \phi_p(\eta))
\end{equation}
where $\Upsilon$ is as defined in \eqref{1.6} and $\phi_p(\eta)$ is the unique solution to \eqref{5.0}. Observe that $\phi_p(\eta)$ and $c_p(\eta)$ are well defined by Proposition \ref{5.1} when \( \eta \) satisfies \eqref{x.1}. \\
We would like to find \( \eta = \eta(p) \) satisfying \( c_p(\eta) = 0 \). This amounts to finding a critical point of the reduced energy.

\begin{lemma}\label{5.8}
The functional \( \eta \mapsto F_p(\eta) \) is of class \( C^1 \) and, for \( p \) large enough, if \( F'_p(\eta) = 0 \) then \( c_p(\eta) = 0 \).
\end{lemma}

\begin{proof}
The regularity of \( F_p \) follows from the regularity of the map \( \eta \mapsto \phi_p(\eta) \) in Proposition \ref{5.1}.

If \( F'_p(\eta) = 0 \), then by differentiating under the integral sign and the orthogonality condition \eqref{0z}, we can obtain

\begin{align*}
0 &= \int_{\Omega} (\Delta(\Upsilon + \phi) +|x|^{\alpha} |\Upsilon + \phi|^{p-1}(\Upsilon + \phi))(\partial_\eta \Upsilon + \partial_\eta \phi)
\\&
= -c(\eta) \int_{\Omega} |x|^{2k-2}e^{U_2(x)} \tilde{Z}(x)(\partial_\eta \Upsilon(x) + \partial_\eta \phi(x))dx
\\&
= -c(\eta) \int_{\Omega} |x|^{2k-2}e^{U_2(x)} \tilde{Z}(x)\partial_\eta \Upsilon(x)dx
\\&\quad
+ c(\eta) \int_{\Omega} |x|^{2k-2}\partial_\eta (e^{U_2(x)}\tilde{Z}(x))\phi(x)dx.
\end{align*}

From \eqref{1.6},  Lemma \ref{lem2.1} and using that $\left|\nabla U_\alpha(y)\cdot y\right|\leq C$ as well as for the functions $V_\alpha, W_\alpha, U_0,V_0,W_0$, we can get
\begin{align*}
\partial_\eta \Upsilon(x) &= \partial_\eta \tau \left( PU_1 + \frac{PV_1}{p} + \frac{PW_1}{p^2} \right) + \tau \left( P\nabla U_{\alpha} \left( \frac{x}{\delta} \right) + \frac{P\nabla V_{\alpha} \left( \frac{x}{\delta} \right)}{p} + \frac{P\nabla W_{\alpha} \left( \frac{x}{\delta} \right)}{p^2} \right) \cdot \left( -\frac{x}{\delta^2} \right) \partial_\eta \delta  \\
&\quad
- (\partial_\eta \tau \eta + \tau) \left( PU_2 + \frac{PV_2}{p} + \frac{PW_2}{p^2} \right)- \tau \eta \left( -\frac{x^k - \rho^k}{\beta^{2k}} \partial_\eta (\beta^k) - \frac{\partial_\eta (\rho^k)}{\beta^k} \right)
\\
&\quad  \times\left( P\nabla U_{0} \left( \frac{x^k - \rho^k}{\beta^k} \right) + \frac{P\nabla V_{0}\left( \frac{x^k - \rho^k}{\beta^k} \right)}{p} + \frac{P\nabla W_{0} \left( \frac{x^k - \rho^k}{\beta^k} \right)}{p^2} \right)  \\
&\quad +O \left(\left| \tau\frac 1 \delta \partial_\eta \delta+ \tau  \frac 1 \beta  \partial_\eta\beta \right|
\right)\\
&= O \left( \frac{1}{p} \left( \log \frac{1}{\delta} + \log \frac{1}{\beta^k} \right) \right) + O \left( \frac{1}{p}  \left( \frac{\partial_\eta \delta}{\delta} + \frac{\partial_\eta (\beta^k)}{\beta^k} \right) \right) \\
&\quad + \tau \eta P \left( \left( 1 + O \left( \frac{1}{p} \right) \right) \frac{\frac{x^k - \rho^k}{\beta^k}}{
1+\left| \frac{x^k - \rho^k}{\beta^k}\right|^{2}
}  \frac{\partial_\eta (\rho^k)}{\beta^k} \right) \\
&= \frac{\tau \eta}{4} P \tilde{Z}(x)k \frac{\partial_\eta \rho}{\rho} \frac{\rho^k}{\beta^k} \left( 1 + O \left( \frac{1}{p} \right) \right)+ O(1) \\
&= b(\eta)P \tilde{Z}(x) \frac{\rho^k}{\beta^k} \left( 1 + O \left( \frac{\log p}{p} \right) \right) + O(1),
\end{align*}
where
\[
b(\eta) := \frac{\sqrt{e}}{4}
\eta^{\frac{(2k-2-\alpha)\eta-(2+\alpha)}
        {(2k\eta-2-\alpha)(1+\eta)}}
k K_1(\eta) \frac{ (2k\eta-2-\alpha)(1+\eta)- (2k\eta^2 + 2+\alpha) \log \eta}{(2k\eta-2-\alpha)^2(1+\eta)^2}
\]
is uniformly bounded from above and below by positive constants and $K_1(\eta)$ is defined in Lemma \ref{lem2.1}.
From the definition of $U_{0}(y)$ and $Z_1(y)$, we have \( e^{U_{0}(y)}|Z_1(y)| + |\nabla (e^{U_{0}(y)} Z_1(y)) \cdot y| = O\left(\frac{1}{|y|^5+1}\right) \), and then we can get
\begin{align*}
\partial_\eta \left( e^{U_2(x)} \tilde{Z}(x) \right)
&=
 \partial_\eta \left( \frac{k^{2}}{\beta^{2k}} e^{U_{0} \left( \frac{x^k - \rho^k}{\beta^k} \right)} Z_1 \left( \frac{x^k - \rho^k}{\beta^k} \right) \right)
\\
&= -2k^{3}\frac{\partial_\eta (\beta)}{\beta^{2k+1}} e^{U_{0} \left( \frac{x^k - \rho^k}{\beta^k} \right)} Z_1 \left( \frac{x^k - \rho^k}{\beta^k} \right)
\\\quad
&+ \frac{k^{3}}{\beta^{2k}} e^{U_{0} \left( \frac{x^k - \rho^k}{\beta^k} \right)} \left( \nabla U_{0}Z_1 + \nabla Z_1\right)\left( \frac{x^k - \rho^k}{\beta^k} \right) \cdot \left( -\frac{x^k - \rho^k}{\beta^{k+1}} \partial_\eta (\beta) - \frac{\rho^{k-1}\partial_\eta (\rho)}{\beta^k} \right)
\\
&= O \left( p \frac{\rho^k}{\beta^{3k}} \frac{\beta^{5k}}{|x^k - \rho^k|^5 + \beta^{5k}} \right).
\end{align*}
Using Lemma 4.7 in \cite{LIA},  we can derive that for \(p\) large (independently on \(\eta\)),
\begin{align*}
0 &= c(\eta)b(\eta)\frac{\rho^k}{\beta^k} \int_{\Omega} |x|^{2k-2}e^{U_2(x)}\tilde{Z}(x)P\tilde{Z}(x)dx \left(1 + O\left(\frac{\log p}{p}\right)\right)
\\
&\quad
+ O\left(c(\eta)\int_{\Omega} |x|^{2k-2}e^{U_2(x)}\tilde{Z}(x)dx\right)
\\
&\quad
+ O\left(c(\eta)p\frac{\rho^k}{\beta^{3k}}\int_{\Omega} |x|^{2k-2}\frac{\beta^{5k}}{|x^k - \rho^k|^5 + \beta^{5k}}|\phi(x)|dx\right)
\\
&
= c(\eta)b(\eta)\frac{8k}{3}\pi\frac{\rho^k}{\beta^k}\left(1 + O\left(\frac{\log p}{p}\right)\right)
\\
&\quad
+ O\left(c(\eta)\beta^{3k}\int_{\Omega} |x|^{2k-2}\frac{|x^k - \rho^k|}{|x^k - \rho^k|^6 + \beta^{6k}}dx\right)
\\
&\quad
+ O\left(c(\eta)p\rho^k\beta^{2k}\|\phi\|_\infty\int_{\Omega} |x|^{2k-2}\frac{1}{|x^k - \rho^k|^5 + \beta^{5k}}dx\right)
\\
&
= c(\eta)b(\eta)\frac{8k}{3}\pi\frac{\rho^k}{\beta^k}\left(1 + O\left(\frac{\log p}{p}\right)\right) + O(c(\eta)) + O\left(c(\eta)\frac{1}{p^2}\frac{\rho^k}
{\beta^k}\right),
\end{align*}
which shows that \(c(\eta) = 0\). \end{proof}

The main order term of the reduced energy \( F_p \) is given by the following lemma:

\begin{lemma}\label{5.9}
For any \( \varepsilon > 0 \) small enough, it holds 
\begin{align}\label{e.2}
F_p(\eta) = \frac{2e\pi}{p} \left( \varphi_k(\eta) + O \left( \frac{\log p}{p} \right) \right) \quad \text{as } p \to +\infty \text{ uniformly in } \eta \in [\eta_{\infty,k} - \varepsilon, \eta_{\infty,k} + \varepsilon],
\end{align}
where
\begin{align*}
\varphi_k(\eta) := \left( 2 +\alpha+ 2k\eta^2 \right) \eta^{-\frac{4k\eta^2}{(2k\eta-2-\alpha)(\eta+1)}}.
\end{align*}
\end{lemma}
\begin{proof}
Multiplying \eqref{5.0} by $ \Upsilon + \phi $ and integrating by parts,
we can have
\begin{align*}
\int_{\Omega} |x|^{\alpha}|\Upsilon+ \phi|^{p} &= \int_{\Omega} |\nabla (\Upsilon + \phi)|^2 + c \int_{\Omega} |x|^{2k-2}e^{U_2(x)}\tilde{Z}(x)(\Upsilon(x) + \phi(x))dx
\\
&
= \int_{\Omega} |\nabla (\Upsilon + \phi)|^2 +  c \int_{\Omega} |x|^{2k-2}e^{U_2(x)}\tilde{Z}(x)\Upsilon(x)dx 
\\
&
= \int_{\Omega} |\nabla (\Upsilon + \phi)|^2 + O \left( \frac{1}{p^4} \right).
\end{align*}
From \eqref{5.2} and \eqref{F.2}, we have
\begin{align}\label{e.1}
F_p(\eta) &= \left( \frac{1}{2} - \frac{1}{p+1} \right) \int_{\Omega} |\nabla (\Upsilon + \phi)|^2 + O \left( \frac{1}{p^5} \right)\notag
\\
&
= \left( \frac{1}{2} + O \left( \frac{1}{p} \right) \right) \left( \int_{\Omega} |\nabla \Upsilon|^2  + 2 \int_{\Omega} \nabla \Upsilon\cdot \nabla \phi + \int_{\Omega} |\nabla \phi|^2 \right) + O \left( \frac{1}{p^5} \right)
\notag\\
&
= \left( \frac{1}{2} + O \left( \frac{1}{p} \right) \right) \int_{\Omega} |\nabla \Upsilon|^2 + O \left( \frac{1}{p^3} \|\Upsilon\|_{H_0^1} \right) + O \left( \frac{1}{p^5} \right). \end{align}

With the aid of equations \eqref{01.040} and \eqref{01.0040}, Lemmas \ref{lem2.2} and \ref{lem2.1} and the relations of paramaters \eqref{1.11}, we can get
\begin{align*}
&\int_{\Omega} |\nabla \Upsilon|^2
\\
&= \tau^2  \int_{\Omega} |x|^{\alpha}e^{U_1(x)} \left( 1 +  \frac{V_{\alpha}(\frac{x}{\delta})-f_1
(U_{\alpha}(\frac{x}{\delta}))}{p} + \frac{W_{\alpha}(\frac{x}{\delta})-
f_2(U_{\alpha}(\frac{x}{\delta}), V_{\alpha}(\frac{x}{\delta}))}{p^2} \right) PU_1(x)dx \\
& -\eta\tau^2  \int_{\Omega} |x|^{\alpha}e^{U_1(x)} \left( 1 +  \frac{V_{\alpha}(\frac{x}{\delta})-f_1
(U_{\alpha}(\frac{x}{\delta}))}{p} + \frac{W_{\alpha}(\frac{x}{\delta})-
f_2(U_{\alpha}(\frac{x}{\delta}), V_{\alpha}(\frac{x}{\delta}))}{p^2} \right)
PU_2(x)dx
 + \eta^2 \tau^2\int_{\Omega} |x|^{2k-2} e^{U_2(x)} \\
&\quad\times\left( 1 +\frac{V_{0} \left( \frac{x^k - \rho^k}{\beta^k} \right)-
f_1 \left( U_{0} \left( \frac{x^k - \rho^k}{\beta^k} \right) \right)}{p}+ \frac{W_{0} \left( \frac{x^k - \rho^k}{\beta^k} \right)
f_2 \left( U_{0} \left( \frac{x^k - \rho^k}{\beta^k} \right), V_{0} \left( \frac{x^k - \rho^k}{\beta^k} \right) \right)}{p^2} \right) PU_2(x) dx \\
& -
\eta \tau^2\int_{\Omega} |x|^{2k-2} e^{U_2(x)} \left( 1 + \frac{V_{0} \left( \frac{x^k - \rho^k}{\beta^k} \right)-
f_1 \left( U_{0} \left( \frac{x^k - \rho^k}{\beta^k} \right) \right)}{p}+ \frac{W_{0} \left( \frac{x^k - \rho^k}{\beta^k} \right)-
f_2 \left( U_{0} \left( \frac{x^k - \rho^k}{\beta^k} \right), V_{0} \left( \frac{x^k - \rho^k}{\beta^k} \right) \right)}{p^2} \right) \\
&\quad\times PU_1(x) dx
 + O \left( \int_{\Omega}  \frac{|\nabla PU_1|^2+ |\nabla PU_2|^2+|\nabla PV_1|^2+ |\nabla PV_2|^2+|\nabla PW_1|^2+ |\nabla PW_2|^2}{p}  dx  \right) \\
& := A_{1}+A_{2}+A_{3}+A_{4}+O \left( \frac{1}{p^2} \right) .
\end{align*}

Next, we estimate them term by term.  Using the estimates in Lemma \ref{lem2.2} and the relations between the parameters in Lemma \ref{lem2.1}, we have
\begin{align*}
A_{1}&=
\left( \frac{e}{p^2} \eta^{-\frac{4k\eta^2}{(2k\eta-2-\alpha)(\eta+1)}}
 + O \left( \frac{\log p}{p^3} \right) \right) \cdot \left( \int_{\Omega} |x|^{\alpha} e^{U_1(x)} \left( 1 + O \left( \frac{\log^4 \left( \frac{|x|}{\delta} \right) + 2}{p} \right) \right) PU_1(x) dx \right) \\
&= 4(2+\alpha)\pi\left( \frac{e}{p^2} \eta^{-\frac{4k\eta^2}{(2k\eta-2-\alpha)(\eta+1)}}
 + O \left( \frac{\log p}{p^3} \right) \right)\big(-2(2+\alpha)\log\delta+O(1)\big),
\end{align*}
\begin{align*}
A_{2}&=  - \eta\left( \frac{e}{p^2} \eta^{-\frac{4k\eta^2}{(2k\eta-2-\alpha)(\eta+1)}}
 + O \left( \frac{\log p}{p^3} \right) \right) \cdot\Bigg( \int_{\{|x| \leq \frac{\rho}{2}\}} |x|^{\alpha}e^{U_1(x)} \left( 1 + O \left( \frac{\log^4 \left( |\frac{x}{\delta}|\right) + 2}{p} \right) \right) PU_2(x) dx  \\
&
\quad+ \int_{\{\Omega\backslash\{|x| \leq \frac{\rho}{2}\}\}} |x|^{\alpha}e^{U_1(x)} \left( 1 + O \left( \frac{\log^4 \left(  |\frac{x}{\delta}|  \right) + 2}{p} \right) \right) PU_2(x) dx\Bigg)\\
&=- \eta\left( \frac{e}{p^2} \eta^{-\frac{4k\eta^2}{(2k\eta-2-\alpha)(\eta+1)}}
 + O \left( \frac{\log p}{p^3} \right) \right) \big(-4k\log\rho+O(1)\big )\\
&
\quad\times\Bigg( \int_{\{x\in \R^2:|x| \leq \frac{\rho}{2\delta}\}} |x|^{\alpha}e^{U_{\alpha}(x)} \left( 1 + O \left( \frac{\log^4 \left(  |\frac{x}{\delta}| \right) + 2}{p} \right) \right) dx
+O \left(\int_{x\in \Omega:|x| > \frac{\rho}{2}}\frac{\delta^{2+\alpha}}
{|x|^{4+\alpha}}\frac{\log^4\frac{1}{\delta
}}{p}\log\frac{1}{\beta
}\right)\Bigg)
\\
&=- 4(2+\alpha)\pi \eta\left( \frac{e}{p^2}  \eta^{-\frac{4k\eta^2}{(2k\eta-2-\alpha)(\eta+1)}}
 + O \left( \frac{\log p}{p^3} \right) \right) \big(-4k\log\rho+O(1)\big) +O \left(\frac{\delta^{2+\alpha}}
{|\rho|^{2+\alpha}}\frac{\log^4\frac{1}{\delta
}}{p}\log\frac{1}{\beta
}\right),
\end{align*}
\begin{align*}A_{3}&=  \eta^2
\left( \frac{e}{p^2} \eta^{-\frac{4k\eta^2}{(2k\eta-2-\alpha)(\eta+1)}}
 + O \left( \frac{\log p}{p^3} \right)\right)
\int_{\Omega} |x|^{2k-2} \left( 1 + O \left( \frac{\log^4 \left( \frac{|x^k-\rho^k|}{\beta^k} \right) + 2}{p} \right) \right) e^{U_2(x)} PU_2(x) dx \\
&=\eta^2
\left( \frac{e}{p^2} \eta^{-\frac{4k\eta^2}{(2k\eta-2-\alpha)(\eta+1)}}
 + O \left( \frac{\log p}{p^3} \right)\right)
\int_{\mathbb{R}^{2}} ke^{U(y)} \left( 1 + \frac{\log^4(|y| + 2)}{p} \right) (-4k \log \beta + O(\log(2 + |y|))) dy \\
&=8\pi\eta^2
\left( \frac{e}{p^2} \eta^{-\frac{4k\eta^2}{(2k\eta-2-\alpha)(\eta+1)}}
 + O \left( \frac{\log p}{p^3} \right)\right)
\big(-4k^{2}\log\beta+O(1)\big)
\end{align*}and
\begin{align*}
A_{4}&= -\eta
\left( \frac{e}{p^2}  \eta^{-\frac{4k\eta^2}{(2k\eta-2-\alpha)
(\eta+1)}}
 + O \left( \frac{\log p}{p^3} \right)\right)\Bigg(\int_{x\in \Omega :|x^k-\rho^k|\leq
\frac{1}{2}\rho^k} |x|^{2k-2} \left( 1 + O \left( \frac{\log^4 \left( \frac{|x^k-\rho^k|}{\beta^k} \right) + 2}{p} \right) \right)\\
& \quad\times e^{U_2(x)} PU_1(x) dx\Bigg)
 +O\left(\int_{\Omega\backslash\big\{|x^k-
\rho^k|\leq
\frac{1}{2}\rho^k\big\}} |x|^{2k-2} \left( 1 + O \left( \frac{\log^4 \left( \frac{|x^k-\rho^k|}{\beta^k} \right) + 2}{p} \right) \right) e^{U_2(x)} U_1(x) dx\right)\\
&=-\eta
\left( \frac{e}{p^2} \eta^{-\frac{4k\eta^2}{(2k\eta-2-\alpha)
(\eta+1)}}
 + O \left( \frac{\log p}{p^3} \right)\right)\big(-2(2+\alpha)\log\rho+O(1)
\big)\\
& \quad\times\Bigg(
\int_{y\in \R^2: |y|\leq
\frac{1}{2}\frac{\rho^k}{\beta^k}} e^{U_0(y)}\left( 1 + O \left( \frac{\log^4|y|+ 2}{p} \right) \right)dy\Bigg)
\\
& \quad+O\left(\int_{y\in \R^2:|y|>
\frac{1}{2}\frac{\rho^k}{\beta^k}} e^{U_0(y)}\left( 1 + O \left( \frac{\log^4\frac{1}{\beta}}{p}
\log\frac{1}{\delta} \right) \right)dy\right)\\
&=-8k\pi\eta
\left( \frac{e}{p^2} \eta^{-\frac{4k\eta^2}{(2k\eta-2-\alpha)
(\eta+1)}}
 + O \left( \frac{\log p}{p^3} \right)\right)\big(-2(2+\alpha)\log\rho+O(1)\big)
+O\left(\frac{\log\frac{1}{\delta}}{p^{2}}
\frac{\beta^{2k}}{\rho^{2k}}\right).
\end{align*}
From Lemma \ref{lem2.1}, we can conclude that
\begin{align*}
\int_{\Omega} |\nabla \Upsilon|^2&= 4\pi\left( \frac{e}{p^2} \eta^{-\frac{4k\eta^2}{(2k\eta-2-\alpha)(\eta+1)}}
 + O \left( \frac{\log p}{p^3} \right) \right)\\
&\quad\times
\Bigg((2+\alpha)\big(-2(2+\alpha)\log\delta
+4k\eta\log\rho\big)+2\big(-4k^2
\eta^2\log\beta
+2(2+\alpha)k\eta\log\rho
\big)\Bigg)
\\
&=
4\pi\left( \frac{e}{p^2} \eta^{-\frac{4k\eta^2}{(2k\eta-2-\alpha)(\eta+1)}}
 + O \left( \frac{\log p}{p^3} \right) \right)\big(p(2+\alpha)+2pk\eta^{2}
+O(1)
\big)
\\
&=\frac{4e\pi}{p} \left( \varphi_k(\eta) + O \left( \frac{\log p}{p} \right) \right),
\end{align*}
which, together with \eqref{e.1}, gives \eqref{e.2}.
\end{proof}

We are now in position to solve the reduced problem and to prove the main result of this section.

\begin{proposition}\label{e.3}
For any $\varepsilon > 0$, there is $p_\varepsilon > 0$ such that for any $p > p_\varepsilon$ the function $F_p$ has a critical point $\eta = \eta(p)$ such that:
\begin{align}\label{e.4}
\eta(p) \in (\eta_{\infty,k} - \varepsilon, \eta_{\infty,k} + \varepsilon), \quad \eta_{\infty,k} := \frac{2(2+\alpha)}{2k-2-\alpha}.
\end{align}
In particular, $\eta(p) \to \eta_{\infty,k}$ as $p \to +\infty$.
\end{proposition}

\begin{proof}
 By the direct computation, one can show that the function $\varphi_k(\eta)$ defined in Lemma \ref{5.9},  has, in the interval $\frac {2+\alpha}{2k}<\eta<1$ only one critical point, which is given by $\eta_{\infty,k} = \frac{2(2+\alpha)}{2k-2-\alpha}$, which is a strict minimum point for $\varphi_k(\eta)$. This,  together with Lemma \ref{5.9} shows that, 
the reduced energy function $F_p(\eta)$ has a minimum point $\eta(p) \in (\eta_{\infty,k} - \varepsilon, \eta_{\infty,k} + \varepsilon)$, for $p$  large enough.
\end{proof}
We are now in position to give the proof of our main theorem.

\begin{proof}[Proof of Theorem \ref{athm1.1}]
Proposition \ref{5.1} shows that there is a unique solution $\phi=\phi_p(\eta)$ to equation \eqref{5.0} for some $c=c_p(\eta)\in \R$. Proposition \ref{e.3} and Lemma \ref{5.8} tell us that we can find $\eta(p) \in (\eta_{\infty,k} - \varepsilon, \eta_{\infty,k} + \varepsilon)$ such that $c=c_p(\eta)=0$ and $\eta(p)\to \eta_{\infty, k}$ as $p\to \infty$. This means that we can find a solution $\phi=\phi_p(\eta)$ to \eqref{1.7}, which gives a solution $u=\Upsilon+\phi$ (as in \eqref{1.5}) to \eqref{1.0} with $\eta(p)$ that satisfies the estimate in the statement of the theorem.\\
The rates of the parameter $\tau_p,\eta_p,\delta_p,\beta_p,\rho_p$ are given by Lemma \ref{lem2.1}. Putting $\eta=\eta_{\infty,k}$ in Lemma \ref{lem2.1} proves the rates in the theorem and \eqref{abrn}.\\
It lasts to prove the estimates in \eqref{ffff+}.

From the convexity of $t \mapsto |t|^{p+1} $, we can get
\begin{align*}
|a + b|^{p+1} \leq \frac{1}{s^p} |a|^{p+1} + \frac{1}{(1-s)^p} |b|^{p+1}, \quad \forall a, b \in \mathbb{R}, \ s \in (0, 1)
\end{align*}
By \eqref{1.6}, \eqref{p.1}, \eqref{p.2} and Lemma \ref{lem2.1}, we can deduce $\Upsilon$ is uniformly bounded in $\Omega$.
Hence, by \eqref{h2}, \eqref{ch4} and \eqref{5.2}, we have
\begin{align*}
p|x|^{\alpha}|u(x)|^{p+1} &\leq p \frac{|x|^{\alpha}}{s^p} |\Upsilon(x)|^{p+1} + p \frac{|x|^{\alpha}}{(1-s)^p} |\phi(x)|^{p+1}
\\
&\leq C
\frac{1}{s^p}
\bigl(|x|^{\alpha}
e^{U_{1}(x)}+|x|^{2k-2}e^{U_{2}(x)}\bigr)
+ \frac{p}{(1-s)^p} \left( \frac{C}{p^3} \right)^{p+1}
\\
&\leq C
\frac{1}{s^p}
\Bigl(\frac{\delta^{2+\alpha}
}{|x|^{4+\alpha}}
+\frac{|x|^{2k-2}\beta^{2 k}}{|x^k - \rho^k|^{4}}\Bigr)+ \frac{p}{(1-s)^p} \left( \frac{C}{p^3} \right)^{p+1}
\\
&\leq C \frac{1}{s^p} e^{- \frac{p}{2}} \left( \frac{1}{|x|^{2+\delta}} + \frac{|x|^{2k-2}}{|x^k - \rho^k|^{2+\delta}} \right) + \frac{p}{(1-s)^p} \left( \frac{C}{p^3} \right)^{p+1}.
\end{align*}
Take $s = e^{-\frac{1}{4}}$, we can get $ p|u|^{p+1} \to 0 $ in $ L^\infty_{\text{loc}}(\Omega \setminus \{0\})$. Next we observe that 
\begin{align}\label{yp12}
\int_{\Omega} p|x|^{\alpha}u(x)_+^{p+1} dx
&=\int_{\{|x|>\delta^{1-\theta}\}}
p|x|^{\alpha}u(x)_+^{p+1} dx
+
\int_{\{|x|\leq\delta^{1-\theta}\}}
p|x|^{\alpha}u(x)_+^{p+1} dx\notag\\
:&=J_{1}+J_{2}
.\end{align}
From the proof of Lemma \ref{lem2.5}, (see \eqref{al0}), we have that
\begin{align}\label{po2.5}
|x| > \delta\quad \Longrightarrow \quad
p|x|^{\alpha}\Upsilon_{+}^{p} \leq C\frac{\delta^{2(2+\alpha)\lambda}
|x|^{\alpha}}{|x|^{2+\alpha+2(2+\alpha)\lambda}}
.
\end{align}
So, if $|x| > \delta^{1-\theta}$, by the uniform boundedness of $\Upsilon$, \eqref{5.2} and \eqref{po2.5}, we can get
\begin{align*}
p|x|^{\alpha}(u_{+}(x))^{p+1} &\leq p|x|^{\alpha}\big(\Upsilon_{+}(x)+|\phi(x)|\big)^{p+1}
\\
&\leq
p |x|^{\alpha}\frac{1}{s^p} (\Upsilon_{+}(x))^{p+1} + p \frac{1}{(1-s)^p} |\phi(x)|^{p+1}
\\
&\leq C \frac{1}{s^p} \frac{\delta^{2(2+\alpha)\lambda}
}{|x|^{2+2(2+\alpha)\lambda}}+ \frac{p}{(1-s)^p} \left( \frac{C}{p^3} \right)^{p+1}\\
&\leq C  \frac{\delta^{(2-\theta)(2+\alpha)\lambda}
}{|x|^{2+2(2+\alpha)\lambda}}+ \frac{p}{(1-s)^p} \left( \frac{C}{p^3} \right)^{p+1},
\end{align*}
where we take $s = \delta^{\frac{(2+\alpha)\lambda\theta}{p}} $.
Hence by Lemma \ref{lem2.1}, we can obtain
\begin{align}\label{yp13}
J_{1} \leq C\delta^{(2+\alpha)\lambda\theta} + \left( \frac{C}{p^3} \right)^p \to 0.
\end{align}
On the other hand, by \eqref{1.11}, Lemma \ref{lem2.1}, \eqref{2.18}, \eqref{e.4} and \eqref{po2.5}, we can obtain
\begin{align}\label{yp14}
J_{2}
&= \int_{\{|y|\leq\frac{1}{\delta^{\theta}}\}} \delta^{2+\alpha} p|y|^{\alpha}u(\delta y)_+^{p+1} dy\notag
\\
&= \int_{\{|y|\leq\frac{1}{\delta^{\theta}}\}} \delta^{2+\alpha} p|y|^{\alpha}\big(\Upsilon_{+}(\delta y)+
|\phi(\delta y)|
\big)^{p+1} dy\notag
\\
&=\int_{\{|y|\leq\frac{1}{\delta^{\theta}}\}} \delta^{2+\alpha} p|y|^{\alpha}\Big(\Upsilon^{p+1}_{+}(\delta y)
+O\big(\Upsilon^{p}_{+}(\delta y)|\phi(\delta y)|
+
|\phi(\delta y)|^{p+1}\big)
\Big) dy\notag
\\
&= \int_{\{|y|\leq\frac{1}{\delta^{\theta}}\}} \delta^{2+\alpha} p|y|^{\alpha}
\tau^{p+1} p^{p+1}\left(1+\frac{U_{\alpha}(y)}{p}
+O\left(\frac{\log(|y|^{2+\alpha}+2)}{p^2} \right)\right)^{p+1}+o(1)\notag\\
&\quad+O\left(
 \int_{\{|y|\leq\frac{1}{\delta^{\theta}}\}} \delta^{2+\alpha} p|y|^{\alpha}
\tau^{p} p^{p}\left(1+\frac{U_{\alpha}(y)}{p}
+O\left(\frac{\log(|y|^{2+\alpha}+2)}{p^2} \right)\right)^{p}|\phi(\delta y)|
 \right)\notag\\
&= \tau^2 p^2 \int_{\{|y|\leq\frac{1}{\delta^{\theta}}\}} |y|^{\alpha}
\left( 1 + \frac{U_{\alpha}(y)}{p} + O\left( \frac{\log(|y|^{2+\alpha}+2)}{p^2} \right) \right)^{p+1} dy+o(1)\notag\\
&\quad+O\left(\left(\frac{1}{p^3} \right)^{p+1}
\int_{\{|y|\leq\frac{1}{\delta^{\theta}}\}}
\left( 1 + \frac{U_{\alpha}(y)}{p} + O\left( \frac{\log(|y|^{2+\alpha}+2)}{p^2} \right) \right)^{p} dy
\right)\notag
\\
&= \left( e\eta^{-\frac{4k\eta^2}{(2k\eta-2-\alpha)
(\eta+1)}} + O\left( \frac{\log(p)}{p} \right) \right) \int_{\{|y|\leq\frac{1}{\delta^{\theta}}\}} |y|^{\alpha} e^{\left( 1+\frac{1}{p} \right)U_{\alpha}(y)}\notag \\
&\quad\times\left( 1 + O\left( \frac{\log^2(|y|^{2+\alpha}+2)}{p}
+\frac{\log^4(|y|^{2+\alpha}+2)}{p^{2}}\right) \right) dy+o(1)\notag
\\
&= 4(2+\alpha)\pi e \left( \frac{2k-2-\alpha}{2(2+\alpha)} \right)^{\frac{16k(2+\alpha)}{
(2k+2+\alpha)^2}}+o(1).
\end{align}
From \eqref{yp12}, \eqref{yp13} and \eqref{yp14}, we can conclude that
\begin{align*}
&\int_{\Omega} p|x|^{\alpha}u(x)_+^{p+1} dx
\rightarrow 4(2+\alpha)\pi e \left( \frac{2k-2-\alpha}{2(2+\alpha)} \right)^{\frac{16k(2+\alpha)}{
(2k+2+\alpha)^2}}.\end{align*}
By \eqref{1.11}, Lemma \ref{lem2.1}, \eqref{2.108}, \eqref{e.4} and \eqref{po2.5}, we have
\begin{align*}
p&\int_{\Omega} |x|^{\alpha}u_{-}(x)^{p+1} dx\\
&\leq \int_{\Omega} p|x|^{\alpha}\big(\Upsilon_{-}(x)+|
\phi(x)|
\big)^{p+1} dx\notag
\\
&=\int_{\Omega} p|x|^{\alpha}\Big(\Upsilon^{p+1}_{-}(x)
+O\big(\Upsilon^{p}_{-}(x)|\phi(x)|
+
|\phi(x)|^{p+1}\big)
\Big) dx
\\
&= \int_{\{|x^k-\rho^k|\leq\beta^{k(1-\theta
)}\}}p|x|^{\alpha}
\Big(\Upsilon^{p+1}_{-}(x)
+O\big(\Upsilon^{p}_{-}(x)|\phi(x)|
\big)
\Big) dx\\
&\quad
+\int_{\{|x^k-\rho^k|>\beta^{k(1-\theta)}\}}
p|x|^{\alpha}\Big(\Upsilon^{p+1}_{-}(x)
+O\big(\Upsilon^{p}_{-}(x)|\phi(x)|
\big)
\Big) dx
+o(1)\\
&= k (p\eta\tau)^{p+1}\int_{\{|y|\leq\beta^{-k\theta}\}}
\frac{\beta^{2k}}{k^{2}}
\frac{p}{(\beta^{k}y+\rho^{k})^{2-\frac{2}{k}-
\frac{\alpha}{k}}}\left( 1 + \frac{U_{0}(y)}{p} + O\left( \frac{\log(|y|^{2}+2)}{p^2} \right) \right)^{p+1}
dy\\
&\quad+O\left(\frac{(p\eta\tau)^{p}}{p^{3}}\int_{\{|y|\leq\beta^{-k\theta}\}}
\frac{\beta^{2k}}{k^{2}}
\frac{p}{(\beta^{k}y+\rho^{k})^{2-\frac{2}{k}-
\frac{\alpha}{k}}}\left( 1 + \frac{U_{0}(y)}{p} + O\left( \frac{\log(|y|^{2}+2)}{p^2} \right) \right)^{p}\right)
\\
&\quad+O\left(
\int_{\{|y|>\beta^{-k\theta}\}}
\frac{p(p\eta\tau)^{p+1}\beta^{2k}}
{\rho^{2k-2-
\alpha}}
\left( 1 + \frac{U_{0}(y)}{p} + O\left( \frac{\log(|y|^{2}+2)}{p^2} \right) \right)^{p+1}dy
\right)
\\
&\quad+O\left(
\int_{\{|y|>\beta^{-k\theta}\}}
\frac{(p\eta\tau)^{p}}{p^{2}}\frac{\beta^{2k}
}{\rho^{2k-2-
\alpha}}
\left( 1 + \frac{U_{0}(y)}{p} + O\left( \frac{\log(|y|^{2}+2)}{p^2} \right) \right)^{p}dy
\right)
\\
&= k (p\eta\tau)^{2}(1+o(1))
\int_{\{|y|\leq\beta^{-k\theta}\}}e^{(1+\frac{1}{p}
)U_{0}(y)}\left( 1 + O\left( \frac{\log^2(|y|+2)}{p}
+\frac{\log^4(|y|+2)}{p^{2}}\right) \right) dy+o(1)\\
&\quad+O\left(\frac{1}{p^{3}}\int_{\{|y|\leq\beta^{-k
\theta}\}}e^{U_{0}(y)}\left( 1 + O\left( \frac{\log^2(|y|+2)}{p}
+\frac{\log^4(|y|+2)}{p^{2}}\right) \right) dy\right)\\
&\quad+O\left(\int_{\{|y|>\beta^{-k\theta}\}}
e^{(1+\frac{1}{p}
)U_{0}(y)}\left( 1 + O\left( \frac{\log^2(|y|+2)}{p}
+\frac{\log^4(|y|+2)}{p^{2}}\right) \right) dy
\right)
\\
&\quad+O\left(\frac{1}{p^{3}}
\int_{\{|y|>\beta^{-k\theta}\}}e^{U_{0}(y)}\left( 1 + O\left( \frac{\log^2(|y|+2)}{p}
+\frac{\log^4(|y|+2)}{p^{2}}\right) \right) dy\right)\\
&= k (1+o(1)) e\eta^{2-\frac{4k\eta^2}{(2k\eta-2-\alpha)
(\eta+1)}}\int_{\{|y|\leq\beta^{-k\theta}\}}e^{(1+\frac{1}{p}
)U_{0}(y)}
 dy+o(1)
\\
&=8k\pi e \left( \frac{2(2+\alpha)}{2k-2-\alpha} \right)^{\frac{2(2k-2-\alpha)^{2}}{
(2k+2+\alpha)^2}}+o(1).
\end{align*}
So the proof is completed.
\end{proof}

\smallskip

\noindent{\bf Acknowledgements}:\,
The first author has been partially supported by  Next Generation EU-CUP-J55F21004240001,  DM 737-2021, risorse $2022-2023$ and  by the INdAM - GNAMPA Project,  CUP E53C25002010001.  The second author is supported by the Program of China Scholarship Council (Grant: 202506770048).

\noindent{\bf Data availability statement}\\
Our manuscript has no associated data.\\


\end{document}